\documentclass[10pt]{article}
\usepackage[margin=1in]{geometry}
\usepackage{amsmath,amssymb,amsthm,mathtools,mathrsfs,yhmath}
\usepackage{enumitem}
\usepackage{microtype}
\usepackage{hyperref}
\usepackage{bm}
\usepackage{stmaryrd}
\usepackage{cleveref}
\allowdisplaybreaks
\hypersetup{colorlinks=true, linkcolor=blue, citecolor=blue, urlcolor=blue}

\newtheorem{theorem}{Theorem}[section]
\newtheorem{lemma}{Lemma}[section]
\newtheorem{corollary}{Corollary}[section]
\newtheorem{remark}{Remark}[section]

\newtheorem{proposition}{Proposition}[section]

\newcommand{\sinc}{\mathrm{sinc}}

\newcommand{\supp}{\operatorname{supp}}
\newcommand{\Id}{\mathrm{Id}}

\newcommand{\ainf}{|a|_{\infty}}

\title{Long-time continuum limits for twisted Schr\"{o}dinger lattices}
\author{Brian Choi\thanks{University of Tennessee at Chattanooga, \texttt{brian-choi@utc.edu,choigh@bu.edu}}%
\vspace{-2ex}}
\date{\today \vspace{-1cm}}

\begin{document}
\maketitle

\begin{abstract}
We prove polynomial-in-time Sobolev bounds on the approximation error between the twisted discrete nonlinear Schr\"{o}dinger equation and its magnetic continuum limit. The analysis exploits complete integrability of the continuum flow through its Birkhoff representation to control approximation by a non-integrable Hamiltonian finite-difference lattice. We establish a focusing-defocusing dichotomy in uniform polynomial control: the bounds hold for arbitrary defocusing data and focusing data of sufficiently small mass, while modulational instability and aliasing yield an obstruction for large focusing data. The accompanying algebraic spatial rates are generically optimal in the Baire-category sense. The proof combines Birkhoff coordinates with modified energies for discrete Sobolev growth, while Fourier filtering recovers uniform spacetime estimates below the energy space despite degenerate lattice dispersion.

\end{abstract}


\section{Introduction}

Our central question is whether lattice models provide a long-time
continuum approximation to the magnetic nonlinear Schr\"{o}dinger
equation (NLS)
\begin{equation}\label{main_eq2}
i\partial_tu=-(\partial_x-ia(t))^2u+\mu|u|^2u,
\qquad u:\mathbb{T}_x\times\mathbb{R}_t\rightarrow\mathbb{C},
\qquad a\in L^\infty(\mathbb{R}_t;\mathbb{R}),
\end{equation}
with $u_0=u(\cdot,0)\in H^s(\mathbb{T})$ for $s>0$,
$\mu\in\{\pm1\}$, and $|a|_\infty:=\|a\|_{L^\infty}$.
We establish a long-time dichotomy: uniform polynomial-in-time error
bounds hold in the defocusing case ($\mu=1$) and extend to the focusing
case ($\mu=-1$) for sufficiently small mass, whereas modulational
instability precludes such uniform bounds for general focusing data.

Consider the finite difference semi-discretization
\begin{equation}\label{main_eq1}
i\partial_tu_h=-\Delta_{h,a(t)}u_h
+\mu P_\chi\left(|P_\chi u_h|^2P_\chi u_h\right),
\qquad u_h(0)=P_\chi\Pi_hu_0,
\end{equation}
on the periodic lattice $\mathbb{T}_h$ (see \Cref{background} for
notations), where
\begin{equation*}
\Delta_{h,a(t)}f(x)
=\frac{e^{-iha(t)}f(x+h)-2f(x)+e^{iha(t)}f(x-h)}{h^2},
\qquad x\in\mathbb{T}_h.
\end{equation*}
We define the regularity-dependent Fourier-Galerkin projector $P_\chi$ by
\begin{equation*}
\widehat{P_\chi f}(k)=\chi(k)\widehat f(k),
\qquad
\chi(k)=
\begin{cases}
\mathbf{1}_{E_h}(k),&0<s<1,\\
1,&s\geq1,
\end{cases}
\end{equation*}
where
$E_h=\{k\in\mathbb{T}_h^*:|hk|\leq\Lambda\pi\}$ for $0<\Lambda<\frac{1}{2}.$
For $0<s<1$, let $P_h:=P_\chi$, and hence \eqref{main_eq1} incorporates
Fourier filtering. A time-dependent gauge transformation \eqref{magnetic_conjugacy} conjugates \eqref{main_eq2} to the cubic NLS, whose local well-posedness is established in \cite{BourgainNLS1}. Our main results establish polynomial-in-time bounds for the nonlinear approximation error $e_h(t):=\mathcal{S}_h u_h(t)-u(t)$. More precise estimates, including explicit time exponents and
dependence on $|a|_\infty$, are given in \Cref{sec:polynomial-regimes}.

\begin{theorem}\label{thm:intro-polynomial}
Let $s>0$, $0 \leq \rho < s$, $R > 0$, and let
$\gamma_{s,\rho}=\min\left\{2,\frac{s-\rho}{2}\right\}$.
Let $u, u_h$ be the solutions of \eqref{main_eq2}, \eqref{main_eq1}, respectively, with $\mu=1$ and
$\|u_0\|_{H^s}\leq R$. Then, there exist $p=p(s,\rho)>0$, $c \in (0,1]$, and
$C=C(s,\rho,R,|a|_{\infty})>0$, with $C$ depending additionally on $\Lambda$ if $0 < s < 1$, such that
\begin{equation}\label{eq:intro-polynomial}
    \|e_h(t)\|_{H^\rho(\mathbb{T})}
    \leq
    C h^{\gamma_{s,\rho}}\langle t\rangle^p,
\end{equation}
for all $t \in \mathbb{R}$ and admissible $0 < h \leq c h_0$, where
\begin{equation}\label{eq:phase-margin}
h_0 :=
\begin{cases}
\min\left\{1,\frac{\pi(1/2-\Lambda)}
{2\langle |a|_\infty\rangle}\right\},
& 0<s<1,\\
1, &s \geq 1.
\end{cases}
\end{equation}
The same conclusion holds for $\mu=-1$, provided $\|u_0\|_{L^2}\leq\epsilon_*$, where $0<\epsilon_* \leq \sqrt{\pi}$ is independent of $s,R,h,a$.
\end{theorem}

The phase-weighted nearest-neighbor coupling in \eqref{main_eq1} arises in coupled-mode models for light propagation in twisted
multicore fibers with cores arranged in a ring
\cite{longhi2007light,castro2016light}.
Evanescent overlap couples neighboring guided modes, while twisting
introduces conjugate phase factors in the nearest-neighbor couplings.
The accumulated phases modify interference around the ring,
providing an optical analogue of the Aharonov-Bohm effect and
allowing suppression of optical tunneling
\cite{parto2017topological}.
Unlike simply-connected domains, where a spatially constant magnetic potential can be eliminated by a gauge transformation, the periodic structure of our model preserves $e^{2\pi i a(t)}$ as a nontrivial holonomy when $a(t)\notin\mathbb{Z}$.
Related models have been used to study localized standing waves
and their stability
\cite{ParkerAcevesTwistedFibers,ParkerShenAcevesZweck},
as well as twist-dependent modulation instability
\cite{maitland2019modulation}.
Twist-induced geometric phases and suppression of optical tunneling
have also been observed experimentally \cite{PartoOpticalAB}. Although \eqref{main_eq2} is conjugate to the cubic NLS by \eqref{magnetic_conjugacy}, this time-dependent translation has no direct analogue for \eqref{main_eq1}. Consequently, our analysis must address this nonautonomous lattice dynamics.

Our proof combines Birkhoff coordinates and modified energies to control the propagation and size of the consistency error in the nonlinear comparison. We exploit the complete integrability of cubic NLS through its Birkhoff representation as coordinatewise rotations \cite{grebert2014defocusing,kappeler2009birkhoff}, using their action-dependent frequencies and the regularity of the coordinate maps to derive linear bounds of the tangent flow. A related Birkhoff argument for Strang splitting appears in
\cite{feng2026temporal}. Since the lattice approximation does not preserve this integrable
structure, we use modified energies to control its Sobolev growth. In \cite{Chauleur}, modified energies based on higher time derivatives yield polynomial discrete Sobolev bounds for an exponential-in-time continuum-limit estimate, motivating our optimized error bounds. To accommodate fractional Sobolev regularities, we adapt the
Fourier-side modified-energy construction of \cite{Bernier} to obtain polynomial orbit bounds uniformly in $h$. This approach builds on the continuum Sobolev-growth literature
\cite{bourgain1996growth,staffilani1997growth,planchon2017growth}.

Compactness arguments in \cite{kirkpatrick2013continuum} establish weak convergence from the discrete to the fractional NLS, while dispersive estimates in \cite{hong2019strong} yield strong $L^2$ convergence. In the periodic setting, \cite{HongPeriodicNLS} uses short-time lattice Strichartz estimates to justify the finite difference scheme for the two-dimensional NLS with $H^1$ data. Adapting this framework, \cite{ChoiPeriodicFDNLS} treats the one-dimensional periodic fractional NLS below the energy space, including $\frac{1}{3}<s\leq1$ for the cubic NLS, deriving exponential-in-time, rather than polynomial, error bounds.

Below the energy space, the uniform $H^1$ control underlying our
modified-energy argument for $s\geq1$ is unavailable. For
$0<s<1$, we combine the restriction norm method \cite{BourgainNLS1}
with Fourier filtering \cite{IgnatZuazua} to recover uniform
dispersive estimates on $\mathbb{T}_h$. Restricting the lattice dynamics to $E_h$ excludes degenerate frequencies and recovers the uniform
$L^4$ embedding; see
\Cref{prop:nonauto-L4,prop:L4-failure-half-band}. We adapt the
restriction norm to the nonautonomous propagator, with related
constructions in \cite{CarvajalPantheeScialom}. The periodic
FPU-to-KdV limit \cite{kwak2025periodic} also uses $L^4$ estimates and
normal-form regularization. This low-regularity analysis is relevant
to simulations of random or partially coherent waves in integrable
turbulence, supercontinuum generation, and nonlinear photonic
lattices
\cite{WalczakRandouxSuret,DudleyGentyEggleton,BuljanEtAl}, where
inputs may be represented by finite Fourier sums with randomized
phases \cite{SuretTimeMicroscopy}. For the filtered scheme
\eqref{main_eq1}, our bounds provide quantitative error control in
terms of the spatial resolution and observation time.

To show the necessity of the mass restriction in \Cref{thm:intro-polynomial}, we show in
\Cref{thm:foc-failure} that the uniform polynomial bound
\eqref{eq:intro-polynomial} fails on sufficiently large bounded
data; see \Cref{future_research} for the focusing extension of \Cref{thm:intro-polynomial}. Modulational instability of NLS is classical \cite{akhmediev1986modulation}, while
\cite{weideman1987recurrence} studies instability and aliasing in spatial discretizations, and \cite{ablowitz1993numerical} examines
the amplification of roundoff errors near unstable focusing dynamics. We choose high-frequency perturbations whose continuum evolution remains confined to modulationally stable modes while lattice initialization aliases these perturbations into low-frequency unstable modes. Their amplification
produces an order-one lower bound for the approximation error at times $T_h \simeq |\log(h)|$.

In \Cref{thm:baire-category}, the generic optimality of the spatial
exponent $\gamma_{s,\rho}$ in \Cref{thm:intro-polynomial} is shown
in the Baire-category sense.
Sharpness of approximation rates is studied in
\cite{DickmeisNesselVanWickeren1984,vanWickeren1987}, with
applications to finite difference error bounds in
\cite{EsserGoebbelsNessel1995}.
In nonlinear dispersive PDEs, related Baire-category arguments
appear in studies of unbounded Sobolev trajectories for modified
NLS on $\mathbb{T}^2$ \cite{HaniUnboundedOrbits} and norm inflation
for scaling-supercritical NLS data
\cite{CampsGassotPathological}. We establish generic sharpness for this nonlinear continuum limit. The proof converts the lattice-continuum
phase defect into nonlinear error lower bounds through high-frequency
perturbations, while the initialization error limits the convergence
rate to second order.

\Cref{background} sets the
numerical-analytic framework; \Cref{filtering} develops the
Fourier-filtered uniform local theory; \Cref{subsec:birkhoff}
establishes Birkhoff-coordinate tangent estimates; and
\Cref{sec:modified-energies} derives polynomial discrete Sobolev
bounds via modified energies. \Cref{sec:polynomial-regimes} proves
the polynomial continuum limit through consistency estimates, while
\Cref{obstruction} establishes the focusing obstruction and
Baire-category optimality.

\section{Analytic framework}\label{background}

Let $\mathbb{T} = \mathbb{R}/(2\pi\mathbb{Z})$ identified with $[-\pi,\pi)$. For an admissible mesh size
$h = \frac{\pi}{M}$ with $M \in \mathbb{N}$, define
\begin{equation*}
\mathbb{T}_h = \{hj: -M \leq j \leq M-1\},
\qquad
\mathbb{T}_h^{*} = \{-M,\ldots,M-1\}.
\end{equation*}
In discrete convolutions, frequencies are identified modulo $2M$.

For $f \in L^2(\mathbb{T})$ and $g \in L^2_h := L^2(\mathbb{T}_h)$,
we use the Fourier convention
\begin{equation*}
\widehat{f}(k) = \frac{1}{2\pi}\int_{\mathbb{T}} f(x)e^{-ikx}dx,
\quad k \in \mathbb{Z};
\qquad
\widehat{g}(k) = \frac{h}{2\pi}\sum_{x \in \mathbb{T}_h}g(x)e^{-ikx},
\quad k \in \mathbb{T}_h^{*}.
\end{equation*}
Denote $H^s = H^s(\mathbb{T})$ and $H_h^s = H^s(\mathbb{T}_h)$. The Sobolev norms are
\begin{equation*}
\| f \|_{H^s}^2
= \sum_{k \in \mathbb{Z}}\langle k \rangle^{2s}|\widehat{f}(k)|^2,
\qquad
\| g \|_{H_h^s}^2
= \sum_{k \in \mathbb{T}_h^{*}}\langle k \rangle^{2s}|\widehat{g}(k)|^2, \qquad \langle k \rangle := (1+k^2)^{\frac{1}{2}}.
\end{equation*}
The cell-average discretization and Shannon interpolation are defined by
\begin{equation*}
\Pi_h f(x) = \frac{1}{h}\int_{x-\frac{h}{2}}^{x+\frac{h}{2}}f(y)dy,
\quad x \in \mathbb{T}_h;
\quad
\mathcal{S}_h g(x) = \sum_{k \in \mathbb{T}_h^{*}}\widehat{g}(k)e^{ikx},
\quad x \in \mathbb{T}.
\end{equation*}
The interpolation
satisfies $\mathcal{S}_h g(x) = g(x)$ for $x \in \mathbb{T}_h$ and
identifies $H_h^s$ isometrically with the $\mathbb{T}_h^{*}$-bandlimited
subspace of $H^s$. We use the same notation for the Fourier multipliers
$P_h$ and $\Delta_{h,a(t)}$ on either space.

Let $\omega_{h,a(t)}(k) = \frac{4}{h^2}
\sin^2\left(\frac{h(k-a(t))}{2}\right)$ and define the accumulated phases by
\begin{equation*}
\phi_k(t) = \int_0^t(k-a(\tau))^2d\tau,
\qquad
\phi_{h,k}(t) = \int_0^t\omega_{h,a(\tau)}(k)d\tau.
\end{equation*}
The corresponding linear propagators are
\begin{equation}\label{eq:linear-propagators}
\widehat{S_a(t,t_0)f}(k) = e^{-i(\phi_k(t)-\phi_k(t_0))}\widehat{f}(k), \qquad
\widehat{S_{h,a}(t,t_0)g}(k) = e^{-i(\phi_{h,k}(t)-\phi_{h,k}(t_0))}\widehat{g}(k).
\end{equation}

The magnetic NLS is conjugate to
the standard NLS by translation and phase. Define
\begin{equation}\label{magnetic_conjugacy}
    (\mathcal{G}_a(t)f)(x)
    =
    \exp\left(i\int_0^t a(\tau)^2 d\tau\right)
    f\left(x-2\int_0^t a(\tau)d\tau\right),
\end{equation}
and let $\Phi_a(t,\tau)$ be the magnetic NLS flow satisfying the cocyle property
\begin{equation}\label{cocycle}
\Phi_a(t,\tau) = \Phi_a(t,\sigma)\Phi_a(\sigma,\tau),\qquad t,\tau,\sigma \in \mathbb{R}.  
\end{equation}
With a slight abuse of notation, let $\{\Phi_t\}_{t \in \mathbb{R}}$ denote the NLS flow in $L^2$ for both nonlinearities. Then,
\begin{equation}\label{eq:magnetic-flow-conjugation}
\Phi_a(t,\tau)
    = \mathcal{G}_a(t)^{-1}\Phi_{t-\tau}\mathcal{G}_a(\tau),
\end{equation}
For a bounded interval $I \subseteq \mathbb{R}$, let
\begin{equation*}
\| f \|_{H^b(I)}
= \inf\left\{\| F \|_{H^b(\mathbb{R})}: F|_I = f\right\}.
\end{equation*}
For $u:  I \times \mathbb{T} \rightarrow \mathbb{C}$ and
$U: I \times \mathbb{T}_h \rightarrow \mathbb{C}$, define the restricted Bourgain norms by
\begin{equation*}
\| u \|_{X^{s,b}(I)}^2 = \sum_{k \in \mathbb{Z}}\langle k \rangle^{2s}\| e^{i\phi_k(t)}\widehat{u}(k,t) \|_{H_t^b(I)}^2, \qquad
\| U \|_{X_h^{s,b}(I)}^2 = \sum_{k \in \mathbb{T}_h^{*}}\langle k \rangle^{2s}\| e^{i\phi_{h,k}(t)}\widehat{U}(k,t) \|_{H_t^b(I)}^2.
\end{equation*}
We also use $X_h^{s,b}(I)$ for Shannon lifts. Finally, denote
\begin{equation*}
\mathcal{N}(f) = \mu|f|^2f,
\qquad
\mathcal{N}_h(f) = \mu P_h\left(|P_hf|^2P_hf\right).
\end{equation*}

\begin{remark}\label{rem:caratheodory}
For all $h>0$, the global well-posedness of \eqref{main_eq1} and
\eqref{main_eq1} follows from Carath\'{e}odory's theorem and the conservation of the $L^2_h$-norm. The corresponding solutions are absolutely continuous in time, i.e., $u_h \in AC(\mathbb{R};H^s_h)$, and satisfy
the ODEs a.e. in $t \in \mathbb{R}$.
\end{remark}

\begin{remark}\label{rem:band-properties}
For the solution $u_h$ of \eqref{main_eq1}, let $U_h = \mathcal{S}_h u_h$. If $s<1$, we have
\begin{equation}\label{main_eq12}
i\partial_tU_h = -\Delta_{h,a(t)}U_h+\mathcal{N}_h(U_h),
\qquad
U_h(0) = \mathcal{S}_h P_h\Pi_hu_0,
\end{equation}
since $\mathcal{S}_h P_h\left(|P_hu_h|^2P_hu_h\right)
= P_h\left(|P_hU_h|^2P_hU_h\right)$. Indeed, for $k_0,k_1,k_2,k_3 \in E_h$, the relation
$k_1-k_2+k_3 \equiv k_0 \pmod{2M}$ implies
$k_1-k_2+k_3 = k_0$ in $\mathbb{Z}$ since
$\Lambda < \frac{1}{2}$. Moreover, $\supp\left(\widehat{U}_h(\cdot,t)\right) \subseteq E_h,\ t \in \mathbb{R}$, which follows by applying $\Id-P_h$ to \eqref{main_eq12}.
\end{remark}

\begin{remark}\label{rem:time-reversal}\label{rem:initial-time}
The equations are invariant under the time-reversal symmetry $u_h(t)\mapsto \overline{u_h(-t)},\ a(t)\mapsto -a(-t)$. Thus, if $u_h$ solves \eqref{main_eq1}, then $\overline{u_h(-t)}$ solves the same equation with magnetic potential $-a(-t)$ and complex-conjugated initial data. Moreover, the translation $u_h(t) \mapsto u_h(t+t_0)$, $a(t) \mapsto a(t+t_0)$ preserves both evolution models for any $t_0 \in \mathbb{R}$. Thus, we take $t_0=0$ and treat $t\geq0$ in the arguments below.
\end{remark}

We record uniform bounds for discretization, reconstruction, and free evolution. 
 
\begin{lemma}\label{lem:Jh-bounded}
For any $s > -\frac{1}{2}$,
\begin{equation*}
    \| \mathcal{S}_h P_{\chi} \Pi_h f \|_{H^s} \lesssim_s \| f \|_{H^s}.
\end{equation*}
\end{lemma}

\begin{proof}
It can be shown by direct computation that
\begin{equation}\label{eq:Jh-fourier}
 \widehat{\mathcal{S}_h \Pi_h f}(k) = \sum_{\ell \in \mathbb{Z}} s_h(k + 2M \ell)\widehat f(k+2M\ell),
\end{equation}
where the cell-average multiplier is given by
\begin{equation*}
 s_h(\xi) = \sinc\left(\frac{h \xi}{2}\right) = 
 \begin{cases}
 \frac{\sin(h\xi/2)}{h\xi/2},& \xi \neq 0,\\
 1,&\xi = 0.
 \end{cases}
\end{equation*}
Hence it suffices to show uniform boundedness of $\mathcal{S}_h \Pi_h$. For $\ell=0$, $|s_h(k)| \leq 1$ yields the desired bound. For $|\ell| \geq 1$, a direct estimation yields
\begin{equation}\label{lem21_est}
    \frac{\pi |\ell|}{h} \leq |k + 2M\ell| \leq \frac{3\pi |\ell|}{h},\qquad |s_h (k + 2 M \ell)| \lesssim \frac{h |k|}{|\ell|}.
\end{equation}
Since $s + 1 > \frac{1}{2}$, we have
\begin{equation*}
    \frac{\langle k \rangle^s |s_h(k + 2M\ell)|}{\langle k + 2M\ell \rangle^{s}} \lesssim (h \langle k \rangle)^{s+1} \frac{|k|}{\langle k \rangle} |\ell|^{-(s+1)} \lesssim |\ell|^{-(s+1)},
\end{equation*}
and
\begin{equation*}
\begin{split}
\langle k \rangle^s \Bigl|\sum_{|\ell| \geq 1}s_h(k + 2M\ell)\widehat f(k+2M\ell)\Bigr| \lesssim \sum_{|\ell| \geq 1}|\ell|^{-(s+1)}\langle k+2M\ell \rangle^{s}|\widehat f(k+2M\ell)|.    
\end{split}    
\end{equation*}
Taking the Cauchy-Schwarz inequality in $\ell$ and the sum in $\ell^2_{k \in \mathbb{T}_h^{*}}$ yields the proof.
\end{proof}

The Fourier multipliers in \eqref{eq:linear-propagators} have modulus one. Hence
\begin{equation*}
 \|S_a(t,t_0)f\|_{H^s}=\|f\|_{H^s},\qquad
 \|S_{h,a}(t,t_0)g\|_{H_h^s}=\|g\|_{H_h^s}.
\end{equation*}
For $b>\frac{1}{2}$, the time-Sobolev trace estimate applied to the phase-conjugated Fourier coefficients gives
\begin{equation*}
 \|u\|_{C(I;H^s)}\lesssim_b\|u\|_{X^{s,b}(I)},\qquad
 \|U\|_{C(I;H_h^s)}\lesssim_b\|U\|_{X_h^{s,b}(I)}.
\end{equation*}
These statements and the following estimates do not require Fourier filtering.

Assume throughout
\begin{equation*}
\frac{3}{8} < b_0 < \frac{1}{2},
\qquad
\frac{1}{2} < b < 1-b_0,
\qquad
\theta = 1 - b - b_0 > 0.
\end{equation*}

\begin{lemma}\label{lem:linear-estimates}
Let $s \in \mathbb{R},\ |I| \leq 1,\ t_0 \in I$. Then, the linear estimates are
\begin{equation*}
\| S_a(t,t_0)f \|_{X^{s,b}(I)} \lesssim \| f \|_{H^s},\qquad \| S_{h,a}(t,t_0)g \|_{X_h^{s,b}(I)} \lesssim \| g\|_{H_h^s},
\end{equation*}
and the Duhamel estimates are
\begin{align*}
\left\|\int_{t_0}^t S_a(t,\tau)F(\tau)d\tau\right\|_{X^{s,b}(I)}
\lesssim |I|^\theta \| F\|_{X^{s,-b_0}(I)},\qquad \left\|\int_{t_0}^t S_{h,a}(t,\tau)G(\tau)d\tau\right\|_{X_h^{s,b}(I)}
\lesssim |I|^\theta \| G\|_{X_h^{s,-b_0}(I)}.
\end{align*}
The implicit constants may depend on $b,b_0$, and are independent of $h,s,a$.
\end{lemma}

\begin{proof}
By the definition of the propagator,
\begin{equation*}
e^{i\phi_k(t)}
\widehat{S_a(t,t_0)f}(k)
= e^{i\phi_k(t_0)}\widehat f(k).
\end{equation*}
After conjugating each spatial Fourier mode by its phase, the free evolution, and similarly the Duhamel operator, reduces to the standard restriction-norm estimates \cite{BourgainNLS1,TaoNDE}. The lattice case follows similarly with the sum restricted to $k \in \mathbb{T}_{h}^{*}$.
\end{proof}

\section{Fourier filtering and uniform local theory}\label{filtering}

The main result of this section is the following uniform spacetime estimate used in the regime $0<s<1$. Except for the negative assertion in \Cref{prop:L4-failure-half-band}, assume $0 < h \leq h_0$ and $0 < \Lambda < \frac{1}{2}$.

\begin{proposition}\label{prop:nonauto-L4}
There exists $C = C(\Lambda,b_0)>0$ such that for any spacetime function $U:I\times\mathbb{T}\rightarrow\mathbb{C}$,
\begin{equation}\label{eq:nonauto-L4}
\| P_h U \|_{L^4(I\times\mathbb{T})}
\leq C \| P_h U\|_{X_h^{0,b_0}(I)},
\end{equation}
and,
\begin{equation}\label{eq:L43-dual}
\|U\|_{X^{0,-b_0}(I)}+\|P_h U\|_{X_h^{0,-b_0}(I)}
\leq C
\|U\|_{L^{\frac{4}{3}}(I\times\mathbb{T})}.
\end{equation}
\end{proposition}

A reparametrization of time preserves Sobolev norms up to equivalence.
\begin{lemma}\label{lem:composition}
Let $0<b<1$. Let $\Theta:I\to J$ be increasing, absolutely continuous, and surjective, with $0<c\le \Theta'(t)\le C<\infty$ for a.e. $t \in I$. Then,
\begin{equation*}
\| f\circ\Theta^{-1}\|_{H^b(J)} \simeq_{b,c,C} \| f\|_{H^b(I)}.
\end{equation*}
\end{lemma}

\begin{proof}
Let $I=[\alpha,\beta]$. Extend $\Theta$ to an increasing bi-Lipschitz map $\Psi:\mathbb{R}\rightarrow\mathbb{R}$ by
\begin{equation*}
\Psi(t)=
\begin{cases}
\Theta(\alpha)+c(t-\alpha),&t<\alpha,\\
\Theta(t),&t\in I,\\
\Theta(\beta)+c(t-\beta),&t>\beta.
\end{cases}    
\end{equation*}
Then, $c|t-s|\leq|\Psi(t)-\Psi(s)|\leq C|t-s|$ for any $t,s \in \mathbb{R}$.

Let $F \in H^b(\mathbb{R})$ and $G=F \circ \Psi^{-1}$. Changing variables, we obtain $\|G\|_{L^2}^2
\leq C\|F\|_{L^2}^2$ and
\begin{equation*}
[G]^2
:= \iint_{\mathbb{R}^2}
\frac{|G(t')-G(s')|^2}
     {|t'-s'|^{1+2b}}
dt' ds' = \iint_{\mathbb{R}^2}
\frac{|F(t)-F(s)|^2\Psi'(t)\Psi'(s)}
     {|\Psi(t)-\Psi(s)|^{1+2b}}
dtds \leq C^2 c^{-(1+2b)}[F]^2.
\end{equation*}
By the Slobodeckij characterization of fractional Sobolev spaces,
\begin{equation*}
\| G \|_{H^b(\mathbb{R})}^2 \simeq_b \| G \|_{L^2(\mathbb{R})}^2 + [G]^2 \lesssim \| F \|_{H^b(\mathbb{R})}^2.
\end{equation*}
If $F|_I=f$, then $(F\circ\Psi^{-1})|_J=f\circ\Theta^{-1}$. Taking the infimum over all extensions $F$ proves the claim. The reverse bound follows similarly by considering
$\Theta^{-1}:J\rightarrow I$.
\end{proof}

The Bourgain counting argument requires a uniform bound on spectral sublevel sets.
\begin{lemma}\label{lem:mag-counting}
For $n\in\mathbb{Z}$, denote 
\begin{equation*}
E_{h,n} = \{k \in E_h: n-k \in E_h\},\qquad \lambda_{h,n,k}=\frac{4}{h^2}\cos\left(h\left(k-\frac{n}{2}\right)\right).  
\end{equation*}
For any $\mu \in \mathbb{R},\ L \geq 0$, there exists $C = C(\Lambda)>0$ such that
\begin{equation*}
\#\{k\in E_{h,n}:|\lambda_{h,n,k}-\mu|\leq L\}
\leq C \langle L \rangle^{\frac{1}{2}}.
\end{equation*}
\end{lemma}

\begin{proof}
Let $r=k - \frac{n}{2}$. For all $k \in E_{h,n}$, we have $r \in \mathbb{Z}-\frac{n}{2}$ and $|hr|= \frac{h}{2} |k - (n-k)| \leq \Lambda\pi$. Define $\psi_h(\xi) = \frac{4}{h^2}\cos(h\xi)$. Then, $\lambda_{h,n,k}=\psi_h(r)$, and $\psi_h$ is uniformly concave on $[-\frac{\Lambda \pi}{h},\frac{\Lambda \pi}{h}]$ since  
\begin{equation*}
-\psi_h^{\prime\prime}(\xi) = 4\cos(h\xi) \geq 4\cos(\Lambda\pi)>0.
\end{equation*}
Moreover,
\begin{equation*}
    S:=\{\xi \in \mathbb{R}:|\xi| \leq \frac{\Lambda\pi}{h},\ |\psi_h(\xi)-\mu|\leq L\}
\end{equation*}
has at most two connected components. Let $[\alpha,\beta]$ be one such component. Uniform concavity yields
\begin{equation*}
\psi_h \left(\frac{\alpha+\beta}{2}\right)
\geq \frac{\psi_h(\alpha)+\psi_h(\beta)}{2} + \frac{\cos\left(\Lambda \pi\right)}{2}(\beta-\alpha)^2.
\end{equation*}
Since $\psi_h([\alpha,\beta]) \subseteq [\mu-L,\mu+L]$, we have
\begin{equation*}
    |\beta-\alpha| \leq 2\sqrt{\frac{L}{\cos\left(\Lambda \pi\right)}}.
\end{equation*}
Since an interval of length $L_0$ contains at most $1 + L_0$ lattice points, we have
\begin{equation*}
\begin{split}
\#\{k\in E_{h,n}:|\lambda_{h,n,k}-\mu|\leq L\} &= \#\left\{r: r = k - \frac{n}{2} \text{ for some } k \in E_{h,n} \wedge |\psi_h(r)-\mu|\leq L\right\}\\
&\leq 2 \left(1+2\sqrt{\frac{L}{\cos\left(\Lambda \pi\right)}}\right).
\end{split}
\end{equation*}
\end{proof}

Bourgain’s embedding argument \cite{BourgainNLS1} for the quadratic phase extends uniformly to the discrete phase.
\begin{corollary}\label{lem:Bourgain-model}
For $n\in\mathbb{Z}$ and sequences $\{f_{n,k}\}_{k\in E_{h,n}}, \{g_{n,k}\}_{k\in E_{h,n}} \subseteq H^{b_0}(\mathbb{R})$, there exists $C = C(\Lambda,b_0) > 0$ such that
\begin{equation*}
\Bigl\|
\sum_{k\in E_{h,n}} e^{i\lambda_{h,n,k}t}f_{n,k}(t)g_{n,k}(t)
\Bigr\|_{L^2_t}^2
\lesssim_{\Lambda,b_0}
\sum_{k\in E_{h,n}}\| f_{n,k} \|_{H^{b_0}}^2
\| g_{n,k} \|_{H^{b_0}}^2,
\end{equation*}
\end{corollary}

\begin{proof}
Let 
\begin{equation*}
B_n(t) = \sum\limits_{k\in E_{h,n}} e^{i\lambda_{h,n,k}t}f_{n,k}(t)g_{n,k}(t).    
\end{equation*}
Taking the Fourier transform in $t$ and applying the Cauchy-Schwarz inequality, we obtain
\begin{equation}\label{Bn}
\begin{split}
|\widehat B_n(\tau)|^2
&\leq M_n(\tau)
\sum_{k\in E_{h,n}}\int
\langle \eta\rangle^{2b_0}|\widehat f_{n,k}(\eta)|^2
\langle\tau-\lambda_{h,n,k}-\eta\rangle^{2b_0}
|\widehat g_{n,k}(\tau-\lambda_{h,n,k}-\eta)|^2 d\eta,
\\
M_n(\tau)
&:=\sum_{k\in E_{h,n}}\int
\frac{d\eta}
{\langle\eta\rangle^{2b_0}\langle\tau-\lambda_{h,n,k}-\eta\rangle^{2b_0}}.    
\end{split}
\end{equation}
The convolution integral estimate yields
\begin{equation*}
M_n(\tau)\lesssim
\sum_{k\in E_{h,n}}
\langle\tau-\lambda_{h,n,k}\rangle^{-(4b_0 - 1)}.
\end{equation*}
By \Cref{lem:mag-counting},
\begin{equation*}
    \sum_{k\in E_{h,n}}
\langle\tau-\lambda_{h,n,k}\rangle^{-(4b_0 - 1)} \lesssim \sum_{j \geq 0}2^{-j(4b_0 - 1)}
\#\{k\in E_{h,n}:|\tau-\lambda_{h,n,k}| \leq 2^{j+1}\} \lesssim \sum_{j \geq 0} 2^{-j(4b_0 - \frac{3}{2})},
\end{equation*}
uniformly in $\tau,n$ where the geometric series converges since $b_0 > \frac{3}{8}$. Integrating in $\tau$ the upper bound of $|\widehat B_n(\tau)|^2$ in \eqref{Bn} yields the desired result.
\end{proof}

\begin{proof}[Proof of \Cref{prop:nonauto-L4}]
For each $k \in E_{h}$, there exists $f_k \in H^{b_0}(\mathbb{R})$ such that $f_k(t) = e^{i\phi_{h,k}(t)}\widehat{P_h U}(k,t)$ for all $t \in I$ and
\begin{equation}\label{beginning}
\left(\sum_{k\in E_h}\| f_k\|_{H^{b_0}(\mathbb{R})}^2\right)^{1/2}
\leq 2\| P_h U\|_{X_h^{0,b_0}(I)}.    
\end{equation}
On $I$,
\begin{equation*}
    B_n(t) := \widehat{(P_h U)^2}(n,t) = \sum_{k\in E_{h,n}}
e^{-i(\phi_{h,k}(t)+\phi_{h,n-k}(t))}f_k(t)f_{n-k}(t).
\end{equation*}
The half-angle formula in trigonometry yields
\begin{equation}\label{pair_phase}
\omega_{h,a(t)}(k)+\omega_{h,a(t)}(n-k)
=\frac{4}{h^2}\left\{1-
\cos\left(h\left(\frac{n}{2}-a(t)\right)\right)
\cos\left(h\left(k-\frac{n}{2}\right)\right)\right\}.
\end{equation}
Define
\begin{equation*}
\Theta_{h,n}(t)
=\int_0^t\cos\left(h\left(\frac{n}{2}-a(\tau)\right)\right) d\tau.
\end{equation*}
Since $k \in E_{h,n}$,
\begin{equation*}
\left|h\left(\frac{n}{2}-a(t)\right)\right|
\leq\Lambda\pi+h \ainf
\leq \frac{\pi}{4} + \frac{\Lambda\pi}{2},
\end{equation*}
and
\begin{equation}\label{deriv_bound} 
0 < \cos\left(\frac{\pi}{4} + \frac{\Lambda\pi}{2}\right)\leq \Theta_{h,n}^{\prime}(t) \leq 1,
\end{equation}
the reparametrization $\Theta_{h,n}:I \rightarrow J_n := \Theta_{h,n}(I)$ is uniformly bi-Lipschitz. Integrating \eqref{pair_phase} yields
\begin{equation*}
B_n(t) = e^{-\frac{4it}{h^2}}
\sum_{k\in E_{h,n}}
e^{i\lambda_{h,n,k}\Theta_{h,n}(t)}f_k(t)f_{n-k}(t).
\end{equation*}
Changing variables $\theta=\Theta_{h,n}(t)$ and considering the Jacobian estimate \eqref{deriv_bound}, we have
\begin{equation}\label{B_estimate}
\| B_n\|_{L^2(I)}^2
\lesssim_\Lambda
\Bigl\|
\sum_{k\in E_{h,n}}e^{i\lambda_{h,n,k}\theta}
f_k\circ\Theta_{h,n}^{-1}(\theta)
f_{n-k}\circ\Theta_{h,n}^{-1}(\theta)
\Bigr\|_{L^2(J_n)}^2.
\end{equation}
There exist some $\widetilde{f}_{n,k},\widetilde{g}_{n,k} \in H^{b_0}(\mathbb{R})$ such that $\widetilde{f}_{n,k}|_{J_n} = f_k\circ\Theta_{h,n}^{-1}$ and $\widetilde{g}_{n,k}|_{J_n}
=f_{n-k}\circ\Theta_{h,n}^{-1}$, and 
\begin{equation}\label{restriction_norm}
\| \widetilde f_{n,k}\|_{H^{b_0}}
\leq 2\| f_k\circ\Theta_{h,n}^{-1}\|_{H^{b_0}(J_n)},\qquad
\| \widetilde g_{n,k}\|_{H^{b_0}}
\leq 2\| f_{n-k}\circ\Theta_{h,n}^{-1}\|_{H^{b_0}(J_n)}.
\end{equation}
By extending $L^2(J_n)$ to $L^2(\mathbb{R})$, applying \Cref{lem:Bourgain-model}, and \eqref{restriction_norm},
\begin{equation*}
\begin{split}
\text{RHS of } \eqref{B_estimate} \leq \Bigl\|
\sum_{k\in E_{h,n}}e^{i\lambda_{h,n,k}\theta}
\widetilde{f}_{n,k}(\theta)\widetilde{g}_{n,k}(\theta)
\Bigr\|_{L^2(\mathbb{R})}^2 &\lesssim \sum_{k\in E_{h,n}}
\| \widetilde{f}_{n,k}\|_{H^{b_0}}^2
\| \widetilde{g}_{n,k}\|_{H^{b_0}}^2\\
&\lesssim \sum_{k\in E_{h,n}}
\| f_k\|_{H^{b_0}(I)}^2
\| f_{n-k}\|_{H^{b_0}(I)}^2. 
\end{split}
\end{equation*}
The Plancherel theorem yields
\begin{equation*}
\begin{split}
\| P_h U\|_{L^4(I \times \mathbb{T})}^4 &= \| (P_h U)^2\|_{L^2(I \times \mathbb{T})}^2 \simeq \sum_{n \in \mathbb{Z}} \| B_n \|_{L^2(I)}^2\\
&\lesssim
\sum_n\sum_{k\in E_{h,n}}
\| f_k\|_{H^{b_0}(I)}^2
\| f_{n-k}\|_{H^{b_0}(I)}^2=
\left(\sum_{k\in E_h}\| f_k\|_{H^{b_0}(I)}^2\right)^2 \lesssim \| P_h U\|_{X_h^{0,b_0}(I)}^4.  
\end{split}
\end{equation*}
where the last inequality is by \eqref{beginning}, and this shows \eqref{eq:nonauto-L4}. By duality and \cite{BourgainNLS1}, the estimate \eqref{eq:L43-dual} follows.
\end{proof}

Without Fourier filtering, the uniform dispersive control fails due to counterexamples motivated by \cite{IgnatZuazua}.
\begin{proposition}\label{prop:L4-failure-half-band}
Let $\Lambda \in [\frac{1}{2},1]$ and $a \equiv 0$. There exists no constant $C>0$, independent of $h$, such that
\begin{equation*}
\| P_h U\|_{L^4(I\times\mathbb{T})}\leq C\| P_h U\|_{X_h^{0,b_0}(I)},
\end{equation*}
holds for all sufficiently small $h$ and spacetime function $U$.
\end{proposition}

\begin{proof}
Let $k_h = \frac{\pi}{2h} = \frac{M}{2}$, where we assume $M \in 2 \mathbb{N}$, which implies $\omega_{h,0}^{\prime\prime}(k_h)=0$. Let $\xi_h = \lfloor\epsilon h^{-\frac{1}{3}}\rfloor$ where $0 < \epsilon \ll 1$ is to be determined. Let $J$ be a closed interval contained in the interior of $I$, and let $\eta \in C_c^{\infty}(\text{int} (I);\mathbb{R})$ such that $\eta|_{J} \equiv 1$. Define
\begin{equation}\label{coherent_packet}
U_h(x,t)
 = \xi_h^{-\frac{1}{2}}\eta(t)
\sum_{n=1}^{\xi_h}
e^{i(k_h-n)x}e^{-it\omega_{h,0}(k_h-n)}.
\end{equation}
Then,
\begin{equation*}
\| U_h\|_{X_h^{0,b_0}(I)}^2
\leq
\sum_{n=1}^{\xi_h}\xi_h^{-1}\| \eta\|_{H^{b_0}(\mathbb{R})}^2
=\| \eta\|_{H^{b_0}}^2.
\end{equation*}
For $1 \leq n \leq \xi_h$,
\begin{equation*}
\begin{split}
\omega_{h,0}(k_h-n)
&=\frac{4}{h^2} \sin^2\left(\frac{\pi}{4}-\frac{hn}{2}\right)
=\frac{2}{h^2}-\frac{2n}{h}+R_{h},\\
R_{h,n}&:=\frac{2}{h^2}\left(hn-\sin(hn)\right).    
\end{split}
\end{equation*}
The Taylor expansion of $R_{h,n}$ yields $|R_{h,n}|
\leq C h n^3\leq C\epsilon^3$ for some $C>0$. Consequently,
\begin{equation}\label{coherence}
    U_h(x,t)
= \xi_h^{-\frac{1}{2}}\eta(t)e^{ik_hx}e^{-\frac{2it}{h^2}}
\sum_{n=1}^{\xi_h}
e^{-in(x-\frac{2t}{h})}e^{-itR_{h,n}}.
\end{equation}
Let $0 < c_0 \ll 1$ such that $\cos(2c_0)>0$, and let $\epsilon \ll 1$ satisfy
\begin{equation}\label{eps_defined}
C \epsilon^3 \sup_{t\in J}| t |\leq c_0.    
\end{equation}
Consider the spacetime tube
\begin{equation*}
    \mathcal{T}_h = \left\{(t,x)\in J\times\mathbb{T}:
\text{dist}_{\mathbb{T}}\left(x,\frac{2t}{h}\right)\leq\frac{c_0}{\xi_h}
\right\}.
\end{equation*}
For $(t,x) \in \mathcal{T}_h$, let $y=x-\frac{2t}{h}\ (\text{mod } 2\pi)$. Then, by definition of $\mathcal{T}_h$ and \eqref{eps_defined},
\begin{equation}\label{phase_control}
    |ny + t R_{h,n}| \leq 2c_0,
\end{equation}
On $\mathcal{T}_h$, by \eqref{coherence} and \eqref{phase_control},
\begin{equation*}
    |U_h(x,t)| = \xi_h^{-\frac{1}{2}} \eta(t) \left|\sum_{n=1}^{\xi_h}
e^{-in(x-\frac{2t}{h})}e^{-itR_{h,n}}\right| \geq \xi_h^{-\frac{1}{2}} \left|\Re \sum_{n=1}^{\xi_h}
e^{-in(x-\frac{2t}{h})}e^{-itR_{h,n}} \right| \gtrsim \xi_h^{\frac{1}{2}}. 
\end{equation*}
Moreover, $|\mathcal{T}_h|\simeq |J| \xi_h^{-1}$. The proof is complete by showing the divergence
\begin{equation*}
\|U_h\|_{L^4(I\times\mathbb{T})}^4
\geq \int_{\mathcal{T}_h}|U_h(t,x)|^4 dxdt
\gtrsim \xi_h^2\cdot \xi_h^{-1} \simeq h^{-\frac{1}{3}}.
\end{equation*}
\end{proof}

\begin{remark}
The counterexample persists for $a\in L^\infty$. Indeed, replace \eqref{coherent_packet} by
\begin{equation*}
U_h(x,t)
=\xi_h^{-\frac{1}{2}}\eta(t)
\sum_{n=1}^{\xi_h}e^{i(k_h-n)x}e^{-i\phi_{h,k_h-n}(t)}.
\end{equation*}
Then, the upper bound in $X_h^{0,b_0}(I)$ is unchanged. For $1\leq n \leq \xi_h$,
\begin{equation*}
\omega_{h,a(t)}(k_h-n)
=\frac{2}{h^2}-\frac{2n}{h}-\frac{2a(t)}{h}
+O\left(h(n+\ainf)^3\right).
\end{equation*}
Choosing $\epsilon, h$ sufficiently small depending on $\ainf$, the lower bound $\|U_h\|_{L^4(I\times\mathbb{T})}\gtrsim h^{-\frac{1}{12}}$ follows similarly.
\end{remark}

\begin{remark}\label{rem:FPU-comparison}
In the periodic FPUT-to-KdV limit of Kwak--Yang \cite{kwak2025periodic}, diagonalization into acoustic branches and removal of the leading transports yield the residual phase
\begin{equation*}
\omega_{h}(k)=\frac{1}{h^2}\left(k-\frac{2}{h}\sin\frac{hk}{2}\right).
\end{equation*}
Uniformly for $|hk|\leq\pi$,
\begin{equation*}
|\omega_h(k)|\simeq |k|^3,
\qquad |\omega_h'(k)|\simeq |k|^2,
\qquad |\omega_h''(k)|\simeq |k|,
\end{equation*}
and hence the lattice dispersion retains the Airy-type phase geometry unlike the lattice Schr\"odinger dispersion.
\end{remark}

The uniform $L^4$ estimate yields the nonlinear bounds needed for uniform local theory.
\begin{lemma}\label{prop:cubic-estimates}
For $s \geq 0$ and $U,V:I\times\mathbb{T}\rightarrow\mathbb{C}$, the following nonlinear bounds hold uniformly in $h,a,I$:
\begin{equation}\label{eq:cubic-discrete}
\|\mathcal N_h(U)\|_{X_h^{s,-b_0}(I)} \lesssim \|P_hU\|_{X_h^{s,b}(I)} \|P_hU\|_{X_h^{0,b}(I)}^2,
\end{equation}
and
\begin{equation}\label{eq:cubic-diff}
\begin{split}
\|\mathcal N_h(U) &-\mathcal N_h(V)\|_{X_h^{s,-b_0}(I)} \lesssim
\|P_h(U-V)\|_{X_h^{s,b}(I)}
\left(
\|P_hU\|_{X_h^{0,b}(I)}
+\|P_hV\|_{X_h^{0,b}(I)}
\right)^2\\
&+ \|P_h(U-V)\|_{X_h^{0,b}(I)}
\left(\|P_hU\|_{X_h^{s,b}(I)}
+\|P_hV\|_{X_h^{s,b}(I)}\right)\left(\|P_hU\|_{X_h^{0,b}(I)}
+\|P_hV\|_{X_h^{0,b}(I)}\right).
\end{split}
\end{equation}
\end{lemma}

\begin{proof}
For spacetime functions $U_1,U_2,U_3$,
the dual estimate \eqref{eq:L43-dual}, and the periodic fractional
Leibniz rule \cite[Proposition 1]{benyi2025fractional} yield
\begin{equation*}
\begin{split}
&\left\|P_h\left((P_hU_1)\overline{P_hU_2}(P_hU_3)\right)
\right\|_{X_h^{s,-b_0}(I)} \lesssim
\left\|\langle \nabla \rangle^s
\left((P_hU_1)\overline{P_hU_2}(P_hU_3)\right)
\right\|_{L^{\frac{4}{3}}(I\times\mathbb{T})}\\
&\lesssim \sum_{j=1}^3
\|\langle \nabla\rangle^sP_hU_j\|_{L^4(I\times\mathbb{T})}
\prod_{\substack{1\leq \ell\leq3\\l\neq j}}
\|P_h U_\ell\|_{L^4(I\times\mathbb{T})} \lesssim \sum_{j=1}^3
\|P_hU_j\|_{X_h^{s,b}(I)}
\prod_{\substack{1\leq \ell\leq3\\l\neq j}}
\|P_h U_\ell\|_{X_h^{0,b}(I)},
\end{split}
\end{equation*}
where the last inequality follows from \eqref{eq:nonauto-L4} and $X^{0,b}_{h} \hookrightarrow X^{0,b_0}_{h}$ for $b_0 < b$. Then, \eqref{eq:cubic-discrete} follows by taking $U_1 = U_2 = U_3$, and \eqref{eq:cubic-diff} follows similarly by expanding the cubic difference $\mathcal{N}_h(U) - \mathcal{N}_h(V)$.
\end{proof}

The mass-dependent time of existence $\delta$ is applied to obtain polynomial Sobolev growth in \Cref{sec:polynomial-regimes}.

\begin{corollary}\label{prop:local}
Let $R_0>0$. There exists $\delta=\delta(R_0)>0$ such that, for any $s \geq 0$ and $u_h\in AC(\mathbb{R};H_h^s)$ solving
\eqref{main_eq1} with
$\|u_h(0)\|_{L_h^2}\leq R_0$, we have, uniformly in $h,a$,
\begin{equation}\label{eq:local-flow-bounds}
\|U_h\|_{X_h^{s,b}([0,\delta])}
\lesssim_{s,R_0}
\|u_h(0)\|_{H_h^s}.
\end{equation}
\end{corollary}

\begin{proof}
Let $\delta \simeq \min\{1, R_0^{-\frac{2}{\theta}}\}$. For $V_h\in X_h^{0,b}([0,\delta])$, define
\begin{equation*}
\begin{split}
B_0 &= \left\{V_h\in X_h^{0,b}([0,\delta]):\|V_h\|_{X_h^{0,b}([0,\delta])} \leq C R_0 \right\},\\
\Gamma_h(V_h)(t)&=S_{h,a}(t,0)U_h(0)-i\int_0^t S_{h,a}(t,\tau)\mathcal N_h(V_h(\tau)) d\tau,    
\end{split}
\end{equation*}
for some $C \gg 1$ depending on the implicit constants in \Cref{lem:linear-estimates}, \Cref{prop:cubic-estimates}. For
$V_h,W_h\in B_0$,
\begin{equation*}
\begin{split}
\|\Gamma_h(V_h)\|_{X_h^{0,b}([0,\delta])}
&\lesssim
\|u_h(0)\|_{L_h^2}
+\delta^\theta
\|V_h\|_{X_h^{0,b}([0,\delta])}^3
\lesssim
R_0+\delta^\theta R_0^3,\\
\|\Gamma_h(V_h)-\Gamma_h(W_h)\|_{X_h^{0,b}([0,\delta])}
&\lesssim
\delta^\theta
\left(
\|V_h\|_{X_h^{0,b}([0,\delta])}
+\|W_h\|_{X_h^{0,b}([0,\delta])}
\right)^2
\|V_h-W_h\|_{X_h^{0,b}([0,\delta])}\\
&\lesssim
\delta^\theta R_0^2
\|V_h-W_h\|_{X_h^{0,b}([0,\delta])},
\end{split}
\end{equation*}
and hence $\Gamma_h$ is a contraction for some $C > 0$ with $U_h = \mathcal{S}_h u_h$ as its solution by the existence and uniqueness of ODEs. Then,
\begin{equation*}
\|U_h\|_{X_h^{0,b}([0,\delta])} \lesssim \|u_h(0)\|_{L_h^2}
+\delta^\theta R_0^2
\|U_h\|_{X_h^{0,b}([0,\delta])},
\end{equation*}
and hence \eqref{eq:local-flow-bounds} at $s=0$ by the definition of $\delta$.

Let $s>0$. For any interval $J=[t_0,t_1]\subseteq[0,\delta]$,
the Duhamel estimate, \Cref{prop:cubic-estimates}, and the $s=0$ estimate yield
\begin{equation*}
\begin{split}
\|U_h\|_{X_h^{s,b}(J)}\lesssim_s \|u_h(t_0)\|_{H_h^s} +|J|^\theta R_0^2 \|U_h\|_{X_h^{s,b}(J)}.
\end{split}
\end{equation*}
Choose $\delta_s=\delta_s(s,R_0)>0$ such that
$\delta_s^\theta R_0^2 \ll 1$. For $|J|\leq\delta_s$, absorb the nonlinear contribution to obtain
\begin{equation*}
\|U_h\|_{X_h^{s,b}(J)} \lesssim_s \|u_h(t_0)\|_{H_h^s}.
\end{equation*}
Cover $[0,\delta]$ by finitely many such intervals. Iterating the
preceding estimate using $X_h^{s,b}(J)\hookrightarrow C(J;H_h^s)$
for $b>\frac{1}{2}$, and then gluing the finitely many restriction
norms, yields \eqref{eq:local-flow-bounds}.
\end{proof}

\section{Birkhoff coordinates and tangent flow}\label{subsec:birkhoff}

We equip $L^2(\mathbb{T})$ with the Birkhoff coordinates given by $\mathcal{B}$, which conjugate the cubic NLS flow to an infinite-dimensional rotation. The chart is global for the defocusing NLS \cite{grebert2014defocusing} and local around the zero solution for the focusing NLS \cite{kappeler2009birkhoff}. Define
\begin{equation*}
H_c^s = H^s \times H^s,
 \qquad
 H_r^s = \{(u,\overline u) : u \in H^s\}, \qquad iH_r^s = \{(u,-\overline u) : u \in H^s\},
\end{equation*}
where we consider $H^s_r,\ i H^s_r$ over $\mathbb{R}$. As real Banach spaces, $H^s_c$ contains two real subspaces, or more precisely, $H^s_c = H^s_r \oplus i H^s_r$. Similarly, define the sequence space
\begin{equation*}
    \mathfrak{h}^s = \left\{ z \in \ell^2(\mathbb{Z};\mathbb{C}):\sum_{n\in\mathbb{Z}}\langle  n \rangle^{2s}|z_n|^2<\infty\right\}
\end{equation*}
and $\mathfrak{h}^s_c,\ \mathfrak{h}^s_r,\ i\mathfrak{h}^s_r$. Define the action
\begin{equation*}
I: \ell^2(\mathbb{Z};\mathbb{C}) \rightarrow \ell^1(\mathbb{Z};\mathbb{R}).\qquad I(z)_n = |z_n|^2,    
\end{equation*}
and recall the Hamiltonian
\begin{equation*}
    H_{\pm} = \frac{1}{2}\int |\nabla u|^2 dx \pm \frac{1}{4} \int |u|^4 dx. 
\end{equation*}

\begin{lemma}\label{birkhoff}
There exists a complex neighborhood $W \subseteq H^0_c$ containing $H^0_r$, on which there is an analytic canonical map $\mathcal{B}: W \rightarrow \mathfrak{h}^{0}_c$. For the defocusing nonlinearity, the global restriction $\mathcal{B}|_{H^s_r}:H^s_r \rightarrow \mathfrak{h}^s_r$ is real bianalytic (the inverse exists and is real-analytic) for any $s \in \{0\} \cup [1,\infty)$. For the focusing nonlinearity, there exists $W_f \subseteq W \cap i H^0_r,\ U_f \subseteq i \mathfrak{h}^0_r$, open in their respective topologies, such that $\mathbf{0} := (0,0) \in W_f$ and the local map $\mathcal{B}: W_f \cap iH^s_r \rightarrow U_f \cap i\mathfrak{h}^s_r$ is real bianalytic for any $s \in \{0\} \cup [1,\infty)$.
\end{lemma}

\begin{proof}
    See \cite[Theorem 20.2]{grebert2014defocusing} and \cite[Theorem 2.1]{kappeler2017scattering} for the construction of $\mathcal{B}: W \rightarrow \mathfrak{h}^{0}_c$, its analytic regularity, and canonicity. See \cite[Corollary 1.1]{kappeler2016semilinearity} for the bianalyticity for $s \in \{0\} \cup [1,\infty)$ in the defocusing regime. In the focusing regime, the existence of $W_f$ and bianalyticity of $\mathcal{B}: W_f \rightarrow U_f$ at $s=0$ is given in \cite[Theorem 1.1]{kappeler2009birkhoff}; let $\mathcal{B}: \widetilde{W}_f \rightarrow \widetilde{U}_f$ be biholomorphic complex extensions. The same statement gives diffeomorphism for $s \in [1,\infty)$. For completion, we extend that result to analyticity.

    Let $N = \lfloor s \rfloor$. We claim $\mathcal{B}(\widetilde{W}_f \cap H^s_c) = \widetilde{U}_f \cap \mathfrak{h}^s_c$. For $\mathbf{u} \in \widetilde{W}_f \cap H^s_c \hookrightarrow \widetilde{W}_f \cap H^N_c$, the difference $\mathcal{B}(\mathbf{u}) - \mathcal{F}(\mathbf{u}) \in \mathfrak{h}^{N+1}_{c}$ by \cite[Theorem 1.1]{kappeler2016semilinearity} where $\mathcal{F}$ is the Fourier transform; more precisely, the stated result applies to the complexified $\mathcal{B}$, and therefore is not restricted only to the defocusing NLS. Since $\mathcal{F}$ is an isometry, $\mathcal{B}(\mathbf{u}) \in \mathfrak{h}^{s}_c$. By \cite[Corollary 1.1]{kappeler2016semilinearity}, the inverse difference $\mathcal{B}^{-1} - \mathcal{F}^{-1}$ is also one-smoothing, and hence the claim by a similar reasoning. 
        
    The proof of \cite{kappeler2016semilinearity} yields holomorphicity of $\mathcal{B}-\mathcal{F}$ in the complexified domain, which includes $\widetilde{W}_f$. The holomorphicity of $\mathcal{B}$ follows from the composition of maps
    \begin{equation*}
        \widetilde{W}_f \cap H^s_c \hookrightarrow \widetilde{W}_f \cap H^N_c \xrightarrow[]{\mathcal{B}-\mathcal{F}} \mathfrak{h}^{N+1}_{c} \hookrightarrow \mathfrak{h}^{s}_{c},
    \end{equation*}
    where the inverse argument follows similarly by $\mathcal{B}^{-1} - \mathcal{F}^{-1} = -\mathcal{F}^{-1}(\mathcal{B}-\mathcal{F})\mathcal{B}^{-1}$. 
\end{proof} 

To study the global dynamics of the focusing NLS within a local Birkhoff chart, we shrink $U_f \subseteq i \mathfrak{h}^{0}_{r}$, if necessary. Since there exists $\delta>0$ such that $B_{i\mathfrak{h}_{r}^{0}}(\mathbf{0},\delta) \subseteq U_f$, we may take $U_f = B_{i\mathfrak{h}_{r}^{0}}(\mathbf{0},\delta)$ and $W_f := \mathcal{B}^{-1}(U_f)$. Then, $U_f$ is invariant under coordinate-wise rotation. Denote 
\begin{equation*}
\mathcal{U}_{+}^{1} = \mathfrak{h}^1,\qquad \mathcal{U}_{-}^1 = \{z \in \mathfrak{h}^{1}: (z,-\overline{z}) \in U_f \cap i \mathfrak{h}^{1}_{r}\}.    
\end{equation*}
Identify $H^s$ with $H^s_r$ by $u \mapsto \mathbf{u} = (u,\overline{u})$ in the defocusing nonlinearity, and similarly $H^s$ with $iH^s_r$ by $\mathbf{u} = (u,-\overline{u})$ in the focusing nonlinearity. Under these identifications, a map on $H^s$ uniquely determines, and is uniquely determined by, the corresponding map on $H_r^s$ or $iH_r^s$, according to the sign of nonlinearity; we use the same notation for these equivalent realizations.

\begin{remark}\label{rotation}
The Hamiltonians $\mathscr{H}_{\pm} := H_{\pm} \circ \mathcal{B}^{-1}$ are real-analytic on $\mathcal{U}_{\pm}^{1}$, respectively, and depend only on the actions $I$. With $\omega_n^{\pm}(I) := \partial_{I_n} \mathscr{H}_{\pm}(I)$, Hamilton's equations yield the global flow representation on $\mathcal{U}_{\pm}^{1}$
\begin{equation}\label{eq:birkhoff-flow-real}
R_t = \mathcal{B} \circ \Phi_t \circ \mathcal{B}^{-1}: \mathcal{U}^{1}_{\pm} \rightarrow \mathcal{U}^{1}_{\pm},\qquad (R_t z)_n = e^{-i \omega_n^{\pm}(I(z)) t}z_n,\qquad \forall t \in \mathbb{R}. 
\end{equation}
The claim that $\mathscr{H}_{+}$ depends only on $I$ follows from \cite[Theorem 2.1]{kappeler2016semilinearity}, and the claim for $\mathscr{H}_{-}$, from \cite[Theorem 1.1]{kappeler2009birkhoff}. The regularity claim follows from \Cref{birkhoff}. Since $\mathcal{B}$ is canonical, Hamilton's equations in Birkhoff
coordinates yield
\eqref{eq:birkhoff-flow-real}.
\end{remark}

\begin{remark}
\label{lem:frequency-correction}
For $z\in\mathcal U_\pm^1$, define $\widetilde{\omega}_n^\pm(z) = \omega_n^\pm(I(z))-n^2,\ n \in \mathbb{Z}$. Then, $\widetilde\omega^\pm: \mathcal{U}_{\pm}^1 \rightarrow \ell^\infty(\mathbb Z;\mathbb R)$ is real-analytic. With $W \subseteq H^{0}_c$ as in \Cref{birkhoff}, the holomorphicity of the frequency correction map, defined on $W \cap H^1_c$ taking values in $\ell^{\infty}(\mathbb{Z};\mathbb{C})$, is stated in \cite[Corollary 2.1]{kappeler2017scattering}. By composition with $\mathcal{B}^{-1}$ and the real-valuedness of $\mathscr{H}_{\pm}$, our claim on the regularity of $\widetilde{\omega}^{\pm}$ follows.
\end{remark}

To control the perturbations of the NLS flow, we derive linear estimates on $D \Phi_t$ where $L^2(\mathbb{T};\mathbb{C})$ and $\mathcal{L}(L^2,L^2)$ are considered over $\mathbb{R}$. It follows from \cite{BourgainNLS1} that for all $t \in \mathbb{R}$, the flow $\Phi_t$ is real analytic on $L^2$, and for every $R,T>0$, there exists $M(R,T) < \infty$ such that
\begin{equation*}
 \sup_{\| u \|_{L^2} \leq R,\ |t| \leq T}
 \| D \Phi_t[u] \|_{\mathcal L(L^2,L^2)}
 \leq M.
\end{equation*}
The time of local existence depends on $R$, and by mass conservation, any $T>0$ can be reached by finite iteration. However, this method yields an at-most exponential dependence of $M$ in $T$. Conjugation by $\mathcal{B}$ reduces this to linear dependence.

\begin{lemma}\label{lem:coordinatewise-rotation-compactness}
Let $\rho \geq 0$. If $Z \Subset\mathfrak{h}^\rho$, then $Z \subseteq \Theta_{Z} \Subset \mathfrak{h}^\rho$ where
\begin{equation}\label{eq:coordinatewise-rotation-set}
 \Theta_Z := \left\{ w \in \mathfrak{h}^\rho: \exists z \in Z \text{ such that } I(w) = I(z) \right\}.
\end{equation}
\end{lemma}

\begin{proof}
The inclusion $Z \subseteq \Theta_{Z}$ is by definition. Let $S^1$, the group of unit circle, act on $\mathbb{C}$ by multiplication. Define
\begin{equation*}
    F: Z \times (S^1)^{\mathbb{Z}} \rightarrow \mathfrak{h}^{\rho},\qquad (z,\varphi) \mapsto (\varphi_n z_n),\qquad \varphi_n \in S^1.
\end{equation*}
By Tychonoff's theorem, $(S^1)^{\mathbb{Z}}$ is compact. Observing that $F(Z \times (S^1)^{\mathbb{Z}}) = \Theta_Z$ and
\begin{equation*}
    |\varphi^{(j)}_n z^{(j)}_{n} - \varphi_n z_n| \leq |z^{(j)}_n - z_n| + |\varphi^{(j)}_n - \varphi_n| |z_n|
\end{equation*}
whenever $(z^{(j)}, \varphi^{(j)}) \xrightarrow[j \rightarrow \infty]{} (z,\varphi)$, the continuity of $F$ follows by the Dominated Convergence Theorem, and hence the compactness of $\Theta_Z$.
\end{proof}

\begin{proposition}\label{prop:birkhoff-tangent-bounds}
Let $\Phi_t$ denote the defocusing flow. For any $\rho \in \{0\} \cup [1,\infty)$ and $K \Subset H^{\rho}$, there exists $C = C(\rho,K) > 0$ such that 
\begin{equation}\label{eq:compact-birkhoff}
 \sup_{u \in K} \| D \Phi_t [u] f \|_{H^\rho}
 \leq C \langle t \rangle \| f \|_{H^\rho},
 \qquad f \in H^\rho,\ t \in \mathbb{R}.
\end{equation}
If $0 < \rho < 1$ and $K \Subset H^1$, then there exists $C = C(\rho,K) > 0$ such that \eqref{eq:compact-birkhoff} holds in $H^{\rho}$.
\end{proposition}

\begin{proof}
Let $Z = \mathcal{B}(K)$ and $\Theta_{Z}$ be defined as \eqref{eq:coordinatewise-rotation-set}. For any $t \in \mathbb{R}$, $R_t(Z) \subseteq \Theta_{Z}$, which follows from the coordinate-wise rotation \eqref{eq:birkhoff-flow-real} at $\rho \geq 1$ and by density at $\rho = 0$. By continuity and compactness,
\begin{equation*}
    \sup_{z \in \Theta_{Z}} \| D\mathcal{B}^{-1} [z]\|_{\mathfrak{h}^{\rho}\rightarrow H^{\rho}} + \sup_{u \in K} \| D\mathcal{B} [u] \|_{H^{\rho}\rightarrow \mathfrak{h}^{\rho}} = O_{\rho,K}(1).
\end{equation*}
Assume $\rho \geq 1$ and let $z, \zeta \in \mathfrak{h}^{\rho}$. Then,
\begin{equation}\label{deriv_rotation}
    (D R_t [z] \zeta)_n = e^{-i \omega_{n}^{+}(z)t} \zeta_n - it e^{-i \omega_{n}^{+}(z)t} z_n D \omega_n^{+} [z] \zeta.
\end{equation}
If we further assume $z \in Z$, then
\begin{equation*}
\begin{split}
\| D R_t [z] \zeta\|_{\mathfrak{h}^{\rho}} &\leq \| \zeta \|_{\mathfrak{h}^{\rho}} + |t| \| z \|_{\mathfrak{h}^{\rho}} \| D \widetilde{\omega}^{+} [z] \zeta\|_{\ell^{\infty}}\\
&\leq \| \zeta \|_{\mathfrak{h}^{\rho}} + C_K |t| \sup_{z^{\prime} \in Z} \| D \widetilde{\omega}^{+} [z^{\prime}]\|_{\mathfrak{h}^{1} \rightarrow \ell^{\infty}} \| \zeta \|_{\mathfrak{h}^{1}} \lesssim \langle t \rangle \| \zeta \|_{\mathfrak{h}^{\rho}},
\end{split}
\end{equation*}
where the second inequality follows by \Cref{lem:frequency-correction} and the last inequality by $\mathfrak{h}^{\rho} \hookrightarrow \mathfrak{h}^{1}$. The claim follows by the chain rule on $\Phi_t = \mathcal{B}^{-1} R_t \mathcal{B}$.

Let $\rho = 0$. By \Cref{birkhoff}, $R_t = \mathcal{B} \circ \Phi_t \circ \mathcal{B}^{-1}: \mathfrak{h}^{0} \rightarrow \mathfrak{h}^{0}$ is well-defined. For $z,\zeta \in \mathfrak{h}^{1}$, consider \eqref{deriv_rotation} at the unit time, subtract $e^{-i \omega_n^{+}} \zeta_n$, and multiply $t e^{-i \omega_n^{+}(z)(t-1)}$ both sides to obtain
\begin{equation*}
    - it e^{-i \omega_{n}^{+}(z)t} z_n D \omega_n^{+} [z] \zeta = t e^{-i \omega_n^{+}(z)(t-1)}\left((D R_1 [z] \zeta)_n - e^{-i \omega_{n}^{+}(z)} \zeta_n \right). 
\end{equation*}
Substituting into \eqref{deriv_rotation}, we have
\begin{equation*}
    (D R_t [z] \zeta)_n = e^{-i \omega_{n}^{+}(z)t} \zeta_n + t e^{-i \omega_n^{+}(z)(t-1)}\left((D R_1 [z] \zeta)_n - e^{-i \omega_{n}^{+}(z)} \zeta_n \right),
\end{equation*}
and hence
\begin{equation}\label{eq:compact-birkhoff1}
    \| D R_t[z] \zeta \|_{\mathfrak{h}^{0}} \leq \left( 1 + |t| \left(\| D R_1[z] \|_{\mathfrak{h}^{0}\rightarrow \mathfrak{h}^{0}} + 1 \right) \right)\| \zeta \|_{\mathfrak{h}^{0}}
\end{equation}
By density argument using $\mathfrak{h}^{1}\hookrightarrow \mathfrak{h}^{0}$, \eqref{eq:compact-birkhoff1} holds for all $\zeta \in \mathfrak{h}^{0}$. By an approximation $\mathfrak{h}^{1} \ni z^{(j)} \rightarrow z \in Z$ in $\mathfrak{h}^{0}$ by finite truncations and the continuity of $DR_t [\cdot],\ DR_1 [\cdot]$ in the operator norm in $\mathfrak{h}^{0}$, \eqref{eq:compact-birkhoff1} passes to the limit, and hence \eqref{eq:compact-birkhoff} at $\rho = 0$. Since a compact subset of $H^1$ is compact in $L^2$, the second statement follows by interpolating at $\rho \in \{0,1\}$.
\end{proof}

\begin{remark}\label{rmk_focusing}
    The derivative estimate for $\Phi_t$ is stated for the defocusing nonlinearity for notational convenience. The corresponding statement holds for the focusing nonlinearity provided the compact subset is compactly contained in the Birkhoff neighborhood $W_f$.
\end{remark}

By conjugacy, \Cref{prop:birkhoff-tangent-bounds} transfers to $\Phi_{a}(t,\tau)$. The result is given for the defocusing flow. 

\begin{proposition}\label{prop:magnetic-conjugation}
For any $\rho\in\{0\}\cup[1,\infty)$ and $K\Subset H^\rho$, there exists
$C=C(\rho,K)>0$ such that
\begin{equation}\label{eq:magnetic-compact-birkhoff}
    \sup_{u\in K}
    \|D\Phi_a(t,\tau)[u]f\|_{H^\rho}
    \leq
    C\langle t-\tau\rangle\|f\|_{H^\rho},
    \qquad f\in H^\rho,\quad t,\tau\in\mathbb{R}.
\end{equation}
If $0<\rho<1$ and $K\Subset H^1$, then there exists $C=C(\rho,K)>0$ such that
\eqref{eq:magnetic-compact-birkhoff} holds in $H^{\rho}$.  All operator constants are independent of $a$.
\end{proposition}

\begin{proof}
For $K\Subset H^\rho$, define $\widetilde{K} = \left\{e^{i\gamma} u(\cdot-y): u \in K,\ \gamma,y \in \mathbb{T} \right\}$. Then, $\widetilde{K} \Subset H^\rho$ by the continuity of translations and phase rotations, and
\begin{equation*}
    \mathcal{G}_a(\tau) K \subseteq \widetilde{K}.
\end{equation*}
Since $\mathcal{G}_a(t)$ is an isometry on $H^\rho$, \Cref{prop:birkhoff-tangent-bounds} gives
\begin{equation*}
    \| D\Phi_a(t,\tau) [u] f \|_{H^\rho} = \| D \Phi_{t-\tau} [\mathcal{G}_a(\tau)u] \mathcal{G}_a(\tau) f \|_{H^\rho} \leq C \langle t-\tau\rangle\|f\|_{H^\rho}.
\end{equation*}
The case $0<\rho<1$ follows from the second statement of
\Cref{prop:birkhoff-tangent-bounds}.
\end{proof}

For the focusing nonlinearity, further assume $\widetilde{K}$ is identified with a set that is compactly contained in the local Birkhoff neighborhood at the regularity required in \Cref{prop:birkhoff-tangent-bounds}.

\section{Modified energies and polynomial Sobolev growth}\label{sec:modified-energies}

We prove that the discrete Sobolev norms exhibit at most polynomial growth globally in time.
\begin{proposition}
\label{prop:polynomial-orbit-growth}
Let $s>0$, $R>0$, and let
$u_h \in AC(\mathbb{R};H^s_h)$ satisfy
$\| u_h(0) \|_{H_h^s} \leq R$.

\begin{enumerate}[label=\textnormal{(\roman*)}]
\item \label{item:orbit-filtered}
Let $0<s<1$ and let $u_h$ solve
\eqref{main_eq1}. If $h \leq h_0$, with $h_0$
given by \eqref{eq:phase-margin}, then
\begin{equation*}
\begin{split}
 \| u_h(t) \|_{H^s_h}
 &\lesssim_{s,R,\Lambda} 
 \begin{cases}
  \langle t \rangle^{\frac{1}{2}}+(|a|_\infty |t|)^s,
     &0<s \leq \frac{1}{2},\\
  \langle t \rangle^2+(|a|_\infty |t|)^s,
     &\frac{1}{2}<s<1.
 \end{cases}
\end{split}
\end{equation*}

\item \label{item:orbit-full-band}
Let $s \geq 1$ and let $u_h$ solve \eqref{main_eq1}.
\begin{equation}\label{eq:fractional-growth}
 \| u_h(t) \|_{H^s_h} \lesssim_{s,R}
  \langle t \rangle^{\frac{1}{2}\lceil s-1\rceil}
  +(|a|_\infty |t|)^s.
\end{equation}
\end{enumerate}
\end{proposition}

\begin{corollary}\label{cor:prop71}
Let $s \geq 1$ and assume the hypothesis of \Cref{prop:polynomial-orbit-growth}. Define
\begin{equation}\label{eq:beta-defect}
 \beta_h(t) = \int_{-|t|}^{|t|}
 \left|\sin\left(\frac{ha(\tau)}{2}\right)\right| d\tau.
\end{equation}
There exists $\epsilon_R > 0$ such that, whenever
$\beta_h(t) \leq \epsilon_R$, we have
\begin{equation*}
\| u_h(t) \|_{H^s_h} \lesssim_{s,R} \langle t \rangle^{\frac{1}{2}\lceil s-1\rceil}.
\end{equation*}    
\end{corollary}

\begin{remark}\label{rem:continuum-orbit-bounds}
For $s\geq0$ and $\mu\in\{\pm1\}$, the periodic cubic NLS flow satisfies
\begin{equation}\label{eq:case-I-continuum-orbit}
 \| \Phi_t u_0 \|_{H^s}\lesssim_{s,R}1,
 \qquad
 \| u_0 \|_{H^s}\leq R,\quad t\in\mathbb{R},
\end{equation}
with $\|u_0\|_{L^2}$ sufficiently small, independently of $s,R$, when $\mu=-1$ and $s>1$.
The cases $s \in \{ 0,1 \}$ follow from mass and energy conservation,
using the Gagliardo-Nirenberg inequality in the focusing case.
For $0<s<1$, \eqref{eq:case-I-continuum-orbit} follows from the Besov bounds in \cite[Theorem 4.5]{killip2018low} based on complete integrability. Without complete integrability, the weaker bound
$\| \Phi_t u_0 \|_{H^s}\lesssim_{s,R}\langle t\rangle^{\frac{1}{2}}$
also follows by combining the uniform bounds for
$0\leq s<\frac{1}{2}$ in \cite[Theorem B.1(ii)]{oh2020global}
with the nonlinear smoothing estimates in \cite[Theorem 1.1]{erdougan2019smoothing}. For $s>1$, use
\cite[Corollary 1.1]{kappeler2017scattering} in the defocusing case
and the small-mass Birkhoff argument in \Cref{birkhoff} and \Cref{rotation} in the focusing case. By \eqref{eq:magnetic-flow-conjugation}, the same bounds hold for
the magnetic flow, independently of $|a|_\infty$.
\end{remark}

Motivated by \cite{Bernier}, we construct coercive modified energies that are almost conserved. Define
\begin{equation*}
    Q(z) = \sum_{k \in \mathbb{T}_h^{*}} q(k) |z_k|^2, \qquad q: \mathbb{T}_h^{*} \rightarrow \mathbb{R}.
\end{equation*}
Let $\chi = \mathbf{1}_{E_h}$ for $s < 1$ and $\chi \equiv 1$ for $s \geq 1$. Let $z = \widehat{u}_h$, and therefore $z \in AC(\mathbb{R};L^2(\mathbb{T}_h^{*}))$ satisfies $\supp z(t) \subseteq \{k \in \mathbb{T}_h^{*} : \chi(k) = 1 \}$ and is a solution of
\begin{equation}\label{eq:abstract-projected-cyclic-flow}
 i \partial_t z_k = \omega_{h,a(t)}(k)z_k
 +\mu\chi(k)\sum_{k_1 - k_{-1} + k_2 = k} z_{k_1}\overline{z_{k_{-1}}} z_{k_2}, \qquad \text{ a.e. } t \in \mathbb{R}.
\end{equation}
Differentiating $Q$ in $t$ and applying the cubic flow in \eqref{eq:abstract-projected-cyclic-flow}, a quartic remainder appears. To eliminate this remainder, a higher order term is added to $Q$, which results in a sextic remainder with a magnetic defect.

Let $\ell \in \mathbb{N}$ and $m: (\mathbb{T}_h^{*})^{2\ell} \rightarrow \mathbb{R}$. Let $\Sigma_\ell=\{\pm 1,\ldots,\pm \ell\}$ and $\mathbf{f}=(f_1,\ldots,f_\ell,f_{-1},\ldots,f_{-\ell})$ where $f_{j}:\mathbb{T}_h^{*} \rightarrow \mathbb{C}$ for $j \in \Sigma_{\ell}$. An order $2\ell$ (energy) multilinear form is defined as $\Lambda_{2\ell}$ where
\begin{equation*}
\begin{split}
\mathcal{V}_\ell
 &=\left\{\mathbf{k}=(k_1,\ldots,k_\ell,k_{-1},\ldots,k_{-\ell})\in(\mathbb{T}_h^{*})^{2\ell}:
 \sum_{j=1}^\ell(k_j-k_{-j})=0\right\},\\
 \Lambda_{2\ell}(m;\mathbf{f})
 &=\sum_{\mathbf{k} \in \mathcal{V}_\ell}m(\mathbf k)
 \prod_{j=1}^\ell f_j(k_j)\overline{f_{-j}(k_{-j})}, \qquad \Lambda_{2\ell}(m;f) := \Lambda_{2\ell}(m;f,\ldots,f). 
\end{split}
\end{equation*}
We call the condition in $\mathcal{V}_{\ell}$ the momentum constraint. Define the insertion $S_{2\ell}:L^\infty(\mathcal{V}_\ell)\rightarrow L^\infty(\mathcal{V}_{\ell+1})$ by
\begin{equation*}
\begin{split}
(S_{2\ell}m)(\mathbf{k})
 =&\sum_{j=1}^\ell
 \chi(k_j+k_{\ell+1}-k_{-(\ell+1)}) m(k_1,\ldots,\underbrace{k_j+k_{\ell+1}-k_{-(\ell+1)}},\ldots,k_\ell,
 k_{-1},\ldots,k_{-\ell})\\
 &-\sum_{j=1}^\ell
 \chi(k_{-j}+k_{-(\ell+1)}-k_{\ell+1}) m(k_1,\ldots,k_\ell,
 k_{-1},\ldots,\underbrace{k_{-j}+k_{-(\ell+1)}-k_{\ell+1}},\ldots,k_{-\ell}).    
\end{split}
\end{equation*}
Define $D_\ell:L^\infty(\mathbb{T}_h^{*})\rightarrow L^\infty((\mathbb{T}_h^{*})^{2\ell})$ by
\begin{equation*}
 D_\ell q(\mathbf{k}) = \sum_{j=1}^\ell q(k_j)-q(k_{-j}).
\end{equation*}

\begin{lemma}\label{lem:finite-cyclic-product-rule}
For any $\ell \in \mathbb{N}$, $m \in L^{\infty}(\mathcal{V}_{\ell};\mathbb{R})$, and $z \in AC(\mathbb{R};L^2(\mathbb{T}_h^{*}))$, a solution to \eqref{eq:abstract-projected-cyclic-flow}, we have
\begin{equation*}
i\partial_t \Lambda_{2\ell}(m;z) = \Lambda_{2\ell}(m D_\ell\omega_{h,a(t)};z )
 +\mu\Lambda_{2\ell+2}(S_{2\ell}m; z), \qquad \text{ a.e. } t \in \mathbb{R}.
\end{equation*}
Suppose there exist $q: \mathbb{T}_h^{*} \rightarrow \mathbb{R}$ and $m_q \in L^{\infty}(\mathcal{V}_{2}; \mathbb{R})$ such that
\begin{equation}\label{eq:abstract-homological-equation}
\prod_{j \in \Sigma_2} \chi(k_j)
\left(
m_q( \mathbf{k} ) D_2 \omega_{h,0} ( \mathbf{k} )
- \frac{\mu}{2} D_2 q ( \mathbf{k} ) \right) = 0,
\qquad
\forall\, \mathbf{k}\in\mathcal V_2.
\end{equation}
Then, $\mathcal{E}(z) := Q(z) + \Lambda_{4}(m_q;z)$ satisfies
\begin{equation}\label{eq:abstract-modified-energy-identity}
 i\partial_t \mathcal E(z) = \mu\Lambda_{6}(S_{4}m_q;z) + \Lambda_{4}(m_q D_2(\omega_{h,a(t)}-\omega_{h,0});z),\qquad \text{ a.e. } t \in \mathbb{R}.
\end{equation}
\end{lemma}

\begin{proof}
A straightforward adaptation of Lemma 3.1 and Corollary 3 of \cite{Bernier} yields the statement.
\end{proof}

To control the growth of $\mathcal{E}$, we construct $m_q$ that satisfies the homological equation \eqref{eq:abstract-homological-equation}. Define
\begin{equation}\label{bernier_weight}
    q(k)=
\begin{cases}
\langle k \rangle^{2s},\qquad  & 0<s<1,\\
h^{-2s}f_s(hk),\qquad & s\geq 1, 
\end{cases}; \qquad
f_s(\theta) = \frac{|\sin(\theta/2)|^{2s}}
{|\sin(\theta/2)|^{2s} + |\cos(\theta/2)|^{2s}},\ \theta \in \mathbb{T}.
\end{equation}
On $\mathcal{V}_2$,
\begin{equation}\label{eq:case-I-factorization}
 D_2\omega_{h,0}
 = \frac{8}{h^2}
 \cos\left(\frac{h(k_{-1}+k_{-2})}2\right)
 \sin\left(\frac{h(k_1-k_{-1})}2\right)
 \sin\left(\frac{h(k_1-k_{-2})}2\right),
\end{equation}
and the resonant set is
\begin{equation*}
\begin{split}
\mathcal{Z} &= \{\mathbf{k} \in \mathcal{V}_2:
D_2 \omega_{h,0}(\mathbf{k})=0\} = \mathcal{Z}_c \cup \mathcal{Z}_d,\\
\mathcal{Z}_c &:= \{k_1=k_{-1},\ k_2=k_{-2}\} \cup \{k_1=k_{-2},\ k_{-1}=k_2\},\\
\mathcal{Z}_d
&=\{\mathbf{k}\in\mathcal V_2:k_1+k_2=\pm \frac{\pi}{h}\}.
\end{split}
\end{equation*}
For $s < 1$, since $\chi = \mathbf{1}_{E_h}$, \eqref{eq:abstract-homological-equation} imposes non-trivial restrictions on $\mathcal{V}_{2} \cap E_h^4$ where $E_h^4$ is the four-fold direct product of $E_h$. It follows that $\mathcal{Z}_d \cap E_h^4 = \emptyset$ and
\begin{equation*}
| D_2 \omega_{h,0}(\mathbf{k}) |
 \simeq_\Lambda
 |(k_1-k_{-1})(k_1-k_{-2})|,
 \qquad \mathbf{k} \in \mathcal{V}_2 \cap E_h^4.
\end{equation*}
For $s \geq 1$, the resonance structure includes the genuinely discrete component $\mathcal{Z}_d$.

\subsection{Fourier-filtered modified energies}

The following divisor estimate shows that the quartic homological equation is nonsingular across $\mathcal{Z}$.

\begin{lemma}\label{lem:case-I-divisor}
Let $0<s<1$. Define
\begin{equation*}
    \mathcal{V}_2^{\mathbb{R}} = \{\mathbf{k} = (k_1,k_2,k_{-1},k_{-2}) \in \mathbb{R}^4: \sum_{j=1}^{2} k_j-k_{-j}=0\}.
\end{equation*}
For $\mathbf{k} \in \mathcal{V}_{2}^{\mathbb{R}} \cap [-\frac{\Lambda \pi}{h},\frac{\Lambda \pi}{h}]^4$, define
\begin{equation*}
m_q(\mathbf{k})
=\frac{\mu}{2}
\frac{D_2 q(\mathbf{k})}{D_2\omega_{h,0}(\mathbf{k})},\qquad D_2\omega_{h,0}(\mathbf{k})\neq 0.
\end{equation*}
Then, $m_q$ admits a unique continuous real-valued extension to $\mathcal{V}_{2}^{\mathbb{R}} \cap [-\frac{\Lambda \pi}{h},\frac{\Lambda \pi}{h}]^4$. Restricting this extension to $\mathcal{V}_{2} \cap E_h^4$ and extending to zero on $\mathcal{V}_2 \setminus E_h^{4}$ yields $m_q\in L^\infty(\mathcal{V}_2;\mathbb{R})$ with the norm independent of $h$.
\end{lemma}

\begin{proof}
For $f\in C^2(\mathbb R)$, the momentum constraint and two applications of the Fundamental Theorem of Calculus yield
\begin{equation*}
D_2 f(\mathbf{k}) = (k_1-k_{-1})(k_1-k_{-2}) \int_0^1 \int_0^1
f^{\prime\prime}\left(k_2+t(k_1-k_{-1})+r(k_1-k_{-2})\right)
\,dt\,dr.
\end{equation*}
Since $\omega_{h,0}^{\prime\prime}(k)=2\cos(hk)$, we have
\begin{equation}\label{eq:case-I-divisor-extension}
m_q(\mathbf{k}) = \frac{\mu}{4}
\frac{\int_0^1 \int_0^1
q^{\prime\prime}\big(k_2+t(k_1-k_{-1})+r(k_1-k_{-2})\big) dt dr}{\int_0^1 \int_0^1
\cos \big(h( k_2 + t (k_1-k_{-1}) + r (k_1-k_{-2}))\big) dt dr},\qquad D_2 \omega_{h,0} (\mathbf{k}) \neq 0.
\end{equation}
For $t,r \in [0,1]$, the momentum constraint yields
\begin{equation}\label{convex}
k_2 + t(k_1-k_{-1}) + r(k_1-k_{-2}) = tr k_1 + (1-t)r k_{-1} + (1-t)(1-r)k_2 + t(1-r)k_{-2},   
\end{equation}
which is a convex combination of four frequencies, and hence $\left| h\big(k_2+t(k_1-k_{-1})+r(k_1-k_{-2})\big) \right| \leq \Lambda \pi$ and the integral in the denominator of \eqref{eq:case-I-divisor-extension} is strictly positive. Observing
\begin{equation}\label{q_second}
| q^{\prime\prime}(k) | = | 2s (1 + k^2)^{s-2} \left( 1 + (2s-1) k^2 \right) | \leq 2s (1 + k^2)^{s-1} \leq 2s,
\end{equation}
both integrals are continuous in $\mathbf{k}$ by the Dominated Convergence Theorem, and
\begin{equation}\label{mq_bound}
| m_q( \mathbf{k} ) | \leq \frac{s}{2\cos (\Lambda \pi)},    
\end{equation}
independent of $h$, for all $\mathbf{k}$ in the constraint set. Uniqueness follows from the density of $\{ D_2 \omega_{h,0}(\mathbf{k}) \neq 0 \}$ in $\mathcal{V}_{2}^{\mathbb{R}} \cap [-\frac{\Lambda \pi}{h},\frac{\Lambda \pi}{h}]^4$. Restriction to lattices preserves \eqref{mq_bound}.
\end{proof}

\begin{proposition}\label{prop:case-I-quartic}
Let $0 < s < 1,\ q(\xi) = \langle \xi \rangle^{2s}$, and $m_q \in L^{\infty}(\mathcal{V}_2;\mathbb{R})$ be defined as in \Cref{lem:case-I-divisor}. Assume $f_j,F: \mathbb{T} \rightarrow \mathbb{C}$ are bandlimited in $E_h$ for $j \in \Sigma_{2}$. The following estimates hold uniformly in $h$ with the implicit constants depending on $s,\Lambda,p_j$.
\begin{enumerate}[label=\textnormal{(\roman*)}]
\item If $0 < s \leq \frac{1}{2}$, then
\begin{align}
|\Lambda_4(m_q;\widehat{f}_1,\widehat{f}_2,\widehat{f}_{-1},\widehat{f}_{-2})| &\lesssim \prod_{j \in \Sigma_2} \| f_j\|_{L^2}, \label{eq:case-I-quartic-L2}\\
|\Lambda_4(m_q;\widehat{F},\widehat{f}_2,\widehat{f}_{-1},\widehat{f}_{-2})| &\lesssim \| F\|_{L^{\frac{4}{3}}} \| f_2\|_{L^2} \|f_{-1}\|_{L^2}\| f_{-2}\|_{L^4}. \label{eq:case-I-quartic-mixed}
\end{align}
The same bound holds for any permutation of the integrability indices $(\frac{4}{3},2,2,4)$.
\item If $\frac{1}{2} < s < 1,\ 1 < p_j < \infty$ for $j \in \Sigma_2$, and $\sum_{j \in \Sigma_2} p_j^{-1}=1$, then
\begin{equation}\label{eq:case-I-quartic-general} |\Lambda_4(m_q;\widehat{f}_1,\widehat{f}_2,\widehat{f}_{-1},\widehat{f}_{-2})| \lesssim \prod_{j \in \Sigma_2}\| f_j\|_{L^{p_j}}.
\end{equation}
\end{enumerate}
\end{proposition}

\begin{lemma}\label{lem:case-I-two-slope}
Let $J \subseteq \mathbb{R}$ be a bounded interval and $c_1, c_2 \in \mathbb{R}$ satisfy $0 < |c_1 - c_2||J| < 2\pi$. For all $f,g \in L^1(\mathbb{T})$,
\begin{equation*}
 \int_{\mathbb{T}\times J}|f(\theta + c_1 y) g(\theta + c_2 y)| d\theta dy
 \leq |c_1 - c_2|^{-1} \| f \|_{L^1} \| g \|_{L^1}.
\end{equation*}
\end{lemma}

\begin{proof}
The change of variable $(\theta,y) \mapsto (\theta + c_1 y, \theta + c_2 y)$ is bijective onto the range due to $|c_1 - c_2| |J| < 2\pi$ with Jacobian determinant of absolute value $|c_1 - c_2|$, and hence the bound due to the Fubini's theorem.
\end{proof}

\begin{proof}[Proof of \Cref{prop:case-I-quartic}]
For $\alpha>0$, let $G_\alpha = \mathcal{F}_{\mathbb{R}}^{-1} \left( \langle \cdot \rangle^{-\alpha} \right)$. By \cite{aronszajn1961theory}, $G_\alpha \in L^1(\mathbb{R})$ for any $\alpha>0$; for $\alpha > 1$, $G_\alpha$ is bounded and exponentially decaying at infinity. By \eqref{q_second},
\begin{equation*}
K := \mathcal{F}_{\mathbb{R}}^{-1}(q^{\prime\prime})
=2s(2s-1)G_{2-2s}+4s(1-s)G_{4-2s}
\in L^1(\mathbb{R}),
\end{equation*}
and therefore if $s \leq \frac{1}{2}$, then
\begin{equation}\label{eq:case-I-Wiener-kernel}
\sum_{ n \in \mathbb{Z} }\sup_{ y \in [n,n+1] } |K(y)| < \infty.
\end{equation}
As in \eqref{eq:case-I-divisor-extension}, define
\begin{equation*}
b(\mathbf{k})
=\left(
\int_0^1\int_0^1
\cos\bigl(h(k_2+t(k_1-k_{-1})+r(k_1-k_{-2}))\bigr)
dt dr
\right)^{-1},\qquad \mathbf{k} \in \mathcal{V}_2^{\mathbb{R}}
\cap \left[-\frac{\Lambda\pi}{h},\frac{\Lambda\pi}{h}\right]^4.
\end{equation*}
By \eqref{convex}, $b(\mathbf{k}) \leq \frac{1}{\cos(\Lambda \pi)}$ and is smooth. Rescaling $\boldsymbol{\theta} = h \mathbf{k} \in \mathcal{V}_2^{\mathbb{R}}\cap[-\Lambda\pi,\Lambda\pi]^4$, there exists
$B \in C_c^\infty(\mathbb{R}^4;\mathbb{R})$, depending only on $\Lambda$, such that $b(\mathbf{k}) = B(h\mathbf{k})$ for all $\mathbf{k} \in \mathcal{V}_2^{\mathbb{R}} \cap \left[-\frac{\Lambda\pi}{h},\frac{\Lambda\pi}{h}\right]^4$. Let $\check{B} = \mathcal{F}^{-1}B$.

By Fourier inversion,
\begin{equation*}
\begin{split}
B(h\mathbf{k}) = \frac{1}{(2\pi)^4}
\int_{\mathbb{R}^4}\check{B}(\boldsymbol{\eta})
e^{-ih\boldsymbol{\eta}\cdot\mathbf{k}}
d\boldsymbol{\eta},\qquad q^{\prime\prime}(\xi) = \frac{1}{2\pi}
\int_{\mathbb{R}} K(y) e^{-iy\xi} dy,
\end{split}
\end{equation*}
the orthogonality relation
\begin{equation*}
\mathbf{1}_{\{k_1 - k_{-1} + k_2 - k_{-2} = 0 \}} = \frac{1}{2\pi}\int_{\mathbb{T}}
e^{i(k_1+k_2-k_{-1}-k_{-2})\theta} d\theta,
\end{equation*}
and the momentum constraint
\begin{equation*}
k_2+t(k_1-k_{-1})+r(k_1-k_{-2}) = rk_{-1}+(1-t-r)k_2+tk_{-2},
\end{equation*}
we have
\begin{align*}
    &\Lambda_4(m_q;\widehat f_1,\widehat f_2,\widehat f_{-1},\widehat f_{-2}) \nonumber\\
&= \frac{\mu}{4(2\pi)^5} \int_{\mathbb{R}^4}\int_0^1\int_0^1 \int_{\mathbb{R}}\check{B}(\boldsymbol{\eta})K(y)\nonumber\\
&\qquad \cdot \sum_{\mathbf{k} \in E_h^4} \mathbf{1}_{\{k_1 - k_{-1} + k_2 - k_{-2} = 0 \}} e^{-ih\boldsymbol{\eta}\cdot\mathbf{k}}
e^{-iy(rk_{-1}+(1-t-r)k_2+tk_{-2})} \prod_{j=1}^2 \widehat f_j(k_j)\overline{\widehat f_{-j}(k_{-j})}
 dy dt dr d\boldsymbol{\eta}\nonumber\\
&= \frac{\mu}{4(2\pi)^6} \int_{\mathbb{R}^4} \int_0^1\int_0^1 \int_{\mathbb{R}} \int_{\mathbb{T}}\check{B}(\boldsymbol{\eta})K(y)\nonumber\\
&\qquad \cdot (\tau_{h\eta_1}f_1)(\theta) (\tau_{h\eta_2}f_2)\bigl(\theta+(t+r-1)y\bigr) \overline{(\tau_{-h\eta_{-1}}f_{-1})(\theta+ry)}\overline{(\tau_{-h\eta_{-2}}f_{-2})(\theta+ty)} d\theta dy dt dr d\boldsymbol{\eta}, 
\end{align*}
where $(\tau_{a} f)(\theta) := f(\theta - a)$.

For notational convenience, let
\begin{equation*}
\begin{split}
g_j &= \tau_{h\eta_j}f_j,\qquad
g_{-j} = \tau_{-h\eta_{-j}}f_{-j},\qquad j=1,2,\\
c_1 &= 0,\qquad c_2 = t+r-1,\qquad c_{-1} = r,\qquad c_{-2} = t.   
\end{split}
\end{equation*}
Suppose $s \leq \frac{1}{2}$. Let $J_n = [n,n+1]$ for $n \in \mathbb{Z}$. Then for $i \neq j$,
\begin{equation*}
0 < |c_i-c_j||J_n| < 2\pi, \qquad \text{ a.e. } (t,r) \in [0,1] \times [0,1].
\end{equation*}
By the Cauchy-Schwarz inequality, \Cref{lem:case-I-two-slope}, and \eqref{eq:case-I-Wiener-kernel},
\begin{equation*}
\begin{split}
\sum_{n \in \mathbb{Z}}\int_{\mathbb{T} \times J_n} |K(y)| \prod_{j \in \Sigma_2}|g_j(\theta+c_j y)| d\theta dy &\leq \sum_{n \in \mathbb{Z} } \sup_{y \in J_n}|K(y)|
\prod_{j=1}^2 \left( \int_{\mathbb{T}\times J_n}
|g_j(\theta+c_jy)|^2 |g_{-j}(\theta+c_{-j}y)|^2 d\theta dy \right)^{1/2}\\
&\lesssim \bigl(r(1-r)\bigr)^{-\frac{1}{2}}
\prod_{j \in \Sigma_2}\| f_j \|_{L^2}.
\end{split}
\end{equation*}
Observing that $\bigl(r(1-r)\bigr)^{-\frac{1}{2}} \in L^1_{t,r} ([0,1]^2)$ and $\check{B} \in L^1(\mathbb{R}^4)$, the uniform $L^2$ estimate \eqref{eq:case-I-quartic-L2} follows.

For the mixed estimate, let $f_1=F$ and $A_j=|g_j(\theta+c_jy)|$ for $j \in \Sigma_2$. 
By the H\"{o}lder inequality applied to
\begin{equation*}
A_1A_2A_{-1}A_{-2} = (A_1^{4/3}A_2^2)^{1/4}
(A_1^{4/3}A_{-1}^2)^{1/4}
(A_1^{4/3}A_{-2}^4)^{1/4}
(A_2^2A_{-1}^2)^{1/4},
\end{equation*}
we have
\begin{equation*}
\begin{split}
&\int_{\mathbb{T} \times J_n}
\prod_{j \in \Sigma_2}|g_j(\theta+c_jy)| d\theta dy\\
&\leq
\left(
\int_{\mathbb{T} \times J_n} A_1^{\frac{4}{3}}A_2^2 d\theta dy
\right)^{1/4}
\left(
\int_{\mathbb{T}\times J_n} A_1^{\frac{4}{3}}A_{-1}^2 d\theta dy
\right)^{1/4} 
\left(
\int_{\mathbb{T}\times J_n} A_1^{\frac{4}{3}}A_{-2}^4 d\theta dy
\right)^{1/4}
\left(
\int_{\mathbb{T}\times J_n} A_2^2A_{-1}^2 d\theta dy
\right)^{1/4}\\
&\lesssim \left(rt(1-t)|1-t-r|\right)^{-\frac{1}{4}}
\| F \|_{L^{\frac{4}{3}}}
\| f_2 \|_{L^2}
\| f_{-1} \|_{L^2}
\| f_{-2} \|_{L^4},
\end{split}
\end{equation*}
where the last inequality is by \Cref{lem:case-I-two-slope}. Then, \eqref{eq:case-I-quartic-mixed} follows as in the uniform $L^2$ estimate. Permuting the slots in the same factorization proves the claim in generality.

To show $\eqref{eq:case-I-quartic-general}$, first apply the H\"older inequality
\begin{equation*}
\int_{\mathbb{T}}
\prod_{j \in \Sigma_2}|g_j(\theta+c_jy)| d\theta
\leq \prod_{j \in \Sigma_2} \| f_j \|_{L^{p_j}},
\end{equation*}
followed by the remaining integrals using the compactness of $[0,1]^{2}$, $K \in L^1(\mathbb{R})$, and
$\check{B} \in L^1(\mathbb{R}^4)$.
\end{proof}

\begin{corollary}
\label{cor:case-I-sextic-coercivity}
Assume the hypotheses in \Cref{prop:case-I-quartic}. For every $f:\mathbb{T}\rightarrow\mathbb{C}$ with $\supp\widehat{f}\subseteq E_h$, we have
\begin{align}
|\Lambda_{6}(S_{4}m_q;\widehat{f})|
&\lesssim \|f\|_{L^2}^{2}\|f\|_{L^4}^{4},
&&0<s\leq\frac{1}{2},
\label{eq:case-I-sextic-lower}\\
|\Lambda_{6}(S_{4}m_q;\widehat{f})|
&\lesssim \|f\|_{L^6}^{6},
&&\frac{1}{2}<s<1.
\label{eq:case-I-sextic-upper}
\end{align}
For $0<s\leq\frac{1}{2}$, there exist $C_1, C_2>0$ depending on $s,\Lambda$ such that $\widetilde{\mathcal{E}}(\widehat{f}) := \mathcal{E}(\widehat{f})
+C_1\|f\|_{L^2}^{4}$ satisfies
\begin{equation*}
Q(\widehat{f})
\leq \widetilde{\mathcal{E}}(\widehat{f})
\leq Q(\widehat{f}) + C_2 \|f\|_{L^2}^{4}.
\end{equation*}
\end{corollary}

\begin{proof}
Let $F=P_h(|f|^2f)$. By definitions of $S_4$ and $\Lambda_{2\ell}$, we have
\begin{equation}\label{eq:case-I-sextic-insertion}
\Lambda_{6}(S_{4}m_q;\widehat{f})
=\Lambda_{4}(m_q;\widehat{F},\widehat{f},
                  \widehat{f},\widehat{f})
+\Lambda_{4}(m_q;\widehat{f},\widehat{F},
                  \widehat{f},\widehat{f})
-\Lambda_{4}(m_q;\widehat{f},\widehat{f},
                   \widehat{F},\widehat{f})
-\Lambda_{4}(m_q;\widehat{f},\widehat{f},
                   \widehat{f},\widehat{F}).
\end{equation}
For $0<s\leq\frac{1}{2}$, apply
\eqref{eq:case-I-quartic-mixed} to each term in
\eqref{eq:case-I-sextic-insertion}, assigning the exponent $\frac{4}{3}$
to $F$ and the exponents $2,2,4$ to the remaining terms. For $\frac{1}{2} < s < 1$, apply
\eqref{eq:case-I-quartic-general}, assigning the exponent $2$
to $F$ and the exponent $6$ to the rest. Our claim follows from the uniform boundedness of $P_h$ on $L^p(\mathbb{T})$ for $p \in (1,\infty)$, and
\begin{equation*}
\| F \|_{L^{\frac{4}{3}}} \lesssim \|f\|_{L^4}^{3},\qquad \|F\|_{L^2}\leq\|f\|_{L^6}^{3}.
\end{equation*}
The coercivity estimate follows from \eqref{eq:case-I-quartic-L2}.
\end{proof}

The magnetic potential perturbs the quartic cancellation by $O(h|a|_\infty)$ .

\begin{lemma}\label{lem:case-I-weighted-magnetic}
Let $0<s<1$ and let $m_q$ be as in \Cref{lem:case-I-divisor}. For every $f:\mathbb{T}\rightarrow\mathbb{C}$ with $\supp\widehat{f}\subseteq E_h$, we have
\begin{equation}\label{eq:case-I-weighted-magnetic}
 \left|\Lambda_4\bigl(m_q D_2( \omega_{h,a(t)} -\omega_{h,0});\widehat{f}\bigr)\right|
 \lesssim_{s,\Lambda} h|a|_{\infty} \|\langle \nabla \rangle^s f\|_{L^4}^2 \|f\|_{L^4}^2,\qquad \text{ a.e. } t \in \mathbb{R}.
\end{equation}
\end{lemma}

\begin{proof}
Let $a = a(t)$ and $\theta_0 := \frac{h}{2}(k_1+k_2)$. Since $\omega_{h,a}(k) = \omega_{h,0}(k-a)$, \eqref{eq:case-I-factorization} and the homological equation \eqref{eq:abstract-homological-equation} yield
\begin{equation}\label{eq:case-I-magnetic-symbol-simple}
m_q D_2(\omega_{h,a} -\omega_{h,0}) = \frac{\mu}{2} D_2 q\bigl(\cos(ha)-1+\sin(ha)\tan\theta_0\bigr).
\end{equation}
Let $\psi \in C^{\infty}(\mathbb{T})$ such that $\psi(\theta) = \tan \theta$ on $[-\Lambda \pi, \Lambda \pi]$. By Fourier inversion and \eqref{eq:case-I-magnetic-symbol-simple}, we have
\begin{equation*}
\begin{split}
&\Lambda_4\bigl(m_q D_2(\omega_{h,a} -\omega_{h,0});\widehat{f}\bigr)\\
&=\frac{\mu}{2}(\cos(ha)-1)\Lambda_4(D_2q;\widehat{f}) + \frac{\mu}{2}\sin(ha)\sum_{n\in\mathbb{Z}}\widehat{\psi}(n) \Lambda_4\bigl(D_2q;
 \widehat{\tau_{-nh/2}f},\widehat{\tau_{-nh/2}f},\widehat{f},\widehat{f}\bigr).
\end{split}
\end{equation*}
To estimate the first term, it suffices to estimate $\Lambda_4$ with the multiplier $q(k_1)$, which corresponds to the differential operator $\langle \nabla \rangle^{2s}$. Observe that
\begin{equation*}
\begin{split}
|\cos (ha) - 1| \left|\Lambda_4 (q(k_1);\widehat{f})\right| \lesssim h|a|_{\infty} \left|\int_{\mathbb{T}}(\langle\nabla\rangle^{2s}f)\overline{|f|^2 f} dx \right|&= h |a|_{\infty} \left|\int_{\mathbb{T}}(\langle \nabla \rangle^s f) \overline{\langle\nabla\rangle^s(|f|^2 f)}dx\right|\\ &\lesssim_s h |a|_{\infty} \|\langle \nabla \rangle^s f\|_{L^4}^2 \|f\|_{L^4}^2,
\end{split}   
\end{equation*}
where the last inequality is by the fractional Leibniz rule \cite[Proposition 1]{benyi2025fractional}. The second estimation proceeds similarly since translations commute with $\langle \nabla \rangle^{s}$ and preserve $L^4$-norms. More precisely,
\begin{equation*}
\left|\sin(ha)
\sum_{n\in\mathbb{Z}}\widehat{\psi}(n)
\Lambda_4\bigl(D_2q;
 \widehat{\tau_{-nh/2}f},\widehat{\tau_{-nh/2}f},\widehat{f},\widehat{f}\bigr)\right|\lesssim_{s,\Lambda,\psi} h|a|_\infty
\|\langle \nabla \rangle^s f\|_{L^4}^2\|f\|_{L^4}^2,
\end{equation*}
since $\{\widehat{\psi}(n)\}$ is absolutely summable. Combining the two estimates proves
\eqref{eq:case-I-weighted-magnetic}.
\end{proof}

\subsection{Higher-order modified energies}

To solve the homological equation for $s>1$, we relax the regularity in \cite[Lemma 3.2]{Bernier} from $C^{\infty}$ to $C^2$.

\begin{lemma}\label{lem:fractional-bernier-symbol}
Let $s>1$ and recall $f_s,q$ from \eqref{bernier_weight}. Then,
$f_s\in C^2(\mathbb{T})$ and satisfies
\begin{align}
f_s(\theta) &\simeq_s |\theta|^{2s},
&& |f_s^{\prime\prime}(\theta)|\lesssim_s|\theta|^{2s-2},
\label{bernier_bound}\\
f_s(-\theta)&=f_s(\theta),
&& f_s(\pi-\theta)=1-f_s(\theta).
\label{bernier_symmetry}
\end{align}
For all $\mathbf{k}\in\mathcal{V}_2$,
\begin{equation*}
 |D_2 q(\mathbf{k})|
 \lesssim_s
 |D_2\omega_{h,0}(\mathbf{k})|
 \sum_{j\in\Sigma_2}|k_j|^{2s-2}.
\end{equation*}
\end{lemma}

\begin{proof}
That $f_s\in C^2$ follows from $|\sin(\frac{\cdot}{2})|^{2s},\ |\cos(\frac{\cdot}{2})|^{2s} \in C^2$ and $|\sin(\frac{\theta}{2})|^{2s} + |\cos(\frac{\theta}{2})|^{2s} > 0$ for all $\theta \in \mathbb{T}$. Near $\theta = 0$, direct differentiation yields \eqref{bernier_bound} locally. Since $f_s, |\theta| > 0$ on $[-\pi,\pi]\setminus\{0\}$, compactness away from zero shows \eqref{bernier_bound}. By direct computation, \eqref{bernier_symmetry} follows.

Assume $|k_j|\leq \frac{\pi}{4h}$ for all $j\in\Sigma_2$.
We apply the representation \eqref{eq:case-I-divisor-extension} to the present weight $q$, justified by $f_s \in C^2$. Then,
\begin{equation*}
    |m_q(\mathbf{k})| \lesssim \sup_{t,r \in [0,1]}\left| 
q^{\prime\prime}\big(k_2+t(k_1-k_{-1})+r(k_1-k_{-2})\big) \right| \lesssim_s \sum_{j\in\Sigma_2}|k_j|^{2s-2},
\end{equation*}
where the first inequality is by
\begin{equation*}
\left|\int_0^1 \int_0^1
\cos \big(h( k_2 + t (k_1-k_{-1}) + r (k_1-k_{-2}))\big) dt dr \right| \geq 2^{-\frac{1}{2}},
\end{equation*}
and the last inequality is by \eqref{bernier_bound}.

Alternatively, assume $\max_{j\in\Sigma_2}|k_j|>\frac{\pi}{4h}$. Let $\theta_j=hk_j$ and
\begin{equation*}
 x=\frac{\theta_1-\theta_{-1}}2,\qquad
 y=\frac{\theta_1-\theta_{-2}}2,\qquad
 z=\frac{\theta_{-1}+\theta_{-2}}2,
\end{equation*}
as real coordinates. Then,
\begin{equation*}
 \theta_1=z+x+y,\qquad
 \theta_{-1}=z-x+y,\qquad
 \theta_{-2}=z+x-y,
\end{equation*}
while the momentum constraint gives $\theta_2=z-x-y$ on $\mathbb{T}$.
By periodicity,
\begin{equation*}
\begin{split}
F(x,y,z)&:=h^{2s}D_2q(\mathbf{k})\\
&=f_s(z+x+y)+f_s(z-x-y)-f_s(z-x+y)-f_s(z+x-y),    
\end{split}
\end{equation*}
and \eqref{eq:case-I-factorization}
becomes $|D_2\omega_{h,0}(\mathbf{k})| =\frac{8}{h^2}|\cos (z) \sin (x) \sin (y)|$. 

Since $x,y\in[-\pi,\pi]$, if $\sin x=0$, then
$x\in\{0,\pm \pi\}$.
Then,
\begin{equation*}
    f_s(z+x+y)=f_s(z-x+y),\qquad
f_s(z-x-y)=f_s(z+x-y),
\end{equation*}
and hence $F(x,y,z)=0$. The case $\sin y=0$ follows similarly. If $\cos z=0$, then $z=\pm \frac{\pi}{2}$. By
\eqref{bernier_symmetry} and evenness,
\begin{equation*}
f_s(z+x+y)+f_s(z-x-y)=f_s(z-x+y)+f_s(z+x-y)=1,
\end{equation*}
and therefore $F(x,y,z)=0$.

Let $\mathcal K := \left\{(x,y,z)\in[-\pi,\pi]^3: |z+x+y|,\ |z-x+y|,\ |z+x-y|\leq\pi\right\}$. It remains to show
\begin{equation}\label{eq:fractional-phase-divisor}
|F(x,y,z)|
\lesssim_s
|\cos z\,\sin x\,\sin y|,
\qquad (x,y,z)\in\mathcal K.
\end{equation}
If none of the three factors on the right-hand side vanishes,
\eqref{eq:fractional-phase-divisor} holds locally by continuity.
If at most two factors vanish, then (at most twice) applications of the
Fundamental Theorem of Calculus (FTC), as in \cite[Lemma 3.2]{Bernier},
show that $F$ vanishes to at least the same order as the corresponding
zeros of $\sin x$, $\sin y$, or $\cos z$. Since these zeros are simple
and $f_s\in C^2(\mathbb T)$, the quotient in
\eqref{eq:fractional-phase-divisor} is locally bounded.

It remains to consider the case in which all three factors vanish,
which occurs if and only if
\begin{equation*}
\theta_1=\theta_2=\theta_{-1}=\theta_{-2}
=\pm\frac{\pi}{2}.
\end{equation*}
Since $f_s$ is smooth near $\pm\frac{\pi}{2}$, we have $F\in C^3$
near these points. Iterating the FTC argument from the proof of
\cite[Lemma 3.2]{Bernier} three times yields
\eqref{eq:fractional-phase-divisor} locally.

Patching these local estimates and using the compactness of
$\mathcal K$, \eqref{eq:fractional-phase-divisor} holds uniformly on $\mathcal K$. Consequently,
\begin{equation*}
 |D_2q(\mathbf{k})|
 \lesssim_s h^{2-2s}|D_2\omega_{h,0}(\mathbf{k})|
 \lesssim_s
 \sum_{j\in\Sigma_2}|k_j|^{2s-2}
 |D_2\omega_{h,0}(\mathbf{k})|.
\end{equation*}
\end{proof}

\begin{lemma}\label{lem:uniform-cyclic-multilinear}
Let $\ell \geq 2$ and $r \geq 0$. Then,
\begin{equation*}
 \Lambda_{2\ell}\left(
 \sum_{j\in\Sigma_\ell}\langle k_j\rangle^{2r};
 |\widehat f|\right)
 \lesssim_{\ell,r}
 \|f\|_{H_h^r}^2
 \|f\|_{L^2_h}^{\ell-1}
 \|f\|_{H_h^1}^{\ell-1}.
\end{equation*}
\end{lemma}

\begin{proof}
This follows from the argument of \cite[Lemma~3.3]{Bernier},
with Young's convolution inequality applied on $\mathbb{T}_h^*$. The same weight-distribution
argument applies to $\langle k\rangle^r$, and the resulting constants are
uniform in $h$.
\end{proof}

\begin{lemma}\label{lem:twist-defect}
For all $\mathbf{k} \in \mathcal{V}_2$,
\begin{equation*}
 \left|D_2(\omega_{h,a(t)}-\omega_{h,0})(\mathbf{k})\right|\lesssim \left|\sin\left(\frac{h a(t)}{2}\right)\right|\sum_{j\in\Sigma_2}\langle k_j\rangle^2,
 \qquad\text{ a.e. }t \in \mathbb{R},
\end{equation*}
with an implicit constant independent of $h$ and $a$.
\end{lemma}

\begin{proof}
Since $\omega_{h,a}(k)=\omega_{h,0}(k-a)$, the cosine addition formulas with \eqref{eq:case-I-factorization} yield
\begin{equation*}
\begin{split}
 \left| D_2(\omega_{h,a}-\omega_{h,0})(\mathbf k) \right|
 &=\frac{8}{h^2}\left|
 \cos\left(\frac{h(k_{1}+k_{2})}{2}-ha\right) - \cos\left(\frac{h(k_{1}+k_{2})}{2}\right)\right| \cdot \prod_{j=1}^2
 \left|\sin\left(\frac{h(k_1-k_{-j})}{2}\right)\right|\\
 &\lesssim \left|\sin\left(\frac{h a(t)}{2}\right)\right||k_1-k_{-1}|\cdot|k_1-k_{-2}| \lesssim \left|\sin\left(\frac{h a(t)}{2}\right)\right|\sum_{j\in\Sigma_2}\langle k_j\rangle^2.
\end{split}
\end{equation*}
\end{proof}

\subsection{Proof of Proposition~\ref{prop:polynomial-orbit-growth}}

\begin{proof}[Proof of \Cref{prop:polynomial-orbit-growth}\textnormal{(i)}]
If $h |a|_{\infty} t > 1$, then by mass conservation and the Shannon interpolation $U_h = \mathcal{S}_h u_h$,
\begin{equation*}
 \| U_h(t) \|_{H^s} = \|u_h(t)\|_{H_h^s} \lesssim_s h^{-s}\|u_h(0)\|_{L_h^2} \lesssim_{R,s} h^{-s} < (|a|_{\infty} t)^s,
\end{equation*}
and thus it suffices to assume $h |a|_{\infty} t \leq 1$.

Let $\delta = \delta(R,\Lambda)>0$ be the local time of existence in \Cref{prop:local} where we partition $[0,t]$ into intervals $I_j=[t_j,t_{j+1}]$ of length $\delta$ with $0 \leq j \lesssim \frac{t}{\delta}$. By \Cref{prop:local},
\begin{equation}\label{eq:case-I-local-X-bounds}
 \|U_h\|_{X_h^{0,b}(I_j)}\lesssim_{R,\Lambda}1,
 \qquad
 \|U_h\|_{X_h^{s,b}(I_j)}
 \lesssim_{R,\Lambda,s}\|U_h(t_j)\|_{H^s}.
\end{equation}
Let $q,m_q$ be as in \Cref{lem:case-I-divisor}. Applying \Cref{lem:case-I-weighted-magnetic} and \Cref{prop:nonauto-L4}, we have
\begin{equation}\label{eq:case-I-magnetic-increment}
\begin{split}
\int_{I_j}\left|\Lambda_4\left( m_q D_2(\omega_{h,a(\tau)} -\omega_{h,0}); \widehat{U_h(\tau)} \right)\right|d\tau &\lesssim h|a|_\infty \|\langle\nabla\rangle^sU_h\|_{L^4(I_j\times\mathbb{T})}^2 \|U_h\|_{L^4(I_j\times\mathbb{T})}^2\\
&\lesssim h|a|_\infty \|U_h(t_j)\|_{H^s}^2.
\end{split}
\end{equation}

Suppose $0 < s \leq \frac{1}{2}$. By \eqref{eq:case-I-sextic-lower}, \Cref{prop:nonauto-L4}, and \eqref{eq:case-I-local-X-bounds},
\begin{equation}\label{sextic_lower} \int_{I_j}\left|\Lambda_6 (S_4 m_q ; \widehat{U_h(\tau)} ) \right| d\tau \lesssim \| U_h \|_{L^4(I_j \times \mathbb{T})}^4 \lesssim 1.
\end{equation}

Since $\widetilde{\mathcal{E}}$ is coercive and $\partial_t \widetilde{\mathcal{E}} = \partial_t \mathcal{E}$ by \Cref{cor:case-I-sextic-coercivity}, we integrate \eqref{eq:abstract-modified-energy-identity} over $I_j$ and estimate the remainders by \eqref{eq:case-I-magnetic-increment} and \eqref{sextic_lower}, and derive
\begin{equation}\label{eq:case-I-lower-recurrence}
 \widetilde{\mathcal{E}}\bigl(\widehat{U_h(t_{j+1})}\bigr)
 \leq (1 + Ch|a|_\infty) \widetilde{\mathcal{E}} \bigl(\widehat{U_h(t_j)}\bigr)+C,
\end{equation}
for some $C=C(s,R,\Lambda)$. Iterating \eqref{eq:case-I-lower-recurrence} on $[0,t]$ at most $1+\frac{t}{\delta}$ times gives
\begin{align}
 \widetilde{\mathcal{E}}\bigl(\widehat{U_h(t)}\bigr)
 &\leq
 (1+Ch|a|_\infty)^{1+\frac{t}{\delta}}
 \widetilde{\mathcal{E}} \bigl( \widehat{U_h(0)} \bigr)
 +C\sum_{j = 0}^{\lceil \frac{t}{\delta}\rceil}(1+Ch|a|_\infty)^{j}\nonumber\\
 &\leq  e^{C h|a|_\infty (1+\frac{t}{\delta})}
 \left(
 \widetilde{\mathcal{E}}(\widehat{U_h(0)})+C\left(1+\lceil \frac{t}{\delta}\rceil\right)
 \right).\label{energyest_lower}
\end{align}
Since $h|a|_\infty t \leq 1$ by hypothesis and
$h|a|_\infty \lesssim_\Lambda 1$ by $h \leq h_0$, the exponential factor is $O(1)$ independent of $h,a,t$. The $H^s$ bound on the initial data implies $\widetilde{\mathcal{E}}\bigl(\widehat{U_h(0)}\bigr) \lesssim_{s,R,\Lambda} 1$, and therefore
\begin{equation*}
    \| U_h(t) \|_{H^s}^2 \leq \widetilde{\mathcal{E}}\bigl(\widehat{U_h(t)}\bigr) \leq \text{RHS of }\eqref{energyest_lower} \lesssim_{s,R,\Lambda} \langle t \rangle.
\end{equation*}

Consider $\frac{1}{2} < s < 1$. Since $h|a|_\infty\tau \leq 1$ for
$0 \leq \tau \leq t$, the preceding argument implies
\begin{equation*}
    \| U_h(\tau) \|_{H^{\frac{1}{4}}} + \| U_h(\tau) \|_{H^{\frac{1}{3}}} \lesssim \langle \tau \rangle^{\frac{1}{2}}.
\end{equation*}
By \eqref{eq:case-I-quartic-general} and Sobolev embedding,
\begin{equation}\label{eq:case-I-upper-correction-growth}
\left|\Lambda_4(m_q;\widehat{U_h(\tau)})\right|\lesssim \| U_h(\tau) \|_{L^4}^4 \lesssim \| U_h(\tau) \|_{H^{\frac{1}{4}}}^4 \lesssim \langle \tau \rangle^2,
\end{equation}
and similarly by \eqref{eq:case-I-sextic-upper},
\begin{equation}\label{eq:case-I-upper-sextic-growth}
\int_0^t\left|\Lambda_6(S_4m_q;\widehat{U_h(\tau)})\right|d\tau\lesssim \int_0^t \| U_h(\tau) \|_{H^{\frac{1}{3}}}^6d\tau \lesssim \langle t \rangle^4.
\end{equation}
By the definition of $\mathcal{E}$ and
\eqref{eq:case-I-upper-correction-growth}, we have
\begin{equation}\label{energyest_upper}
 \| U_h(\tau) \|_{H^s}^2
 = Q\bigl(\widehat{U_h(\tau)}\bigr)
 \leq \left|\mathcal{E}\bigl(\widehat{U_h(\tau)}\bigr)\right|
 + C\langle \tau \rangle^2,
\end{equation}
and hence it suffices to show $\left|\mathcal{E}\bigl(\widehat{U_h(\tau)}\bigr)\right| \lesssim \langle \tau \rangle^{4}$ by integrating and iterating \eqref{eq:abstract-modified-energy-identity} over $I_j$.

By \eqref{eq:case-I-magnetic-increment}, \eqref{eq:case-I-sextic-upper}, and \eqref{energyest_upper}, we derive the recurrence
\begin{equation}\label{eq:case-I-upper-recurrence}
\left|\mathcal{E}\bigl(\widehat{U_h(t_{j+1})}\bigr)\right|\leq (1+Ch|a|_\infty) \left|\mathcal{E}\bigl(\widehat{U_h(t_j)}\bigr)\right| + C \int_{I_j}\| U_h(\tau) \|_{H^{\frac{1}{3}}}^6 d\tau + Ch|a|_\infty\langle t_j \rangle^2,
\end{equation}
for some $C=C(s,R,\Lambda)>0$. Iterating
\eqref{eq:case-I-upper-recurrence} on $[0,t]$ at most
$1+\frac{t}{\delta}$ times and using
\eqref{eq:case-I-upper-sextic-growth}, we have
\begin{equation*}
\begin{split}
\left| \mathcal{E}\bigl(\widehat{U_h(t)}\bigr) \right| & \leq (1+Ch|a|_\infty)^{1+\lceil\frac{t}{\delta}\rceil} \left(\left| \mathcal{E}\bigl(\widehat{U_h(0)}\bigr) \right| + C\int_0^t \|U_h(\tau)\|_{H^{\frac{1}{3}}}^6 d\tau \right) +Ch|a|_\infty \langle t \rangle^2 \sum_{j=0}^{\lceil \frac{t}{\delta} \rceil} (1+Ch|a|_\infty)^j \\
& \leq e^{Ch|a|_\infty(1+\lceil\frac{t}{\delta}\rceil)} \left( \left| \mathcal{E}\bigl(\widehat{U_h(0)}\bigr) \right| + C \langle t \rangle^4 + \langle t \rangle^2 \right) \lesssim \langle t \rangle^{4},
\end{split}
\end{equation*}
where the last inequality follows from $\left|\mathcal{E}\bigl(\widehat{U_h(0)}\bigr)\right|
\lesssim_{s,R,\Lambda}1$, due to the $H^s$ bound on the initial data and \eqref{eq:case-I-upper-correction-growth}, and the uniform boundedness of the exponential factor, due to $h|a|_\infty t \leq 1,\ h|a|_\infty \lesssim_\Lambda 1$.
\end{proof}

\begin{proof}[Proof of \Cref{cor:prop71}]
Let $s=1$. For
$f:\mathbb{T}_h \rightarrow \mathbb{C}$, recall 
\begin{equation*}
 H_{0,h}(f)
 = \frac{1}{2}\sum_{k \in \mathbb{T}_h^*}
 \omega_{h,0}(k)|\widehat{f}(k)|^2
 +\frac{\mu}{4}\Lambda_4(1;\widehat{f}).
\end{equation*}
In the defocusing case, $H_{0,h}$ directly controls the kinetic
energy; in the focusing case, the discrete
Gagliardo-Nirenberg inequality yields
\begin{equation*}
 \Lambda_4(1;\widehat{f})
 \lesssim
 \| f \|_{L_h^2}^4
 +\| f \|_{L_h^2}^3
 \left(
 \sum_{k \in \mathbb{T}_h^*}
 \omega_{h,0}(k)|\widehat{f}(k)|^2
 \right)^{\frac{1}{2}},
\end{equation*}
and by Young's inequality, for $\mu \in \{\pm 1\}$,
\begin{equation*}
 H_{0,h}(f)
 \geq \frac{1}{4}\sum_{k \in \mathbb{T}_h^*}
 \omega_{h,0}(k)|\widehat{f}(k)|^2
 -C\left(
 \| f \|_{L_h^2}^4+\| f \|_{L_h^2}^6
 \right).
\end{equation*}
By
$\omega_{h,0}(k) \simeq |k|^2$ and mass conservation, there exists $C_R > 0$ such that $Y(t) := 1+H_{0,h}\bigl(u_h(t)\bigr)+C_R$ satisfies
\begin{equation}\label{eq:frozen-H1-coercivity}
 Y(t) \geq 1,
 \qquad
 \| u_h(t) \|_{H_h^1}^2 \lesssim_R Y(t),
 \qquad
 Y(0) \lesssim_R 1.
\end{equation}

It remains to show $Y(t) \lesssim 1$. By \Cref{lem:finite-cyclic-product-rule},
\begin{equation*}
 i\partial_t H_{0,h}\bigl(u_h(t)\bigr)
 = \frac{\mu}{4}\Lambda_4\left(
 D_2(\omega_{h,a(t)}-\omega_{h,0});
 \widehat{u_h(t)}
 \right),
 \qquad \text{ a.e. }t.
\end{equation*}
Applying \Cref{lem:twist-defect} and
\Cref{lem:uniform-cyclic-multilinear} with $\ell=2$, $r=1$,
we obtain
\begin{equation*}
 |Y'(t)|
 \lesssim
 \left|\sin\left(\frac{ha(t)}{2}\right)\right|
 \| u_h(t) \|_{L_h^2}\| u_h(t) \|_{H_h^1}^3 \lesssim_R
 \left|\sin\left(\frac{ha(t)}{2}\right)\right|
 Y(t)^{\frac{3}{2}}.
\end{equation*}
Integrating the resulting inequality for $Y^{-\frac{1}{2}}$
yields

\begin{equation*}
 Y(t)^{-\frac{1}{2}} \geq Y(0)^{-\frac{1}{2}}
 -C\int_0^t \left|\sin\left(\frac{ha(\tau)}{2}\right)\right|d\tau \geq Y(0)^{-\frac{1}{2}}-C\beta_h(t),
\end{equation*}
for some $C=C(R)>0$.
By \eqref{eq:frozen-H1-coercivity}, there exists $\epsilon_R>0$ such that $\beta_h(t) \leq \epsilon_R$ implies
$Y(t) \lesssim Y(0) \lesssim_R 1$.

Let $s>1$. Since $\beta_h(\tau) \leq \beta_h(t)$ for
$0 \leq \tau \leq t$, we have $\| u_h(\tau) \|_{H_h^1} \lesssim_R 1$ on $[0,t]$ by the preceding argument. For the weight $q$ in \eqref{bernier_weight}, the equivalence $\| f \|_{L_h^2}^2+Q(\widehat{f})
 \simeq_s \| f \|_{H_h^s}^2$ follows from \eqref{bernier_bound}. Define
\begin{equation*}
m_q(\mathbf{k})
=
\begin{cases}
\frac{\mu}{2}\frac{D_2q(\mathbf{k})}{D_2\omega_{h,0}(\mathbf{k})},
& \mathbf{k}\in\mathcal{V}_2\setminus\mathcal{Z},\\
0,
& \mathbf{k}\in\mathcal{Z}.
\end{cases}
\end{equation*}
By \Cref{lem:fractional-bernier-symbol},
$D_2q$ vanishes on $\mathcal{Z}$, and hence $m_q$ satisfies
\eqref{eq:abstract-homological-equation}.
Then, the definition of $S_4$ yields
\begin{equation}\label{eq:full-band-divisor-bounds}
 |m_q(\mathbf{k})|
 \lesssim_s \sum_{j \in \Sigma_2}\langle k_j\rangle^{2s-2},
 \qquad
 |S_4m_q(\mathbf{k})|
 \lesssim_s \sum_{j \in \Sigma_3}\langle k_j\rangle^{2s-2}.
\end{equation}

Applying \Cref{lem:uniform-cyclic-multilinear} with $r=s-1,\ \ell=2,3$, and using the $H_h^1$ bound, we obtain
\begin{equation}\label{eq:full-band-lower-order}
 \left| \Lambda_4(m_q;\widehat{u_h(\tau)}) \right|
 +\left| \Lambda_6(S_4m_q;\widehat{u_h(\tau)}) \right|
 \lesssim_{s,R} \| u_h(\tau) \|_{H_h^{s-1}}^2.
\end{equation}
By \eqref{eq:full-band-lower-order}, mass conservation, and the definition of $\mathcal{E}$,
\begin{equation}\label{nrg_comp}
 \| u_h(\tau) \|_{H_h^s}^2
 \lesssim_{s,R}
 1+\left| \mathcal{E}\bigl(\widehat{u_h(\tau)}\bigr) \right|
 +\| u_h(\tau) \|_{H_h^{s-1}}^2,
\end{equation}
while the initial $H_h^s$ bound gives
$\left| \mathcal{E}\bigl(\widehat{u_h(0)}\bigr) \right|
\lesssim_{s,R}1$.

By \Cref{lem:twist-defect},
\eqref{eq:full-band-divisor-bounds}, and Young's inequality for products,
\begin{equation*} \left|m_q D_2(\omega_{h,a(\tau)}-\omega_{h,0})\right|
 \lesssim_s \left|\sin\left(\frac{ha(\tau)}{2}\right)\right|
 \sum_{i,j \in \Sigma_2}
 \langle k_i\rangle^{2s-2}\langle k_j\rangle^2\lesssim \left|\sin\left(\frac{ha(\tau)}{2}\right)\right|
 \sum_{j \in \Sigma_2}\langle k_j\rangle^{2s}.
\end{equation*}
Applying  \Cref{lem:uniform-cyclic-multilinear} with $\ell=2$,
$r=s$ to $m_q D_2(\omega_{h,a(\tau)}-\omega_{h,0})$, the identity \eqref{eq:abstract-modified-energy-identity} with \eqref{eq:full-band-lower-order} yields
\begin{equation}\label{eq:magnetic-energy-derivative}
 \left|
 \partial_\tau\mathcal{E}\bigl(\widehat{u_h(\tau)}\bigr)
 \right|
 \lesssim_{s,R}
 \| u_h(\tau) \|_{H_h^{s-1}}^2
 +\left|\sin\left(\frac{ha(\tau)}{2}\right)\right|
 \| u_h(\tau) \|_{H_h^s}^2,
 \qquad \text{ a.e. }\tau \in [0,t],
\end{equation}
Integrating \eqref{eq:magnetic-energy-derivative} and applying
\eqref{nrg_comp}, the recurrence
\begin{equation*}
\| u_h(t) \|_{H_h^s}^2
 \lesssim_{s,R}
 1+\| u_h(t) \|_{H_h^{s-1}}^2
 +\int_0^t \| u_h(\tau) \|_{H_h^{s-1}}^2d\tau + \int_0^t \left|\sin\left(\frac{ha(\tau)}{2}\right)\right|
 \| u_h(\tau) \|_{H_h^s}^2d\tau
\end{equation*}
follows, from which we induct on $n=\lceil s-1\rceil$.

If $n=1$, then $1<s \leq 2$, and the $H_h^1$ bound controls
$H_h^{s-1}$. If $n \geq 2$, the induction hypothesis at $s-1$
applies since $\| u_h(0) \|_{H_h^{s-1}} \leq R$. Thus,
\begin{equation*}
 \| u_h(t) \|_{H_h^s}^2
 \leq C\langle t\rangle^n
 +C\int_0^t
 \left|\sin\left(\frac{ha(\tau)}{2}\right)\right|
 \| u_h(\tau) \|_{H_h^s}^2d\tau,
\end{equation*}
for some $C=C(s,R)>0$. By the Gr\"onwall inequality,
\begin{equation*}
\| u_h(t) \|_{H_h^s}^2 \leq C\langle t\rangle^n
 e^{C \beta_h(t)} \lesssim_{s,R}\langle t\rangle^n,
\end{equation*}
thereby closing the induction.
\end{proof}

\begin{proof}[Proof of \Cref{prop:polynomial-orbit-growth}\textnormal{(ii)}]
Let $\epsilon_R>0$ be given by \Cref{cor:prop71} and $\beta_h(t)>\epsilon_R$. By the triangle inequality,
\begin{equation*}
h^{-1} \leq \beta_h(t)^{-1} |a|_{\infty} t < \epsilon_R^{-1}|a|_\infty t.
\end{equation*}
Mass conservation and $|k| \leq \frac{\pi}{h}$ on $\mathbb{T}_h^*$ imply
\begin{equation*}
 \| u_h(t) \|_{H_h^s}
 \lesssim h^{-s}\| u_h(0) \|_{L_h^2}
 \lesssim_{s,R} (|a|_\infty t)^s.
\end{equation*}
The claim follows by combining the two regimes, $\beta_h(t) > \epsilon_R$ and $\beta_h(t) \leq \epsilon_R$.
\end{proof}

\section{Proof of the main result}\label{sec:polynomial-regimes}

We prove \Cref{thm:intro-polynomial} with explicit time exponents. Recall $\gamma_{s,\rho} = \min\left\{2,\frac{s-\rho}{2}\right\}$.

\subsection{Consistency estimates}

Fix $\widehat{R}\in C_c^\infty(\mathbb{R})$ with $\widehat{R}=1$ on $|\xi|\leq1$ and $\widehat{R}=0$ on $|\xi|\geq2$, and define
\begin{equation*}
 \widehat{\mathcal{R}_N f}(k)=\widehat{R}\left(\frac{k}{N}\right)\widehat{f}(k), \qquad N\geq1,
\end{equation*}
and
\begin{equation*}
 U_{h,N}=\mathcal{R}_N U_h,\qquad e_{h,N}=U_{h,N}-u.
\end{equation*}
Assume the support condition
\begin{equation}\label{eq:support-condition}
2N\leq
\begin{cases}
 \frac{\Lambda\pi}{h}, & 0<s<1,\\
\frac{\pi}{2h}, & s\geq1.
\end{cases}
\end{equation}
We estimate
\begin{equation*}
 e_h=(U_h-U_{h,N})+e_{h,N}.
\end{equation*}
Applying $\mathcal{R}_N$ to the equation for $U_h=\mathcal{S}_h u_h$ yields 
\begin{equation}\label{eq:approx-equation}
 i\partial_t U_{h,N} = -(\partial_x-ia(t))^2 U_{h,N}
 + \mathcal N(U_{h,N}) + F_{h,N},
\end{equation}
where $F_{h,N} = F_{h,N}^{(1)} + F_{h,N}^{(2)} + F_{h,N}^{(3)}$ and
\begin{equation*}
\begin{split}
 F_{h,N}^{(1)}
 &=\mathcal{R}_N\left[-\Delta_{h,a(t)}+(\partial_x-ia(t))^2\right] U_h,\\
 F_{h,N}^{(2)} &= \mathcal{R}_N\mathcal N(U_h)-\mathcal N(U_{h,N}),\\
 F_{h,N}^{(3)} &=
 \begin{cases}
 0,&0<s<1,\\
 \mu\mathcal{R}_N\left[\mathcal{S}_h(|u_h|^2u_h)-|U_h|^2U_h\right],&s\geq1.
 \end{cases}
\end{split}
\end{equation*}

\begin{lemma}\label{lem:initialization}
Let $0 \leq \rho < s$ and $\| u_0 \|_{H^s} \leq R$. Then,
\begin{equation*}
 \|e_{h,N}(0)\|_{H^\rho}
 \lesssim R\left(
 N^{-(s-\rho)}+h^2N^{\max\{2-(s-\rho),0\}}
 +h^{s+1}N^{\rho+1}\right).
\end{equation*}
Moreover,
\begin{equation}\label{eq:tail-WV}
 \|U_h(t)-U_{h,N}(t)\|_{H^\rho}
 \lesssim N^{-(s-\rho)}\|U_h(t)\|_{H^s}.
\end{equation}
\end{lemma}

\begin{proof}
The support assumptions above imply
$U_{h,N}(0)=\mathcal{R}_N\mathcal S_h\Pi_hu_0$, and
\begin{equation*}
 e_{h,N}(0) =(\mathcal{R}_N - \Id) u_0 +\mathcal{R}_N(\mathcal{S}_h \Pi_h u_0 - u_0).
\end{equation*}
By the Plancherel Theorem, 
\begin{equation*}
 \|(\Id-\mathcal{R}_N)f\|_{H^\rho}
 \lesssim N^{-(s-\rho)} \| f \|_{H^s}.
\end{equation*}
This controls the first term and proves \eqref{eq:tail-WV} likewise by taking $f=U_h(t)$.

For the second term, separate $\ell=0$ from $\ell\neq0$ in
\eqref{eq:Jh-fourier}. The former is controlled by
$|s_h(k)-1|\lesssim h^2k^2$. For the latter, the estimates
in the proof of \Cref{lem:Jh-bounded}, in particular \eqref{lem21_est}, apply with $\langle k\rangle^\rho|k|\lesssim N^{\rho+1}$. Consequently,
\begin{equation*}
 \|\mathcal{R}_N(\mathcal{S}_h\Pi_hu_0-u_0)\|_{H^\rho}
 \lesssim \left( h^2N^{\max\{2-(s-\rho),0\}}
 +h^{s+1}N^{\rho+1} \right)\|u_0\|_{H^s}.
\end{equation*}
\end{proof}

\begin{proposition}\label{prop:approx-equation}
Under the support condition \eqref{eq:support-condition}, the following estimates hold. The constants are independent of $h,a,N,I$.
\begin{enumerate}[label=\textnormal{(\roman*)}]
\item For $0<s<1$, $h\leq h_0$, and $|I|\leq1$,
\begin{equation}\label{eq:case-I-forcing-Xminus}
 \|F_{h,N}\|_{X^{0,-b_0}(I)}
 \lesssim h^2\langle |a|_\infty\rangle^4N^{4-s}
 \|U_h\|_{L^\infty_tH^s_x(I)} +N^{-s}\|U_h\|_{X_h^{s,b}(I)}\|U_h\|_{X_h^{0,b}(I)}^2.
\end{equation}
\item For $s\geq1$ and $0\leq\rho<s$,
\begin{equation*}
\begin{split}
 \|F_{h,N}^{(1)}(t)\|_{H^\rho}
 &\lesssim h^2\langle |a|_\infty\rangle^4
 N^{\max\{4-(s-\rho),0\}}\|U_h(t)\|_{H^s},\\
 \|F_{h,N}^{(2)}(t)\|_{H^\rho}
 &\lesssim N^{-(s-\rho)}\|U_h(t)\|_{H^s}^3,\\
 \|F_{h,N}^{(3)}(t)\|_{H^\rho}
 &\lesssim h^{s-\rho}\|U_h(t)\|_{H^s}^3.
\end{split}
\end{equation*}
\end{enumerate}
\end{proposition}

\begin{proof}
We first prove \textnormal{(i)}. By Taylor expansion,
\begin{equation}\label{eq:symbol-pointwise}
 |\omega_{h,a(t)}(k)-(k-a(t))^2|
 \lesssim h^2|k-a(t)|^4
 \lesssim h^2\langle |a|_\infty\rangle^4\langle k\rangle^4.
\end{equation}
Since $\mathcal{R}_N$ is supported on $|k|\leq2N$, the embedding $L^2(I\times\mathbb{T})
\hookrightarrow X^{0,-b_0}(I)$ and the H\"older inequality with $|I|\leq1$ imply
\begin{equation*}
 \|F_{h,N}^{(1)}\|_{X^{0,-b_0}(I)}
 \lesssim \|F_{h,N}^{(1)}\|_{L^2(I\times\mathbb{T})} \lesssim h^2\langle |a|_\infty\rangle^4N^{4-s}
 \|U_h\|_{L^\infty_tH^s_x(I)}.
\end{equation*}

For the nonlinear term, decompose
\begin{equation*}
 F_{h,N}^{(2)}
 =\mathcal{R}_N\bigl(\mathcal{N}(U_h)
 -\mathcal{N}(U_{h,N})\bigr)
 +(\mathcal{R}_N-\Id)\mathcal{N}(U_{h,N}).
\end{equation*}
Since $P_hU_h=U_h$, \Cref{prop:nonauto-L4} implies
\begin{equation}\label{prop6.1_est}
\begin{split}
 \|U_h\|_{L^4(I\times\mathbb{T})}
 +\|U_{h,N}\|_{L^4(I\times\mathbb{T})}
 &\lesssim \|U_h\|_{X_h^{0,b}(I)},\\
 \|\langle\nabla\rangle^sU_h\|_{L^4(I\times\mathbb{T})}
 +\|\langle\nabla\rangle^sU_{h,N}\|_{L^4(I\times\mathbb{T})}
 &\lesssim \|U_h\|_{X_h^{s,b}(I)},\\
 \|U_h-U_{h,N}\|_{L^4(I\times\mathbb{T})}
 &\lesssim N^{-s}\|U_h\|_{X_h^{s,b}(I)}.
\end{split}
\end{equation}
The multipliers $\mathcal{R}_N$ and
$N^s(\Id-\mathcal{R}_N)\langle\nabla\rangle^{-s}$
are uniformly bounded on $L^p(\mathbb{T})$ for
$1<p<\infty$. Thus, the periodic
fractional Leibniz estimate
\cite[Proposition 1]{benyi2025fractional} and the $L^4 \times L^4 \times L^4 \rightarrow L^{\frac{4}{3}}$ H\"{o}lder inequality yield
\begin{equation}\label{prop6.1_est2}
\begin{split}
 \|F_{h,N}^{(2)}\|_{L^{\frac{4}{3}}(I\times\mathbb{T})}
 \lesssim{}&
 \|U_h-U_{h,N}\|_{L^4(I\times\mathbb{T})}
 \left(
 \|U_h\|_{L^4(I\times\mathbb{T})}^2
 +\|U_{h,N}\|_{L^4(I\times\mathbb{T})}^2
 \right)\\
 &+N^{-s}
 \|\langle\nabla\rangle^sU_{h,N}\|_{L^4(I\times\mathbb{T})}
 \|U_{h,N}\|_{L^4(I\times\mathbb{T})}^2\\
 \lesssim{}&
 N^{-s}\|U_h\|_{X_h^{s,b}(I)}
 \|U_h\|_{X_h^{0,b}(I)}^2.
\end{split}
\end{equation}
Combining the estimate on $F_{h,N}^{(1)}$, \eqref{prop6.1_est}, and \eqref{prop6.1_est2} with the continuum estimate in \eqref{eq:L43-dual} proves
\eqref{eq:case-I-forcing-Xminus}.

We now prove \textnormal{(ii)}. The estimate on $F_{h,N}^{(1)}$ follows as before from the dispersion relations \eqref{eq:symbol-pointwise}. For $F^{(2)}_{h,N}$, we decompose
\begin{equation*}
 F_{h,N}^{(2)} =\mathcal{R}_N\bigl(\mathcal{N}(U_h) -\mathcal{N}(U_{h,N})\bigr)
 +(\mathcal{R}_N-\Id)\mathcal{N}(U_{h,N}),
\end{equation*}
and the estimation of the second term is immediate by \eqref{eq:tail-WV} and $\| \mathcal{N}(U_{h,N}) \|_{H^{s}} \lesssim \|U_h(t)\|_{H^s}^3$. Then,
\begin{equation*}
 \mathcal{N}(U_h)-\mathcal{N}(U_{h,N})
 =\mu(U_h-U_{h,N})\overline{U_h}U_h
 +\mu U_{h,N}\bigl(\overline{U_h}-\overline{U_{h,N}}\bigr)U_h
 +\mu U_{h,N}\overline{U_{h,N}}(U_h-U_{h,N}).
\end{equation*}
Since $\mathcal{R}_N$ is uniformly bounded, the Sobolev product
estimate applies for $s\geq1$, and \eqref{eq:tail-WV} controls
$U_h-U_{h,N}$ in $H^\rho$, the desired bound follows.

Let $F = |U_h|^2 U_h$. Since $\supp(\widehat F) \subseteq (-3M,3M)$, we have 
\begin{equation*}
\begin{split}
 \widehat{\mathcal{S}_h(|u_h|^2u_h)}(k)
 &=\sum_{\ell\in\mathbb{Z}}\widehat F(k+2M\ell)\\
 &=\widehat F(k-2M)+\widehat F(k)+\widehat F(k+2M),
 \qquad k\in\mathbb{T}_h^*,
\end{split}
\end{equation*}
where the support condition allows $\ell \in \{0,\pm 1\}$. Then,
\begin{equation*}
 \|\mathcal{S}_h(|u_h|^2u_h)-F\|_{H^\rho}^2=
 \sum_{k\in\mathbb{T}_h^*}
 \langle k\rangle^{2\rho}
 \left|\widehat F(k-2M)+\widehat F(k+2M)\right|^2 + \sum_{k\notin\mathbb{T}_h^*}
 \langle k\rangle^{2\rho}|\widehat F(k)|^2 \lesssim \sum_{|n|\gtrsim h^{-1}}
 \langle n\rangle^{2\rho}|\widehat F(n)|^2.
\end{equation*}
The desired bound follows from the uniform boundedness of $\mathcal{R}_N$ and the Sobolev algebra property.
\end{proof}

\begin{corollary}\label{cor:optimized-consistency}
Let $0\leq\rho<s$ and $\|u_0\|_{H^s}\leq R$. For $h > 0$ satisfying \eqref{eq:support-condition}, let $N = \lceil h^{-\frac{1}{2}} \rceil$. Then,
\begin{equation*}
\begin{split}
 \|e_{h,N}(0)\|_{H^\rho}
 &\lesssim h^{\gamma_{s,\rho}},\\
 \|U_h(t)-U_{h,N}(t)\|_{H^\rho}
 &\lesssim h^{\gamma_{s,\rho}}\|U_h(t)\|_{H^s},\\
 \|F_{h,N}(t)\|_{H^\rho}
 &\lesssim h^{\gamma_{s,\rho}}
 \langle |a|_\infty\rangle^4
 \left(1+\|U_h(t)\|_{H^s}^3\right),
\end{split}
\end{equation*}
where the third estimate holds for $s \geq 1$.
\end{corollary}

\subsection{Proof of the main theorem}

Throughout the proof, let $0<h\leq c h_0$, where $c\in(0,1]$ is chosen such that \eqref{eq:support-condition} holds for $N=\lceil h^{-1/2}\rceil$ while $h_0$ is given by \eqref{eq:phase-margin}. Assume $\mu=1$ and $\|u_0\|_{H^s}\leq R$. The focusing
extension is given in \Cref{future_research}. Define $\mathcal U(t,\tau)=D\Phi_a(t,\tau)[u(\tau)]$. By \eqref{cocycle},
\begin{equation}\label{eq:case-I-tangent-cocycle}
\mathcal U(t,\tau)\mathcal U(\tau,t_0)=\mathcal U(t,t_0), \qquad t,\tau,t_0 \in \mathbb{R},
\end{equation}
holds. For $\rho=0,\ s>0$ or $0<\rho<s,\ s>1$,
\Cref{rem:continuum-orbit-bounds}, Rellich compactness, and
\Cref{prop:magnetic-conjugation} imply
\begin{equation}\label{eq:comparison-orbit-tangent}
 \|\mathcal{U}(t,\tau)f\|_{H^\rho}
 \lesssim_{R,s,\rho}\langle t-\tau\rangle\|f\|_{H^\rho},
 \qquad f\in H^\rho,\quad t,\tau\in\mathbb{R}.
\end{equation}

\begin{lemma}\label{lem:case-I-local-defect}
Let $0<s<1$ and $\|u_0\|_{H^s}\leq R$. Let $I_j=[t_j,t_{j+1}]$ and define
\begin{equation}\label{eq:case-I-defect-definition}
 d_j(t)=e_{h,N}(t)-\mathcal{U}(t,t_j)e_{h,N}(t_j),
 \qquad t\in I_j.
\end{equation}
There exist $\delta = \delta(R,\Lambda) \in (0,1]$ and $\epsilon(R)>0$ such that if $|I_j| \leq \delta$ and
\begin{equation*}
\|e_{h,N}(t_j)\|_{L^2} +\|F_{h,N}\|_{X^{0,-b_0}(I_j)} \leq\epsilon,    
\end{equation*}
then
\begin{equation}\label{eq:case-I-local-error-X}
\begin{split}
\|e_{h,N}\|_{X^{0,b}(I_j)} &\lesssim_R
 \|e_{h,N}(t_j)\|_{L^2} +\|F_{h,N}\|_{X^{0,-b_0}(I_j)},\\
\|d_j\|_{C(I_j;L^2)} &\lesssim_R \|F_{h,N}\|_{X^{0,-b_0}(I_j)} +\sum_{m=2}^3 \left( \|e_{h,N}(t_j) \|_{L^2} + \|F_{h,N}\|_{X^{0,-b_0}(I_j)} \right)^m.    
\end{split}
\end{equation}
\end{lemma}

\begin{proof}
By mass conservation, the continuum local theory \cite{BourgainNLS1}, and \Cref{prop:local}, there exists $\delta=\delta(R,\Lambda)\in(0,1]$ such that, for every $I_j$ with $|I_j|\leq\delta$,
\begin{equation*}
 \|u\|_{X^{0,b}(I_j)}
 +\|U_h\|_{X_h^{0,b}(I_j)}
 \lesssim_{R,\Lambda} 1.
\end{equation*}
Define
\begin{equation*}
\mathcal{Q}_u(v) = 2u|v|^2+\overline{u}\,v^2+|v|^2v.
\end{equation*}
Subtracting \eqref{main_eq2} from \eqref{eq:approx-equation}
yields
\begin{equation}\label{eq:case-I-error-equation}
 i\partial_t e_{h,N} = -(\partial_x-ia(t))^2e_{h,N} + 2|u|^2e_{h,N} + u^2\overline{e_{h,N}} + \mathcal{Q}_u(e_{h,N}) + F_{h,N}.
\end{equation}
By Duhamel formula,
\begin{equation*}
\begin{split}
 e_{h,N}(t) = S_a(t,t_j)e_{h,N}(t_j) - i\int_{t_j}^{t}S_a(t,\tau) \Bigl( 2|u|^2e_{h,N}+u^2\overline{e_{h,N}} + \mathcal{Q}_u(e_{h,N})+F_{h,N} \Bigr)(\tau) d\tau.
\end{split}
\end{equation*}
Using $\|u\|_{X^{0,b}(I_j)} \lesssim R$,
\Cref{lem:linear-estimates}, and \eqref{eq:L43-dual},
we obtain
\begin{equation*}
 \|e_{h,N}\|_{X^{0,b}(I_j)} \lesssim \| e_{h,N} (t_j) \|_{L^2} + \delta^\theta \| F_{h,N} \|_{X^{0,-b_0}(I_j)} + \delta^\theta\left( \|e_{h,N}\|_{X^{0,b}(I_j)} + \|e_{h,N}\|_{X^{0,b}(I_j)}^2 + \|e_{h,N}\|_{X^{0,b}(I_j)}^3 \right),
\end{equation*}
and shrinking $\delta$, if necessary, the linear term is absorbed, and thus
\begin{equation}\label{lem62:duhamel}
 \|e_{h,N}\|_{X^{0,b}(I_j)} \lesssim_R \| e_{h,N} (t_j) \|_{L^2} + \delta^\theta \| F_{h,N} \|_{X^{0,-b_0}(I_j)} + \delta^\theta\left(\|e_{h,N}\|_{X^{0,b}(I_j)}^2 + \|e_{h,N}\|_{X^{0,b}(I_j)}^3 \right).
\end{equation}
Apply a fixed point argument in a closed ball of $X^{0,b}(I_j)$ 
\begin{equation*}
\left\{ v \in X^{0,b}(I_j): \| v\|_{X^{0,b}(I_j)}
\leq C(R) \left(\|e_{h,N}(t_j)\|_{L^2} +\|F_{h,N}\|_{X^{0,-b_0}(I_j)}\right)\right\}.
\end{equation*}
Then, the estimate \eqref{lem62:duhamel} closes for some $\epsilon>0$ that depends on $R$, and the corresponding difference estimate shows that this map is a contraction. The unique fixed point coincides with $e_{h,N}$. The first claim in \eqref{eq:case-I-local-error-X} has been shown.

We show the second claim. Observe that
$w(t):=\mathcal{U}(t,t_j)e_{h,N}(t_j)$ satisfies the linearized NLS
\begin{equation*}
 i\partial_t w
 =-(\partial_x-ia(t))^2w+2|u|^2w+u^2\overline{w},
 \qquad
 w(t_j)=e_{h,N}(t_j).
\end{equation*}
Subtracting the corresponding Duhamel formulas and using
$d_j=e_{h,N}-w$, we obtain
\begin{equation*}
 d_j(t)
 =-i\int_{t_j}^{t}S_a(t,\tau)
 \Bigl(
 2|u|^2d_j+u^2\overline{d_j}
 +\mathcal{Q}_u(e_{h,N})+F_{h,N}
 \Bigr)(\tau)\,d\tau.
\end{equation*}
Applying \Cref{lem:linear-estimates} and the same multilinear
estimates as in the first claim, we have
\begin{equation*}
 \|d_j\|_{X^{0,b}(I_j)}
 \lesssim_R\delta^\theta\Bigl(
 \|d_j\|_{X^{0,b}(I_j)}
 +\|F_{h,N}\|_{X^{0,-b_0}(I_j)}
 +\|e_{h,N}\|_{X^{0,b}(I_j)}^2
 +\|e_{h,N}\|_{X^{0,b}(I_j)}^3
 \Bigr).
\end{equation*}
By the preceding choice of $\delta$, the linear term is absorbed. The second claim follows from the first estimate in \eqref{eq:case-I-local-error-X} and the embedding $X^{0,b}(I_j)\hookrightarrow C(I_j;L^2)$.
\end{proof}

\begin{proof}[Proof of \Cref{thm:intro-polynomial}]
Let $0 \leq t \leq T$. Define
\begin{equation*}
 q_s=
 \begin{cases}
 \frac{1}{2},&0<s\leq\frac{1}{2},\\
 2,&\frac{1}{2}<s<1,
 \end{cases}
 \qquad
 r_s=\frac{1}{2}\lceil s-1\rceil,\quad s\geq1,
\end{equation*}
from \Cref{prop:polynomial-orbit-growth}.

Let $0<s<1$. Choose $\epsilon,\delta$ from
\Cref{lem:case-I-local-defect} and partition $[0,T]$ into intervals of the form $I_j=[t_j,t_{j+1}]$, with $t_0=0$, $|I_j|\leq\delta$,
and at most $\langle T\rangle$ intervals up to a constant depending on $R$. By \Cref{prop:polynomial-orbit-growth} and
\Cref{rem:continuum-orbit-bounds},
\begin{equation}\label{interpol_s_small}
 \|U_h\|_{C([0,T];H^s)}+\|e_h\|_{C([0,T];H^s)}
 \lesssim\langle T\rangle^{q_s}.
\end{equation}
Then, \Cref{prop:local} and \eqref{eq:case-I-forcing-Xminus} imply
\begin{equation*}
 \max_j\|F_{h,N}\|_{X^{0,-b_0}(I_j)}
 \lesssim h^{\frac{s}{2}}\langle T\rangle^{q_s}.
\end{equation*}
Recall $\| e_{h,N}(0)\|_{L^2} \lesssim h^{\frac{s}{2}}$ by \Cref{cor:optimized-consistency}. Denote
\begin{equation*}
 L_1^{(1)}:=C\langle T\rangle^2
 \left(\|e_{h,N}(0)\|_{L^2}
 +\max_j\|F_{h,N}\|_{X^{0,-b_0}(I_j)}\right) \lesssim h^{\frac{s}{2}}\langle T \rangle^{q_s+2},
\end{equation*}
where $C \gg 1$; let $0 < c_0 \ll \epsilon$ where $\epsilon>0$ is from \Cref{lem:case-I-local-defect}, and $C,c_0$ are determined in \eqref{eq:case-I-induction}. If $\langle T\rangle^2L_1^{(1)}>c_0$,
mass conservation implies
\begin{equation*}
 \|e_h\|_{C([0,T];L^2)}\lesssim1
 \lesssim\langle T\rangle^2L_1^{(1)}
 \lesssim h^{\frac{s}{2}}\langle T\rangle^{q_s+4},
\end{equation*}
and it remains to consider $\langle T\rangle^2L_1^{(1)}\leq c_0$.

By \eqref{eq:case-I-tangent-cocycle}, \eqref{eq:case-I-defect-definition}, and induction,
\begin{equation}\label{eq:case-I-telescoping}
 e_{h,N}(t_m)=\mathcal{U}(t_m,0)e_{h,N}(0)
 +\sum_{j<m}\mathcal{U}(t_m,t_{j+1})d_j(t_{j+1}).
\end{equation}
For all $m$, we prove $\|e_{h,N}(t_m)\|_{L^2}\leq2L_1^{(1)}$ by induction. Since $\| e_{h,N}(t_0) \|_{L^2} \leq L_1^{(1)}$, the base case holds. Under the induction
hypothesis, for $j<m$,
\begin{equation*}
 \|e_{h,N}(t_j)\|_{L^2}
 +\|F_{h,N}\|_{X^{0,-b_0}(I_j)}
 \leq3L_1^{(1)}\leq3c_0\leq\epsilon.
\end{equation*}
Applying \Cref{lem:case-I-local-defect} to
the telescoping formula \eqref{eq:case-I-telescoping}, along with the tangent bound \eqref{eq:comparison-orbit-tangent}, we obtain
\begin{equation}\label{eq:case-I-induction}
 \|e_{h,N}(t_m)\|_{L^2}
 \leq L_1^{(1)} + C_1\langle T\rangle^2
 \left((L_1^{(1)})^2+(L_1^{(1)})^3\right)
 \leq\frac{3}{2}L_1^{(1)},
\end{equation}
for some $C_1>0$, where the last inequality holds if $c_0 + c_0^2 \ll_{C,C_1} 1$. This closes the induction.

For any $t$, we may re-partition the interval $[0,T]$ so that $t$ becomes a node, and therefore, the induction conclusion holds for any $t \in [0,T]$ as follows:
\begin{equation*} \|e_{h,N}\|_{C([0,T];L^2)}\leq2L_1^{(1)}.
\end{equation*}
Adding back the tail $U_{h,N}(t) - U_h(t)$, estimating it as \Cref{cor:optimized-consistency}, and combining the two regimes of $\langle T \rangle^{2} L_1^{(1)}$, we have
\begin{equation*}
 \|e_h\|_{C([0,T];L^2)}
 \lesssim h^{\frac{s}{2}}\langle T\rangle^{q_s+4}.
\end{equation*}
Interpolating the estimate above with \eqref{interpol_s_small},
\begin{equation*}
 \|e_h(t)\|_{H^\rho} \lesssim h^{\frac{s-\rho}{2}}
 \langle t\rangle^{q_s+4(1-\frac{\rho}{s})},
 \qquad0\leq\rho<s<1.
\end{equation*}

Let $s\geq1$ and $\rho=0$. Let $\epsilon_{R}>0$ be the constant from \Cref{cor:prop71} such that its conclusion applies at the Sobolev scale $1,s$ whenever $\beta_h(T)\leq\epsilon_R$. If $\beta_h(T) > \epsilon_R$, then $\epsilon_R < \beta_h(T) \lesssim hT$ by \eqref{eq:beta-defect}, and mass conservation implies
\begin{equation*}
 \|e_h\|_{C([0,T];L^2)} \lesssim 1 \lesssim h^2 \langle T \rangle^2 \leq h^{\gamma_{s,0}}\langle T\rangle^{3r_s+2},
\end{equation*}
and henceforth assume $\beta_h(T) \leq \epsilon_{R}$. \Cref{cor:prop71} yields a crude bound:
\begin{equation}\label{eq:small-defect-Hs-L2-proof}
 \|U_h\|_{C([0,T];H^1)} \lesssim 1,
 \qquad \|U_h\|_{C([0,T];H^s)} \lesssim \langle T\rangle^{r_s}.
\end{equation}
Nonlinear variation of constants along $U_{h,N}$ and the mean value theorem imply
\begin{equation}\label{eq:nonlinear-variation}
\begin{split}
 e_{h,N}(t) &= U_{h,N}(t)-u(t)\\
 &=\Phi_a(t,0)[U_{h,N}(0)]-\Phi_a(t,0)[u_0] + \int_0^t D\Phi_a(t,\tau)[U_{h,N}(\tau)]
 \bigl(-iF_{h,N}(\tau)\bigr)d\tau\\
 &=\int_0^1D\Phi_a(t,0)
 [u_0+\vartheta e_{h,N}(0)]e_{h,N}(0) d\vartheta + \int_0^t D\Phi_a(t,\tau)[U_{h,N}(\tau)] \bigl(-iF_{h,N}(\tau)\bigr) d\tau.
\end{split}
\end{equation}
By \eqref{eq:small-defect-Hs-L2-proof} and the uniform boundedness of
$\mathcal{R}_N$ on $H^1$,
\begin{equation*}
 \|U_{h,N}(\tau)\|_{H^1}\lesssim1,
 \qquad
 \|u_0+\vartheta e_{h,N}(0)\|_{H^1}
 =\|(1-\vartheta)u_0+\vartheta U_{h,N}(0)\|_{H^1}=O(1),
\end{equation*}
uniformly for $0\leq\tau\leq T$ and $0\leq\vartheta\leq1$. Thus all base points in \eqref{eq:nonlinear-variation} lie in a fixed
$H^1$ ball, which is precompact in $L^2$. \Cref{prop:magnetic-conjugation} with $\rho=0$ implies
\begin{equation*}
\begin{split}
 \|e_{h,N}(t)\|_{L^2}
 &\lesssim
 \langle t\rangle\|e_{h,N}(0)\|_{L^2}
 +\int_0^t\langle t-\tau\rangle
       \|F_{h,N}(\tau)\|_{L^2} d\tau\\
 &\lesssim h^{\gamma_{s,0}}
 \left(\langle T\rangle
 +\langle T\rangle^{3r_s}
  \int_0^t\langle t-\tau\rangle d\tau\right)
 \lesssim h^{\gamma_{s,0}}\langle T\rangle^{3r_s+2},
\end{split}
\end{equation*}
where the initial-error and consistency estimates follow from \Cref{cor:optimized-consistency}. Finally, since $e_h=e_{h,N}+U_h-U_{h,N}$, the tail estimate in \Cref{cor:optimized-consistency} and \eqref{eq:small-defect-Hs-L2-proof} yield
\begin{equation*}
 \|U_h-U_{h,N}\|_{C([0,T];L^2)}
 \lesssim h^{\gamma_{s,0}}\langle T\rangle^{r_s},
\end{equation*}
and therefore
\begin{equation*}
 \|e_h\|_{C([0,T];L^2)}
 \lesssim h^{\gamma_{s,0}}\langle T\rangle^{3r_s+2}.
\end{equation*}
At $s=1$, \eqref{eq:fractional-growth} and the uniform continuum bound give $\|e_h(t)\|_{H^1}\lesssim\langle t\rangle$, so interpolation yields
\begin{equation*}
 \|e_h(t)\|_{H^\rho}
 \lesssim h^{\frac{1-\rho}{2}}\langle t\rangle^{2-\rho},
 \qquad0<\rho<1.
\end{equation*}
Let $s>1$ and $0<\rho<s$. Since $r_s\leq s$, a coarse Sobolev bound is
\begin{equation}\label{eq:positive-rho-high-error}
 \|U_h\|_{C([0,T];H^s)}+\|e_{h,N}\|_{C([0,T];H^s)}
 \lesssim\langle T\rangle^s.
\end{equation}

Denote
\begin{equation*}
 L_1^{(2)}=Ch^{\gamma_{s,\rho}}\langle T\rangle^{3s+2},
\end{equation*}
where $C \gg 1$ is chosen first and $0 < c_0 \ll 1$ afterward,
with the choices specified in \eqref{eq:positive-rho-low-bootstrap}.
If $\langle T\rangle^{s+2}\sqrt{L_1^{(2)}}>c_0$, then
\begin{equation*}
 \|e_h\|_{C([0,T];H^\rho)}
 \lesssim\langle T\rangle^s
 \lesssim\langle T\rangle^{3s+4}L_1^{(2)}
 \lesssim h^{\gamma_{s,\rho}}\langle T\rangle^{6(s+1)},
\end{equation*}
and it remains to assume
$\langle T\rangle^{s+2}\sqrt{L_1^{(2)}}\leq c_0$.

Since $\mathcal{U}(t,\tau)=D\Phi_a(t,\tau)[u(\tau)]$ is the
propagator for the linearization about $u$, variation of constants
applied to \eqref{eq:case-I-error-equation} yields
\begin{equation}\label{eq:positive-rho-variation}
 e_{h,N}(t)=\mathcal{U}(t,0)e_{h,N}(0)
 +\int_0^t\mathcal{U}(t,\tau)
 \Bigl(-i\bigl[F_{h,N}+\mathcal{Q}_u(e_{h,N})\bigr](\tau)\Bigr)d\tau.
\end{equation}
We argue by bootstrapping, starting from
$\|e_{h,N}(0)\|_{H^\rho}\leq L_1^{(2)}$ by
\Cref{cor:optimized-consistency} and the choice of $C$.
Suppose $\|e_{h,N}\|_{C([0,t];H^\rho)}\leq2L_1^{(2)}$.
Since $2L_1^{(2)}\leq2c_0^2\ll1$, we have
$\|e_{h,N}(\tau)\|_{H^\rho}\leq1$ throughout $[0,t]$.

Choose $\sigma=\frac{s+\rho}{2}\in(\rho,s)$. Since
$\sigma>\frac{1}{2}$, Sobolev multiplication, the uniform continuum
bound in \Cref{rem:continuum-orbit-bounds}, and interpolation yield
\begin{equation*} \|\mathcal{Q}_u(e_{h,N})\|_{H^\rho}\lesssim\|e_{h,N}\|_{H^\rho}\|e_{h,N}\|_{H^\sigma}\left(1+\|e_{h,N}\|_{H^\sigma}\right)\lesssim\langle T\rangle^s\|e_{h,N}\|_{H^\rho}^{\frac{3}{2}},
\end{equation*}
where the last step uses \eqref{eq:positive-rho-high-error},
$\|e_{h,N}\|_{H^\rho}\leq1$, and
$\|e_{h,N}\|_{H^\sigma}
\leq\|e_{h,N}\|_{H^\rho}^{\frac{1}{2}}
\|e_{h,N}\|_{H^s}^{\frac{1}{2}}$.
Applying the tangent bound \eqref{eq:comparison-orbit-tangent}
to \eqref{eq:positive-rho-variation}, together with
\Cref{cor:optimized-consistency}, the bootstrap estimate closes, as given by
\begin{equation}\label{eq:positive-rho-low-bootstrap}
 \|e_{h,N}\|_{C([0,t];H^\rho)}
 \leq L_1^{(2)}
 +C_1\langle T\rangle^{s+2}(2L_1^{(2)})^{\frac{3}{2}}
 \leq\frac{3}{2}L_1^{(2)},
\end{equation}
for some $C_1>0$, where $C\gg1$ absorbs the initial-error and
forcing contributions, and the last inequality follows from the
smallness condition if $c_0\ll_{C_1}1$.

Finally, $e_h=e_{h,N}+U_h-U_{h,N}$, the tail estimate in
\Cref{cor:optimized-consistency} yields
\begin{equation*}
 \|e_h\|_{C([0,T];H^\rho)}
 \lesssim L_1^{(2)}+h^{\gamma_{s,\rho}}\langle T\rangle^s
 \lesssim h^{\gamma_{s,\rho}}\langle T\rangle^{3s+2}.
\end{equation*}
Combining both regimes, we obtain
\begin{equation*}
 \|e_h(t)\|_{H^\rho}
 \lesssim h^{\gamma_{s,\rho}}\langle t\rangle^{6(s+1)},
 \qquad s>1,\quad0<\rho<s.
\end{equation*}
\end{proof}

\section{Obstructions and sharpness}\label{obstruction}

While the focusing nonlinearity restricts uniform polynomial control over large initial data, spatial sharpness governs the algebraic convergence rate for any time $T$.

\subsection{Focusing instability and polynomial bounds}

\begin{proposition}\label{thm:foc-failure}
Let $0\leq\rho<s$, and let $u$ and $u_h$ denote the corresponding solutions of  \eqref{main_eq2} and
\eqref{main_eq1}, respectively, with $\mu=-1$. There exists $R>0$ such that, for any
$\theta>0$ and $p\geq0$, there do not exist $C>0$ and $h_0>0$ for which
\begin{equation}\label{eq:foc-polynomial-bound}
 \| \mathcal{S}_h u_h(t)-u(t) \|_{H^\rho(\mathbb{T})}
 \leq C h^\theta \langle t\rangle^p
\end{equation}
holds for all $0 < h \leq h_0$, $\|u_0\|_{H^s}\leq R$, and $t \in \mathbb{R}$. 
\end{proposition}

Fix $A > 2^{-\frac{1}{2}}$;
all implicit constants may or may not depend on $A$.

\begin{remark}\label{future_research}
The focusing extension relies on $W_f$ being an open neighborhood
of zero in $L^2$. Mass conservation gives $L^2$-stability of zero
for both flows, so sufficiently small $\|u_0\|_{L^2}$ keeps the
continuum and Shannon-lifted discrete orbits in a fixed closed $L^2$-ball contained in $W_f$. The established Sobolev bounds and Rellich-Kondrachov compactness allow the same polynomial approximation argument without smallness of $\|u_0\|_{H^s}$.

The instability construction requires $A>2^{-1/2}$ and
$R\gtrsim A$, so it does not apply to arbitrarily small $H^s$-balls.
It also forces $\epsilon_*\leq\sqrt{\pi}$: if
$\epsilon_*>\sqrt{\pi}$, choose
$2^{-1/2}<A<\epsilon_*/\sqrt{2\pi}$. The $L^2$-ball of radius
$\epsilon_*$ then contains the modulationally unstable plane wave
of amplitude $A$ and, for all sufficiently small $h$, the
perturbations $u_0^{(h)}$ used in the instability construction.
Thus uniform polynomial error bounds cannot hold throughout
that ball. These results do not give a complete focusing-defocusing dichotomy: the intermediate regime between the small-mass extension
and the instability construction remains open.
\end{remark}

Note
\begin{equation*}
    A \exp\left(-i \int_0^t \left(\omega_{h,a(\tau)}(0) - A^2\right) d\tau\right)
\end{equation*}
is a plane-wave solution. Let $u_h(x,t)$ be an exact solution to the magnetic DNLS where 
\begin{equation}\label{gauge}
    u_h(x,t) = \exp\left(-i \int_0^t \left(\omega_{h,a(\tau)}(0) - A^2\right) d\tau\right) ( A + v_h(x,t) ).
\end{equation}
Substituting this ansatz yields
\begin{equation*}
\begin{split}
i\partial_t v_h &= \left(-\Delta_{h,a(t)} - \omega_{h,a(t)}(0) \right)v_h - A^2(v_h + \overline{v_h}) - \mathcal{R}(v_h),\\
    \mathcal{R}(v_h) &= P_\chi \left(2A|v_h|^2+Av_h^2+|v_h|^2v_h\right).    
\end{split}
\end{equation*}
Due to the complex conjugate term, re-write the equation as a system
\begin{equation}\label{full_perturb}
    \partial_t X_h(t) = \mathcal{L}_h(t) X_h(t) + F_h(t),
\end{equation}
where
\begin{equation*}
\begin{split}
X_h &= \begin{pmatrix} v_h & \overline{v_h} \end{pmatrix}^T, \qquad F_h = \begin{pmatrix} i\mathcal{R}(v_h) & -i\overline{\mathcal{R}(v_h)} \end{pmatrix}^T\\
\mathcal{L}_h(t) &= \begin{pmatrix} 
    -i\left(-\Delta_{h,a(t)} - \omega_{h,a(t)}(0) - A^2\right) & iA^2 \\ 
    -iA^2 & i\left(-\Delta_{h,-a(t)} - \omega_{h,a(t)}(0) - A^2\right) 
    \end{pmatrix}.
\end{split}
\end{equation*}

Let $\mathscr{U}_h(t,t_0)$ be the linear propagator associated with the linear system given by
\begin{equation*}
    \partial_t \mathscr{U}_h(t,t_0) = \mathcal{L}_h(t) \mathscr{U}_h(t,t_0), \qquad \mathscr{U}_h(t_0,t_0) = \Id.
\end{equation*}
Since $a \in L^\infty(\mathbb{R};\mathbb{R})$, the coefficients of $\mathcal{L}_h(t)$ are bounded and measurable, and the Carath\'{e}odory argument of \Cref{rem:caratheodory} applies. We consider $\mathscr{U}_h(t,t_0)$ on
$\mathbb{H}_h^1 := H_h^1 \times H_h^1$ endowed with $\|(f,g)\|_{\mathbb{H}_h^1}
:= \left(\|f\|_{H_h^1}^2+\|g\|_{H_h^1}^2\right)^{\frac{1}{2}}$.

\begin{lemma}
\label{lem:foc-linear-evolution-growth}
Let $A > 2^{-\frac{1}{2}}$ and $\omega_{h,0}(k) = \frac{4}{h^2}\sin^2 \left(\frac{hk}{2}\right)$. Define
\begin{equation*}
\lambda_h(k) = \sqrt{\omega_{h,0}(k)(2A^2-\omega_{h,0}(k))},\qquad \Lambda_h = \max_{k\in\mathbb{T}_h^*:\ \omega_{h,0}(k) \in (0,2A^2)}\lambda_h(k).    
\end{equation*}
There exists $C_L > 0$ depending only on $A$ such that, for any $0 < h < 1$ and $t \geq t_0$,
\begin{equation*}
\| \mathscr{U}_h(t,t_0) \|_{\mathbb{H}_h^1 \rightarrow \mathbb{H}_h^1} \leq C_L \exp\left(\left(\Lambda_h + C_L \ainf^{2} h^{2}\right)(t-t_0)\right).
\end{equation*}
\end{lemma}

\begin{proof}
For any $h \leq \frac{\pi}{2}$, $\Lambda_h$ is well-defined since $\omega_{h,0}(\pm 1) \in (0,2A^2)$. Moreover, $\Lambda_h$ satisfies the uniform bound
\begin{equation}\label{gammah}
    \frac{2}{\pi} \sqrt{2A^2 - 1} \leq \Lambda_h \leq A^2,
\end{equation}
where the upper bound follows by maximizing $\sqrt{x(2A^2-x)}$ over $x \in [0,2A^2]$, and the lower bound follows by the Jordan's inequality.

In the Fourier representation, the homogeneous equation of \eqref{full_perturb} is
\begin{equation*}
\partial_t \widehat{X}_h(k,t) = \widehat{\mathcal{L}}_h \widehat{X}_h(k,t),\qquad \widehat{X}_h(k,t) = \begin{pmatrix} \widehat{v}_h(k,t) & \overline{\widehat{v}_h(-k,t)} \end{pmatrix}^T,    
\end{equation*}
where
\begin{equation*}
\widehat{\mathcal{L}}_h(k,t) = \begin{pmatrix} 
-i\left(\eta_{h}(k,t) - A^2\right) & iA^2 \\ 
-iA^2 & i\left(\eta_{h}(-k,t) - A^2\right)
\end{pmatrix},\qquad \eta_{h}(k,t) := \omega_{h,a(t)}(k) - \omega_{h,a(t)}(0).
\end{equation*}
For a general $a(t)$, $\eta_{h}(k,t) \neq \eta_{h}(-k,t)$. To recover symmetry and antisymmetry, define
\begin{equation*}
\begin{split}
\Sigma_{h}(k,t) &= \eta_h(k,t) + \eta_h(-k,t) = 2 \omega_{h,0}(k) \cos (ha(t)),\\
D_{h}(k,t) &= \eta_h(k,t) - \eta_h(-k,t) = -\frac{4}{h^2} \sin(hk) \sin(ha(t)).  
\end{split}
\end{equation*}
Then, $\widehat{\mathcal{L}}_h$ decomposes into an even and odd part
\begin{equation*}
\widehat{\mathcal{L}}_h = B_h - \frac{i}{2} D_h \Id_2,\qquad B_h(k,t) := \begin{pmatrix} 
-i\left(\frac{\Sigma_h(k,t)}{2} - A^2\right) & iA^2 \\ 
-iA^2 & i\left(\frac{\Sigma_h(k,t)}{2} - A^2\right)
\end{pmatrix},   
\end{equation*}
where $\Id_2$ is the identity operator on $\mathbb{C}^2$. The odd part is diagonal, and thus can be gauged out via an integrating factor, and we have
\begin{equation}\label{B_h}
    \partial_t \widehat{Y}_h(k,t) = B_h(k,t) \widehat{Y}_h(k,t),\qquad \widehat{Y}_h(k,t) := \exp \left(\frac{i}{2} \int_{t_0}^{t} D_h(k,\tau) d\tau\right) \widehat{X}_h(k,t). 
\end{equation}
Decompose the matrix action by $B_h$ into
\begin{equation*}
\begin{split}
B_h(k,t) &= B_h^{0}(k) + \widetilde{B}_h(k,t)\\   & :=
\begin{pmatrix} 
-i\left(\omega_{h,0}(k) - A^2\right) & iA^2 \\ 
-iA^2 & i\left(\omega_{h,0}(k) - A^2\right)
\end{pmatrix}
+
\begin{pmatrix} 
2i\omega_{h,0}(k) \sin^2 \left( \frac{ha(t)}{2} \right)& 0 \\ 
0 & -2i\omega_{h,0}(k) \sin^2 \left( \frac{ha(t)}{2} \right)
\end{pmatrix}.
\end{split}
\end{equation*}
Since $(B_h^{0}(k))^2 = \lambda_h(k)^2 \Id_2$, the eigenvalues of $B_h^{0}$ are $\pm \lambda_h(k)$. If $\omega_{h,0}(k) \in (0,2A^2)$, then $\lambda_h(k) > 0$. The matrix exponential can be computed explicitly as
\begin{equation*}
e^{\tau B_h^0(k)}
=
\cosh \left(\lambda_h(k)\tau\right) \Id_2
+
\frac{\sinh \left(\lambda_h(k)\tau\right)}
{\lambda_h(k)}
B_h^0(k),
\end{equation*}
for $\tau \geq 0$. By the bounds on $\Lambda_h$ in \eqref{gammah} and
\begin{equation}\label{B_h_op}
    \| B_h^{0}(k)\|_{\mathbb{C}^2\rightarrow \mathbb{C}^2} = A^2 + |\omega_{h,0}(k) - A^2| \leq 2A^2,
\end{equation}
we have
\begin{equation}\label{exp_bound}
    \| e^{\tau B_h^{0}(k)}\|_{\mathbb{C}^2\rightarrow \mathbb{C}^2} \leq \left(1 + \frac{2\pi A^2}{\sqrt{2A^2 - 1}}\right) e^{\tau \Lambda_h} =: C_A e^{\tau \Lambda_h}.
\end{equation}
A similar analysis in other cases, when $\lambda_h(k) = 0$ or $\lambda_h(k) \in i \mathbb{R}$, yields the same estimate \eqref{exp_bound}. 

To obtain growth estimates in \eqref{B_h}, consider two disjoint cases. First, consider $\omega_{h,0}(k) \leq 3A^2$. From the Duhamel formula,
\begin{equation*}
    |\widehat{Y}_h(k,t)| \leq C_A e^{(t-t_0)\Lambda_h} |\widehat{Y}_h(k,t_0)| + \int_{t_0}^{t} C_A e^{(t-\tau)\Lambda_h}\| \widetilde{B}_h(k,\tau)\|_{\mathbb{C}^2\rightarrow \mathbb{C}^2} |\widehat{Y}_h(k,\tau)| d\tau. 
\end{equation*}
Since $\| \widetilde{B}_h(k,\tau)\|_{\mathbb{C}^2\rightarrow \mathbb{C}^2} = 2 \omega_{h,0}(k) \sin^2 \left(\frac{h a(\tau)}{2}\right) \leq \frac{3}{2} A^2 \ainf^2 h^2$, by the Gronwall inequality,
\begin{equation*}
    |\widehat{Y}_h(k,t)| \leq C_A \exp\left(\left(\Lambda_h + \frac{3}{2} C_A A^2 \ainf^2 h^2\right)(t-t_0)\right)|\widehat{Y}_h(k,t_0)|.
\end{equation*}

The analysis above depends crucially on the upper bound on $\omega_{h,0}$. Now consider $\omega_{h,0}(k) > 3A^2$. Since $B_h^{0}$ is not self-adjoint, its eigenvectors need not be orthogonal, and therefore the standard inner product in $\mathbb{C}^2$ is insufficient to capture our desired exponential bound. Define
\begin{equation*}
H=
\begin{pmatrix} 
1& \frac{A^2}{A^2 - \omega_{h,0}(k)} \\ 
\frac{A^2}{A^2 - \omega_{h,0}(k)} & 1
\end{pmatrix},\qquad \mathcal{E} = \langle H \widehat{Y}_h, \widehat{Y}_h \rangle,
\end{equation*}
which satisfies
\begin{equation}\label{norm_equiv}
H B_h^{0} + (B_h^{0})^{*}H = \mathbf{0},\qquad \frac{1}{2} |\widehat{Y}_h|^2 \leq \mathcal{E} \leq \frac{3}{2} |\widehat{Y}_h|^2,   
\end{equation}
since $\big|\frac{A^2}{A^2 - \omega_{h,0}(k)}\big| < \frac{1}{2}$. Then,
\begin{equation}\label{B_h1}
H \widetilde{B}_h + (\widetilde{B}_h)^* H = \begin{pmatrix} 
0& -\frac{4 i \omega_{h,0}(k) \sin^2 \left(\frac{h a(t)}{2}\right) A^2}{A^2 - \omega_{h,0}(k)} \\ 
\frac{4 i \omega_{h,0}(k) \sin^2 \left(\frac{h a(t)}{2}\right) A^2}{A^2 - \omega_{h,0}(k)} & 0,
\end{pmatrix}
\end{equation}
and the energy method yields
\begin{equation*}
    \frac{d}{dt}\mathcal{E}(t) = \langle (H \widetilde{B}_h + (\widetilde{B}_h)^* H) \widehat{Y}_h,\widehat{Y}_h\rangle \leq \frac{4 \omega_{h,0}(k) \sin^2 \left(\frac{h a(t)}{2}\right) A^2}{\omega_{h,0}(k) - A^2} |\widehat{Y}_h|^2 < 3 A^2 \ainf^2 h^2 \mathcal{E}(t).
\end{equation*}
The first inequality follows from direct computation using \eqref{B_h1}. The second inequality follows from $|\widehat{Y}_h|^2 \leq \left(1-\frac{A^2}{\omega_{h,0}(k) - A^2}\right)^{-1} \mathcal{E}$, a consequence of $1-\frac{A^2}{\omega_{h,0}(k) - A^2}$ being the smallest eigenvalue of $H$, and an estimate $\frac{\omega_{h,0}(k)}{\omega_{h,0}(k) - 2A^2} < 3$, a consequence of $\omega_{h,0}(k) > 3A^2$. Apply the Gronwall inequality to $\mathcal{E}(t)$ and use the norm equivalence \eqref{norm_equiv} to obtain
\begin{equation*}
    |\widehat{Y}_h(k,t)| \leq \sqrt{3}  \exp\left(\frac{3}{2}A^2 \ainf^2 h^2 (t-t_0)\right) |\widehat{Y}_h(k,t_0)|.
\end{equation*}
Then, for any $t \geq t_0$ and $k \in \mathbb{T}_h^{*}$,
\begin{equation*}|\widehat{Y}_h(k,t)| \leq C_L \exp\left(\left(\Lambda_h + C_L \ainf^2 h^2\right)(t-t_0)\right)|\widehat{Y}_h(k,t_0)|,\qquad C_L := 4 A^2 C_{A},
\end{equation*}
and the proof is complete by the definition of $\mathscr{U}_h(t,t_0)$ and $|\widehat{Y}_h(k,t)| = |\widehat{X}_h(k,t)|$ from \eqref{B_h}.    
\end{proof}

While \Cref{lem:foc-linear-evolution-growth} shows an upper bound on the linear propagator, both upper and lower bounds on the nonlinear dynamics are needed. A standard bootstrapping argument shows the exponential bound persists in the nonlinear regime. The lower bound follows from modulational instability and numerical aliasing.

\begin{lemma}
Assume the hypotheses of \Cref{lem:foc-linear-evolution-growth}, and fix $A > 2^{-\frac{1}{2}},\ s >0$. Fix the shorthands for constants $C_L, C_0, C_1$ as
\begin{equation*}
    C_L = 4 A^2 C_A = 4 A^2 \left(1 + \frac{2\pi A^2}{\sqrt{2A^2 - 1}}\right), \quad C_0 = \frac{\pi A}{\sqrt{2}},\quad \| f g \|_{H_h^1} \leq C_1 \| f \|_{H_h^1} \| g \|_{H_h^1},
\end{equation*}
where $C_1$ is the $h$-independent Sobolev algebra constant. Let $\epsilon_0>0$ and let $h>0$ be sufficiently small, with
\begin{align}
    \epsilon_0 &< \min \left\{\frac{1}{2 C_L C_0}, \frac{\sqrt{2A^2 - 1}}{16 \pi C_L^2 C_0 C_1}, \frac{A^2 \sqrt{2A^2-1}}
{24\sqrt{2}\pi C_L^3 C_0^2 C_1}\right\}, \label{epsilon_small_MI}\\
h &< \min\left\{
1, \left(\frac{\epsilon_0}{A}\right)^{\frac{1}{s+1}}, \frac{e}
{8\sqrt{2} (s+1) C_L^3 C_0 \ainf^2}
\right\} \label{h_small_MI}.
\end{align}
Let $m_h \in \mathbb{T}_h^{*} \cap \mathbb{N}$ satisfy $\lambda_h(m_h) = \Lambda_h$ and let $k_h = 2M + m_h$. Then, there exist $u_h \in AC(\mathbb{R};H_h^1)$, a solution of \eqref{main_eq1}, and $c_0 > 0$ depending only on $A$ such that
\begin{equation}\label{lb:MI}
\begin{split}
|\widehat{u}_h(m_h,t)| + |\widehat{u}_h(-m_h,t)| &\geq c_0 e^{\Lambda_h t} |\delta_h|,\qquad t \in [0,T_h],\\
|\widehat{u}_h(m_h,T_h)| + |\widehat{u}_h(-m_h,T_h)| &\geq \frac{c_0 \epsilon_0}{e},
\end{split}
\end{equation}
where
\begin{equation*}
\delta_h = h^s s_h(k_h),\qquad T_h = \frac{1}{\Lambda_h + C_L \ainf^{2} h^{2}} \log\left(\frac{\epsilon_0}{|\delta_h|}\right).
\end{equation*}
\end{lemma}

\begin{proof}
The proof consists of two steps. First, prepare highly-oscillating initial data $u_0^{(h)}$ on $\mathbb{T}$ that projects to slowly-oscillating discrete data $\Pi_h u_0^{(h)}$ on $\mathbb{T}_h$. Then, derive the lower bound \eqref{lb:MI} using the plane-wave modulational instability that perturbs from these discrete data. Let
\begin{equation}\label{init_MI}
\begin{split}
u_0^{(h)}(x) &= A + h^s \cos\left(k_h x\right),\qquad x \in \mathbb{T},\\
\Pi_h u_{0}^{(h)}(x) &= A + \delta_h \cos\left(m_h x\right),\qquad x \in \mathbb{T}_h.
\end{split}
\end{equation}
For $0<s<1$, $\pm m_h\in E_h$ for all sufficiently small $h$,
and hence $P_\chi\Pi_hu_0^{(h)}=\Pi_hu_0^{(h)}$. A uniform bound $1 \leq | m_h| < \frac{\pi}{\sqrt{2}}A$ follows from $\omega_{h,0}(m_h) \in (0,2A^2)$, and this implies
\begin{equation}\label{deltah}
\frac{2}{3\pi^2} h^{s+1} \leq |\delta_h| \leq 2^{-\frac{3}{2}}A h^{s+1}.   
\end{equation}
Let $\gamma_h = \Lambda_h + C_L \ainf^{2} h^{2}$. Apply the Duhamel formula to \eqref{full_perturb} to obtain
\begin{equation}\label{MI:duhamel}
    X_h(t) = \mathscr{U}_h(t,0) X_h(0) + \int_0^t \mathscr{U}_h(t,\tau) F_h(\tau) d\tau,
\end{equation}
where $X_h(0) = \delta_h \cos(m_h \cdot) \begin{pmatrix} 1 & 1 \end{pmatrix}^T$, and therefore $\| X_h(0) \|_{\mathbb{H}_h^1} \leq C_0 |\delta_h|$ where we may assume $C_0 \geq 1$. The linear estimate
\begin{equation*}
    \| \mathscr{U}_h(t,0) X_h(0) \|_{\mathbb{H}_h^1} \leq C_L C_0 e^{\gamma_h t} |\delta_h|
\end{equation*}
follows by \Cref{lem:foc-linear-evolution-growth}. Since $\mathbb{H}_h^1$ is an algebra, there exists $C_1 > 0$ such that
\begin{equation*}
    \| F_h (\tau) \|_{\mathbb{H}_h^1} \leq C_1 \| X_h (\tau) \|_{\mathbb{H}_h^1}^2 (1+\| X_h (\tau) \|_{\mathbb{H}_h^1}).
\end{equation*}
To implement a standard bootstrapping argument, assume there exists a maximal $0 < T^{*} \leq T_h$ where we assume, for any $t \in [0,T^{*}]$,
\begin{equation*}
\| X_h(t) \|_{\mathbb{H}_h^1} \leq 2 C_L C_0 e^{\gamma_h t} |\delta_h|.    
\end{equation*}
Then, $T^{*} = T_h$ follows from the integral term in \eqref{MI:duhamel} by taking $\epsilon_0$ sufficiently small (the second condition in \eqref{epsilon_small_MI}), since 
\begin{equation*}
\begin{split}
    \left\| \int_0^t \mathscr{U}_h(t,\tau) F_h(\tau)\right\|_{\mathbb{H}_h^1} &\leq \int_0^t C_L e^{\gamma_h(t-\tau)} \| F_h(\tau) \|_{\mathbb{H}_h^1} d\tau\\
    &\leq \int_0^t C_L C_1 e^{\gamma_h (t-\tau)} \| X_h(\tau) \|_{\mathbb{H}_h^1}^2 (1+\| X_h(\tau) \|_{\mathbb{H}_h^1}) d\tau \leq \frac{1}{2}C_L C_0 e^{\gamma_h t} |\delta_h|.
\end{split}
\end{equation*}

Apply the transformation \eqref{B_h} to the main equation \eqref{full_perturb} and consider the Fourier component at $k = m_h$
\begin{equation}\label{eq:gauged-forced-evolution}
\partial_t \widehat{Y}_h(m_h,t)
=
B_h(m_h,t)\widehat{Y}_h(m_h,t)
+
e^{
\frac{i}{2}\int_0^t D_h(m_h,t^{\prime}) dt^{\prime}
}
\widehat{F}_h(m_h,t).
\end{equation}
Recall $B_h(m_h,t) = B_h^{0}(m_h) + \widetilde{B}_h(m_h,t)$, and let
\begin{equation}\label{Q_h}
Q_h = \frac{1}{2\Lambda_h}
\left(B_h^{0}(m_h)+\Lambda_h \Id_2\right)    
\end{equation}
denote the projector onto the unstable eigenspace of $B_h^{0}(m_h)$. Apply $Q_h$ to \eqref{eq:gauged-forced-evolution} and consider the Duhamel integral equation on $[0,T_h]$
\begin{equation*}
\begin{split}
&Q_h\widehat{Y}_h(m_h,t)
=
e^{\Lambda_h t}Q_h\widehat{Y}_h(m_h,0)
\\
&+
\int_0^t
e^{\Lambda_h(t-\tau)}
Q_h\widetilde{B}_h(m_h,\tau)
\widehat{Y}_h(m_h,\tau)d\tau
+
\int_0^t
e^{\Lambda_h(t-\tau)}
e^{
\frac{i}{2}\int_0^\tau D_h(m_h,t^{\prime}) dt^{\prime}
}
Q_h\widehat{F}_h(m_h,\tau) d\tau.
\end{split}
\end{equation*}
Recalling $\widehat{Y}_h(m_h,0) = \frac{\delta_h}{2} \begin{pmatrix} 1 & 1 \end{pmatrix}^T$, we have
\begin{equation*}
\begin{split}
\left|
Q_h
\begin{pmatrix}
1\\
1
\end{pmatrix}
\right|^2
&=
\frac{
\Lambda_h^2+\left(2A^2-\omega_{h,0}(m_h)\right)^2
}{2\Lambda_h^2}
\geq \frac{1}{2},
\end{split}
\end{equation*}
and hence,
\begin{equation}\label{lb_1}
\left|
e^{\Lambda_h t}Q_h\widehat{Y}_h(m_h,0)
\right|
\geq
\frac{e^{\Lambda_h t} |\delta_h|}{2\sqrt{2}}.
\end{equation}
By \eqref{Q_h}, \eqref{B_h_op}, and the lower bound in \eqref{gammah}, we have
\begin{equation*}
\|Q_h\|_{\mathbb{C}^2\rightarrow\mathbb{C}^2}
\leq
\frac{1}{2\Lambda_h}
\left(
\|B_h^{0}(m_h)\|_{\mathbb{C}^2\rightarrow\mathbb{C}^2}
+\Lambda_h
\right) \leq C_A,
\end{equation*}
and since $\widetilde{B}_h(m_h,\tau)$ is diagonal and $\omega_{h,0}(m_h)<2A^2$, we have
\begin{equation*}
\|\widetilde{B}_h(m_h,\tau)\|_{\mathbb{C}^2\rightarrow\mathbb{C}^2}
=
2\omega_{h,0}(m_h)
\sin^2\left(\frac{ha(\tau)}{2}\right) \leq A^2 \ainf^2 h^2.
\end{equation*}
Altogether, along with $|\widehat{Y}_h(m_h,\tau)| \leq 2 C_L C_0 e^{\gamma_h \tau} |\delta_h|$,
\begin{equation}\label{lb_2}
\left|
\int_0^t e^{\Lambda_h(t-\tau)}
Q_h\widetilde{B}_h(m_h,\tau)\widehat{Y}_h(m_h,\tau)d\tau
\right|
\leq
\frac{C_L C_0}{2}e^{\Lambda_h t}|\delta_h|
\left(e^{C_L\ainf^2h^2t}-1\right)
\leq
\frac{e^{\Lambda_h t}|\delta_h|}{8\sqrt{2}}.
\end{equation}
Indeed, by restrictions on $|\delta_h|, \epsilon_0$ in \eqref{deltah}, \eqref{epsilon_small_MI}, respectively, and $\frac{C_L}{\Lambda_h} \leq C_L^2$,
\begin{equation*}
C_L \ainf^2 h^2 t \leq
C_L^2 \ainf^2 h^2 \log\left(\frac{\epsilon_0}{|\delta_h|}\right) \leq
C_L^2 \ainf^2 h^2 (s+1) \log\left(\frac{1}{h}\right) \leq 1,
\end{equation*}
and hence $e^{C_L\ainf^2h^2t}-1 \leq 2 C_L\ainf^2h^2t$, and the bound on the Duhamel term follows from \eqref{h_small_MI}.

The first condition of \eqref{epsilon_small_MI} implies
\begin{equation*}
|\widehat{F}_h(m_h,\tau)|
\leq
8 C_L^2 C_0^2 C_1 e^{2\gamma_h\tau} |\delta_h|^2.
\end{equation*}
Recalling $\|Q_h\|_{\mathbb{C}^2\rightarrow\mathbb{C}^2} \leq C_A = \frac{C_L}{4A^2}$ and $e^{C_L \ainf^2 h^2 t} \leq 1 + 2 C_L \ainf^2 h^2 t \leq 3$, the third condition implies
\begin{equation}\label{lb_3}
\left|
\int_0^t
e^{\Lambda_h(t-\tau)}
e^{\frac{i}{2}\int_0^\tau D_h(m_h,t^{\prime})dt^{\prime}}
Q_h\widehat{F}_h(m_h,\tau)d\tau
\right| \leq \frac{3\pi C_L^3 C_0^2 C_1}
{A^2\sqrt{2A^2-1}}
\epsilon_0 e^{\Lambda_h t}|\delta_h|
\leq
\frac{e^{\Lambda_h t}|\delta_h|}{8\sqrt{2}}.
\end{equation}

Since $P_\chi u_h=u_h$ along the flow, we have $P_\chi v_h=v_h$. By the change of variable \eqref{gauge}, we have
\begin{equation*}
|\widehat{u}_h(m_h,t)| + |\widehat{u}_h(-m_h,t)| \geq |\widehat{Y}_h(m_h,t)| \gtrsim |Q_h \widehat{Y}_h(m_h,t)| \gtrsim e^{\Lambda_h t} |\delta_h|,\qquad t \in [0,T_h],   
\end{equation*}
where the lower bound follows from \eqref{lb_1}, \eqref{lb_2}, \eqref{lb_3}. To show the endpoint lower bound, use $\frac{C_L}{\Lambda_h} \leq C_L^2$ and the third condition in \eqref{h_small_MI} to obtain
\begin{equation*}
C_L \ainf^2 h^2 T_h \leq (s+1) C_L^2 \ainf^2 h^2  \log\left(\frac{1}{h}\right) \leq \frac{(s+1) C_L^2 \ainf^2 h}{e} < 1,
\end{equation*}
and hence,
\begin{equation*}
|\widehat{u}_h(m_h,T_h)|
+
|\widehat{u}_h(-m_h,T_h)|
\geq c_0 e^{\Lambda_hT_h}|\delta_h| \geq
\frac{c_0\epsilon_0}{e}
\end{equation*}
\end{proof}

\begin{proof}[Proof of \Cref{thm:foc-failure}]
Define $u(x,0) = u_{0}^{(h)}(x)$ by \eqref{init_MI}. Then, $\| u_0^{(h)} \|_{H^s} \simeq_s \langle A \rangle$. Let $R \gtrsim_{s} \langle A \rangle$ independent of $h$. Suppose there exists $h_0 > 0$ such that \eqref{eq:foc-polynomial-bound} holds for some $C>0$. Then,
\begin{equation}\label{contradiction_MI}
    \text{RHS of } \eqref{eq:foc-polynomial-bound} = C h^{\theta} \langle T_h \rangle^{p} \xrightarrow[h \rightarrow 0]{} 0.
\end{equation}
Moreover, $u_0^{(h)}$ is invariant under $x \mapsto x + \frac{2\pi}{k_h}$, and hence so is $u(\cdot,t)$ for all $t \in \mathbb{R}$ by the (spatial) translation invariance and uniqueness of \eqref{main_eq2}, and therefore $\text{supp}\left(\widehat{u}(t)\right) \subseteq k_h \mathbb{Z}$. The desired contradiction is obtained by
\begin{equation*}
    \| \mathcal{S}_h u_h(T_h) - u(T_h) \|_{H^{\rho}}  \gtrsim |\widehat{u_h}(m_h,T_h)| + |\widehat{u_h}(-m_h,T_h)| \geq \frac{c_0 \epsilon_0}{e},
\end{equation*}
using the lower bound \eqref{lb:MI} and taking $h$ sufficiently small in \eqref{contradiction_MI}.
\end{proof}

\subsection{Genericity of spatial optimality}\label{sec:baire-sharpness}

For $f\in H^s$, let $u(t;f), u_h(t;f)$ denote the solutions of \eqref{main_eq2}, \eqref{main_eq1}, respectively, and let $U_h(t;f)=\mathcal{S}_h u_h(t;f)$.

\begin{proposition}\label{thm:baire-category}
Let $0 \leq \rho < s$. There exists a dense $G_\delta$ set $\mathcal{G}\subseteq H^s$ such that
\begin{equation}\label{eq:generic-sharpness-all-s}
 \varlimsup_{h\rightarrow 0}
 h^{-(\gamma_{s,\rho} + \epsilon)}
 \sup_{|t|\leq T} \| e_h(t;f) \|_{H^\rho} = \infty,
\end{equation}
for every $f\in\mathcal{G}$, $T>0$, and $\epsilon>0$. Furthermore if $s-\rho\geq4$, then for every non-constant $f\in H^s$
and $T>0$, there exists $k_0 \in \mathbb{Z} \setminus \{0\}$, depending only on $f$, such that
\begin{equation}\label{eq:initialization-saturation-lower}
 \varliminf_{h \rightarrow 0}
 h^{-2}\sup_{|t|\leq T}\|e_h(t;f)\|_{H^\rho} \geq \frac{k_0^2\langle k_0\rangle^\rho}{24}|\widehat f(k_0)| > 0.
\end{equation}
\end{proposition}

\begin{proof}
First, assume $s-\rho\geq4$. We show \eqref{eq:initialization-saturation-lower} follows from the initialization defect. Fix $k_0 \neq 0$ such that $\widehat f(k_0)\neq0$. By \eqref{eq:Jh-fourier}, we have
\begin{equation*}
 \widehat{e_h(0;f)}(k_0)
 =(s_h(k_0)-1)\widehat f(k_0)
   +\sum_{\ell\neq0}s_h(k_0+2M\ell)\widehat f(k_0+2M\ell).
\end{equation*}
The proof of \Cref{lem:Jh-bounded} yields
\begin{equation*}
 \left|\widehat{e_h(0;f)}(k_0)\right|
 \geq |s_h(k_0)-1||\widehat f(k_0)|
       -C_s|k_0|h^{s+1}\|f\|_{H^s},
\end{equation*}
for some $C_s>0$. Then, \eqref{eq:initialization-saturation-lower} follows by the Taylor expansion of $s_h$, and
\begin{equation*}
 \varliminf_{h \rightarrow 0}
 h^{-2}\|e_h(0;f)\|_{H^\rho}
 \geq \frac{k_0^2\langle k_0\rangle^\rho}{24}|\widehat f(k_0)|.
\end{equation*}

Assume $0<s-\rho<4$. Fix $R>0$, $\|f\|_{H^s}\leq R$,
and $0<\eta\leq1$. For $k\in E_h$, let
$g_h(x)=\langle k\rangle^{-s}e^{ikx}$. Define
\begin{equation*}
 d(\tau)=u(\tau;f+\eta g_h)-u(\tau;f),
 \qquad
 d_h(\tau)=U_h(\tau;f+\eta g_h)-U_h(\tau;f).
\end{equation*}
Then, \eqref{eq:Jh-fourier} yields
$d(0)=\eta g_h,\ d_h(0)=\eta s_h(k)g_h$.
Define the remainders by
\begin{equation}\label{eq:increment-splitting}
 r(\tau)=d(\tau)-\eta S_a(\tau,0)g_h,
 \qquad
 r_h(\tau)=d_h(\tau)-\eta s_h(k)S_{h,a}(\tau,0)g_h.
\end{equation}
We claim that there exists $T\in(0,1]$ such that
\begin{equation}\label{eq:baire-remainder-bound}
 \sup_{0\leq\tau\leq t}
 \left(\|r(\tau)\|_{H^\rho}+\|r_h(\tau)\|_{H^\rho}\right)
 \lesssim t^\theta\eta\langle k\rangle^{\rho-s},
 \qquad 0<t\leq T.
\end{equation}

For $0<s<1$, further assume $h\leq h_0$ from
\eqref{eq:phase-margin}. Applying \Cref{prop:local} and the local
well-posedness theory of periodic NLS \cite{BourgainNLS1}, there
exists $\delta=\delta(R)>0$ such that the uniform bound
\begin{equation}\label{unif_bound}
 \|u(\cdot;f)\|_{X^{s,b}([0,\delta])}
 +\|U_h(\cdot;f)\|_{X_h^{s,b}([0,\delta])}
 +\|u(\cdot;f+\eta g_h)\|_{X^{s,b}([0,\delta])}
 +\|U_h(\cdot;f+\eta g_h)\|_{X_h^{s,b}([0,\delta])}
 \lesssim1
\end{equation}
holds. Let $T\leq\min\{1,\delta\}$ and $I=[0,t]$.
The Duhamel formula yields
\begin{equation*}
 r_h(\tau)
 =-i\int_0^\tau S_{h,a}(\tau,\sigma)
 \bigl\{
 \mathcal{N}_h\bigl(U_h(\sigma;f+\eta g_h)\bigr)
 -\mathcal{N}_h\bigl(U_h(\sigma;f)\bigr)
 \bigr\}d\sigma.
\end{equation*}
By \Cref{lem:linear-estimates}, \Cref{prop:cubic-estimates},
and \eqref{unif_bound}, we have
\begin{equation*}
 \|r_h\|_{X_h^{\rho,b}(I)}
 \lesssim t^\theta
 \|\mathcal{N}_h\bigl(U_h(\cdot;f+\eta g_h)\bigr)
 -\mathcal{N}_h\bigl(U_h(\cdot;f)\bigr)\|_{X_h^{\rho,-b_0}(I)}
 \lesssim t^\theta\|d_h\|_{X_h^{\rho,b}(I)},
\end{equation*}
and furthermore by \eqref{eq:increment-splitting} and
\Cref{lem:linear-estimates}, we have
\begin{equation*}
 \|d_h\|_{X_h^{\rho,b}(I)}
 \lesssim\eta\langle k\rangle^{\rho-s}
 +t^\theta\|d_h\|_{X_h^{\rho,b}(I)}.
\end{equation*}
Shrink $T\ll1$, if necessary, to obtain
\begin{equation*}
 \|d\|_{X^{\rho,b}(I)}+\|d_h\|_{X_h^{\rho,b}(I)}
 \lesssim\eta\langle k\rangle^{\rho-s},
 \qquad
 \|r\|_{X^{\rho,b}(I)}+\|r_h\|_{X_h^{\rho,b}(I)}
 \lesssim t^\theta\eta\langle k\rangle^{\rho-s},
\end{equation*}
where the corresponding bounds for $d,r$ follow from the
continuum Duhamel and cubic estimates. Then,
\eqref{eq:baire-remainder-bound} follows from
$X^{\rho,b}(I)\hookrightarrow C(I;H^\rho)$ and its discrete analogue.

For $s\geq1$, the argument is analogous, with the Sobolev product
estimate replacing the dispersive estimates; the counterpart of
\eqref{unif_bound} follows directly from the contraction mapping
argument. The Sobolev product estimate gives
\begin{equation*}
 \bigl\||z|^2z-|w|^2w\bigr\|_{H_h^\rho}
 \lesssim_{s,\rho}
 \left(\|z\|_{H_h^s}+\|w\|_{H_h^s}\right)^2
 \|z-w\|_{H_h^\rho},
\end{equation*}
uniformly in $h$, and consequently,
\begin{equation*}
 \|r_h(\tau)\|_{H^\rho}
 =\left\|
 \int_0^\tau S_{h,a}(\tau,\sigma)\mathcal{S}_h
 \bigl\{
 \mathcal{N}\bigl(u_h(\sigma;f+\eta g_h)\bigr)
 -\mathcal{N}\bigl(u_h(\sigma;f)\bigr)
 \bigr\}d\sigma
 \right\|_{H^\rho} \lesssim\int_0^\tau\|d_h(\sigma)\|_{H^\rho}d\sigma.
\end{equation*}
Combined with \eqref{eq:increment-splitting}, this yields
\begin{equation*}
 \|d_h(\tau)\|_{H^\rho}
 \lesssim \eta\langle k\rangle^{\rho-s}
 +\int_0^\tau\|d_h(\sigma)\|_{H^\rho}d\sigma.
\end{equation*}
Applying Gronwall's inequality and substituting into the bound
for $r_h$, together with the identical continuum argument, gives
\begin{equation*}
 \sup_{0\leq\tau\leq t}
 \left(\|r(\tau)\|_{H^\rho}+\|r_h(\tau)\|_{H^\rho}\right)
 \lesssim t\eta\langle k\rangle^{\rho-s}
 \leq t^\theta\eta\langle k\rangle^{\rho-s},
\end{equation*}
since $t,\theta\leq1$.

Let $T,\epsilon > 0$ and $m,N \in \mathbb{N}$. For general $s - \rho > 0$, define
\begin{equation*}
 \mathcal{A}^{\epsilon,T}_{m,N}
 =\left\{f\in H^s:\sup_{|t|\leq T}\|e_{\frac{\pi}{N^{\prime}}}(t;f)\|_{H^\rho}
 \leq m\left(\frac{\pi}{N^{\prime}}\right)^{\gamma_{s,\rho}+\epsilon},\ 
 \forall\ N^{\prime}\geq N\right\}.
\end{equation*}

The closedness of $\mathcal{A}^{\epsilon,T}_{m,N}$ follows from the continuous dependence of the discrete and continuum flows on initial data. We claim that each of these sets has empty interior. Assuming the claim, the Baire Category Theorem implies
\begin{equation*}
 \mathcal{G}
 :=\bigcap_{j,m,N\geq1}
 \left(H^s\setminus\mathcal{A}^{\frac{1}{j},\frac{1}{j}}_{m,N}\right)
\end{equation*}
is a dense $G_\delta$ subset of $H^s$. For $f\in\mathcal{G}$ and $\epsilon,T>0$, let $j\geq1$
such that $\frac{1}{j}\leq\min\{\epsilon,T\}$.
Since $f\notin\mathcal{A}^{\frac{1}{j},\frac{1}{j}}_{m,N}$ for every $m,N\geq1$,
\begin{equation*}
 h^{-(\gamma_{s,\rho}+\epsilon)}
 \sup_{|t|\leq T}\|e_h(t;f)\|_{H^\rho}
 \geq
 h^{-(\gamma_{s,\rho}+\frac{1}{j})}
 \sup_{|t|\leq\frac{1}{j}}\|e_h(t;f)\|_{H^\rho} > m
\end{equation*}
holds infinitely often as $h\rightarrow0$, proving \eqref{eq:generic-sharpness-all-s}.

It remains to prove the claim. If $s-\rho\geq4$,
\eqref{eq:initialization-saturation-lower} implies that each
$\mathcal{A}^{\epsilon,T}_{m,N}$ is contained in the closed,
nowhere-dense subspace of constant functions, and hence the claim.

Alternatively, suppose $0 < s-\rho < 4$. Assume $B_{H^s}(f_*,r_*)\subseteq\mathcal{A}^{\epsilon,T}_{m,N}$ for some $f_{*} \in H^s$ and $r_{*}>0$, along with some parameters $\epsilon,T,m,N$. Choose $R>\|f_*\|_{H^s}+r_*$. Fix $0<\eta<\min\{1,r_*\}$ and $0<t<T$
sufficiently small and independent of $h$ such that \eqref{eq:baire-remainder-bound} yields
\begin{equation*}
\|r(t)\|_{H^\rho}+\|r_h(t)\|_{H^\rho}
\leq\frac{\eta}{2}\langle k\rangle^{\rho-s}.    
\end{equation*}
Let $k_h \in \mathbb{Z}$
such that
\begin{equation*}
 \left|k_h-\left(\frac{12\pi}{t}\right)^{\frac{1}{4}}
 h^{-\frac{1}{2}}\right|\leq\frac{1}{2}.
\end{equation*}
Note that $k_h \in E_h$ for sufficiently small $h$. Apply $f = f_{*}$ and $k = k_h$ in the preceding construction to obtain
\begin{equation}\label{error_baire}
\begin{split}
e_h(t;f_*+\eta g_h)-e_h(t;f_*) &= d_h(t)-d(t)\\
 &=\eta \bigl(s_h(k_h) S_{h,a}(t,0) g_h -S_a(t,0) g_h \bigr) + r_h(t)-r(t).    
\end{split}
\end{equation}
The Taylor expansion on phases yields
\begin{equation*}
 \phi_{h,k_h}(t)-\phi_{k_h}(t)
 =-\frac{h^2}{12}\int_0^t(k_h-a(\tau))^4d\tau
 +O\left(h^4\int_0^t|k_h-a(\tau)|^6d\tau\right)
 \xrightarrow[h \rightarrow 0]{}-\pi.
\end{equation*}
The limit $s_h(k_h)\xrightarrow[h \rightarrow 0]{} 1$ and the phase difference above yield, for all sufficiently small $h$, 
\begin{equation*}
\eta \|s_h(k_h)S_{h,a}(t,0)g_h-S_a(t,0)g_h\|_{H^\rho}
= \eta\langle k_h\rangle^{\rho-s}
 \left|s_h(k_h)e^{-i(\phi_{h,k_h}(t)-\phi_{k_h}(t))}-1\right|
 \geq \eta \langle k_h\rangle^{\rho-s}.
\end{equation*}
Apply \eqref{eq:baire-remainder-bound} to the decomposition \eqref{error_baire} by taking $t \ll 1$, and observe $f_*, f_*+\eta g_h \in B_{H^s}(f_*,r_*)$ to obtain
\begin{equation*}
 \frac{\eta}{2}\langle k_h\rangle^{\rho-s} \leq \|e_h(t;f_*+\eta g_h)-e_h(t;f_*)\|_{H^\rho} \leq 2m h^{\frac{s-\rho}{2}+\epsilon},
\end{equation*}
which is a contradiction as $h \rightarrow 0$ since $k_h \simeq h^{-\frac{1}{2}}$.
\end{proof}

\section{Conclusion}

We rigorously justify the approximation of the integrable magnetic NLS by a non-integrable Hamiltonian lattice for arbitrary defocusing and small-mass focusing data, relying on continuum stability and mesh-uniform discrete Sobolev estimates. Conversely, the large-data focusing obstruction demonstrates that continuum integrability alone is insufficient for polynomial approximation. Open future directions include developing higher-order, mesh-uniform normal-form reductions to sharpen temporal bounds, and achieving polynomial approximation below the energy space without Fourier filtering.

\bibliographystyle{abbrv}
\bibliography{ref1}

\end{document}